\documentclass[11pt]{amsart}
\usepackage[margin=1in]{geometry}
\usepackage{fullpage}
\usepackage{tikz-cd,ytableau}
\usetikzlibrary{decorations.pathmorphing}
\usepackage{amsmath,amsfonts,amssymb,amsthm,epsfig,epstopdf,array,mathrsfs,dsfont,enumerate, fancyhdr, graphicx, proof, comment, multicol,tabularx,float,tikz,,xypic,hyperref,appendix,setspace,wrapfig}
\hypersetup{
	colorlinks=true, linkcolor=color1, urlcolor=black, linktoc=all, citecolor=color2 }
\let\C\relax
\usepackage[color,matrix,arrow]{xy}
\newcommand{\ol}{\overline}
\newcommand{\ul}{\underline}
\newcommand{\C}{\mathbb{C}}
\newcommand{\Z}{\mathbb{Z}}
\newcommand{\Q}{\mathbb{Q}}

\newcommand{\Xcal}{\mathcal{X}}

\newcommand{\Pbb}{\mathbb{P}}
\newcommand{\Brm}{\mathrm{B}}
\newcommand{\Filt}{\mathrm{Filt}}
\newcommand{\Grad}{\mathrm{Grad}}
\newcommand{\Map}{\mathrm{Map}}
\newcommand{\Dscr}{\mathscr{D}}
\newcommand{\Fcal}{\mathcal{F}}
\newcommand{\Ucal}{\mathcal{U}}
\newcommand{\Kcal}{\mathcal{K}}
\newcommand{\Ffrak}{\mathfrak{F}}
\newcommand{\Vcal}{\mathcal{V}}
\newcommand{\Ecal}{\mathcal{E}}
\newcommand{\Ocal}{\mathcal{O}}
\newcommand{\Bcal}{\mathcal{B}}
\newcommand{\Dcal}{\mathcal{D}}
\newcommand{\Hcal}{\mathcal{H}}

\newcommand{\Hbb}{\mathbb{H}}
\newcommand{\Hrm}{\mathrm{H}}

\newcommand{\loc}{\mathrm{loc}}
\newcommand{\Lrm}{\mathrm{L}}
\newcommand{\lfrak}{\mathfrak{l}}
\newcommand{\Hom}{\mathrm{Hom}}

\newcommand{\Ann}{\mathrm{Ann}}
\newcommand{\gr}{\mathrm{gr}}
\newcommand{\Wt}{\mathrm{wt}}
\newcommand{\Csf}{\mathsf{C}}

\newcommand{\norm}{\mathrm{norm}}
\newcommand{\cyc}{\mathrm{cyc}}

\newcommand{\dinv}{\mathrm{dinv}}
\newcommand{\cl}{\mathrm{cl}}
\newcommand{\Ind}{\mathrm{Ind}}
\newcommand{\bfrak}{\mathfrak{b}}
\newcommand{\nfrak}{\mathfrak{n}}
\newcommand{\hfrak}{\mathfrak{h}}
\newcommand{\gfrak}{\mathfrak{g}}

\newcommand{\mfrak}{\mathfrak{m}}
\newcommand{\afrak}{\mathfrak{a}}
\newcommand{\pfrak}{\mathfrak{p}}

\newcommand{\zfrak}{\mathfrak{z}}
\newcommand{\Xfrak}{\mathfrak{X}}
\newcommand{\Yfrak}{\mathfrak{Y}}

\newcommand{\Gfrak}{\mathfrak{G}}

\newcommand{\Cfrak}{\mathfrak{C}}

\newcommand{\Ifrak}{\mathfrak{I}}
\newcommand{\Rfrak}{\mathfrak{R}}
\newcommand{\Rcal}{\mathcal{R}}
\newcommand{\Nscr}{\mathscr{N}}
\newcommand{\hrm}{\mathrm{h}}
\newcommand{\ebf}{\mathbf{e}}
\newcommand{\hbf}{\mathbf{h}}

\newcommand{\Lbb}{\mathbb{L}}

\newcommand{\Nbb}{\mathbf{N}}
\newcommand{\Mbb}{\mathbf{M}}
\newcommand{\tr}{\mathrm{tr}}
\newcommand{\socle}{\mathrm{Socle}}
\newcommand{\Proj}{\mathrm{Proj}}

\newcommand{\DP}{\mathrm{DP}}
\newcommand{\Ad}{\mathrm{Ad}}

\newcommand{\Scal}{\mathcal{S}}
\newcommand{\Dbb}{\mathbb{D}}
\newcommand{\Sym}{\mathrm{Sym}}
\newcommand{\Cscr}{\mathscr{C}}

\newcommand{\Abb}{\mathbb{A}}
\newcommand{\Arm}{\mathrm{A}}
\newcommand{\pol}{\mathrm{pol}}

\newcommand{\Pcal}{\mathcal{P}}

\newcommand{\Stab}{\mathrm{Stab}}
\newcommand{\Hscr}{\mathscr{H}}

\newcommand{\End}{\mathrm{End}}
\newcommand{\GL}{\mathrm{GL}}
\newcommand{\SL}{\mathrm{SL}}

\newcommand{\Lcal}{\mathcal{L}}
\newcommand{\Lbf}{\mathbf{L}}

\newcommand{\Fscr}{\mathscr{F}}
\newcommand{\Gscr}{\mathscr{G}}

\newcommand{\gl}{\mathfrak{gl}}
\newcommand{\sln}{\mathfrak{sl}_n}
\newcommand{\Acal}{\mathcal{A}}
\newcommand{\wt}{\widetilde}
\newcommand{\e}{\mathrm{e}}

\newcommand{\ind}{\mathrm{ind}}
\newcommand{\HM}{\mathrm{HM}}
\newcommand{\Mod}{\mathrm{-mod}}

\newcommand{\Supp}{\mathrm{Supp}}
\newcommand{\Ker}{\mathrm{Ker}}

\newcommand{\nil}{\mathrm{nil}}
\newcommand{\SYT}{\mathrm{SYT}}
\newcommand{\ASYT}{\mathrm{ASYT}}

\newcommand{\Hilb}{\mathrm{Hilb}}

\newcommand{\IHilb}{\mathrm{IHilb}}
\newcommand{\sign}{\mathrm{sgn}}

\newcommand{\vol}{\mathrm{vol}}
\newcommand{\red}{\mathrm{red}}

\newcommand{\ad}{\mathrm{ad }}
\newcommand{\Coh}{\mathrm{Coh}}

\newcommand{\Ass}{\mathrm{Ass}}

\newcommand{\ch}{\mathrm{ch}}

\newcommand{\triv}{\mathrm{triv}}

\renewcommand{\DH}{\mathrm{DH}}
\newcommand{\Ascr}{\mathscr{A}}
\newcommand{\Gbb}{\mathbb{G}}
\newcommand{\Ical}{\mathcal{I}}
\newcommand{\Zcal}{\mathcal{Z}}
\newcommand{\ufrak}{\mathfrak{u}}

\newtheorem{maintheorem}{Theorem}

\theoremstyle{plain}
\newtheorem{thm}{Theorem}[section]
\newtheorem{prop}[thm]{Proposition}
\newtheorem{coro}[thm]{Corollary}
\newtheorem{lem}[thm]{Lemma}
\newtheorem{conj}[thm]{Conjecture}
\theoremstyle{definition}
\newtheorem{defn}{Definition}[section]
\newtheorem{lemconst}[defn]{Lemma-Construction}
\theoremstyle{remark}
\newtheorem{rem}{Remark}[section]

\definecolor{color2}{HTML}{318CE7}
\definecolor{color1}{HTML}{0047ab}
\usepackage{microtype}
\title{From Cherednik algebras to torus knot homology via cuspidal $D$-modules}
\author{Xinchun Ma}
\email{xinchma@gmail.com}
\begin{document}
	%\spacing{1.05}
	%   \vspace{-40pt}
	\begin{abstract}
		We prove a rational generalization of Haiman's master formula, conjectured by Gorsky and Negu\c{t}: the bigraded Frobenius character of the associated graded of the finite-dimensional representation $\Lrm_{m/n}$ of the rational Cherednik algebra at slope $m/n$ can be expressed using the generator of the same slope in the elliptic Hall algebra acting on its polynomial representation. The filtration on $\Lrm_{m/n}$ is induced by the Hodge filtration on the cuspidal character $D$-module. 
		As a corollary, we relate the triply-graded Khovanov-Rozansky homology of the torus knot to the sum of hook multiplicity spaces of $\Lrm_{m/n}$, proving a conjecture of Gorsky, Oblomkov, Rasmussen and Shende. We further provide a realization of the torus knot homology in terms of explicit coherent sheaves on Hilbert schemes of points in the plane. 
	\end{abstract}
	\maketitle
	\tableofcontents
	\section{Introduction}
	\subsection{The master formula and Hilbert schemes}
	The symmetric group $S_n$ acts on the polynomial ring $R:=\C[x_1,\dots,x_n;y_1,\dots,y_n]$ by permuting the two sets of variables simultaneously. Let $(R^{S_n}_+)$ be the ideal generated by invariant polynomials whose constant terms are $0$. The diagonal coinvariant ring  is the quotient $\DH_n=R/(R^{S_n}_+).$ 
	With respect to the $\C^*\times \C^*$ scaling action on $\DH_n$, one may consider the bigraded Frobenius character of $\DH_n$. Let $e_n$ denote the $n$-th elementary symmetric function and $\nabla$ the nabla operator \cite{science_fiction} on the ring of symmetric functions over $\C(q,t)$ with Macdonald polynomials as eigenvectors. The master formula \cite{haiman_vanishing} states that the bigraded Frobenius character of $\DH_n$ equals 
	\begin{align}
		\label{formula:master} \ch_{S_n \times \C^*\times \C^*}(\DH_n)=\nabla e_n.
	\end{align}
	Upon specialization, one deduces that $\dim(\DH_n)=(n+1)^{n-1}$ and $ \dim(\DH_n^\sign)=\frac{1}{n+1}\binom{2n}{n}$, which were open conjectures prior to the landing of the more general formula \eqref{formula:master}.
	
	Haiman's proof engineers the geometry of the Hilbert scheme of $n$ points in the plane, which we denote by $\Hilb^n$. In particular, he shows the existence of the Procesi bundle, a rank $n!$ vector bundle over the Hilbert scheme \cite{haiman_macdonald} and that the space of global sections of the Procesi bundle over the punctual Hilbert scheme coincides with $\DH_n$.
	
	\subsection{Simple modules of rational Cherednik algebras}
	The number $(n+1)^{n-1}$, resp. $\frac{1}{n+1}\binom{2n}{n}$, is a special case of a  family of integers $m^{n-1}$, resp. $\frac{1}{m+n}\binom{m+n}{n}$, for $m$ positive and coprime to $n$, which enumerates various combinatorial objects such as rational parking functions, resp. rational Dyck paths. 
	
	For $n>0$, we let $\hfrak_n$ denote the $(n-1)$-dimensional standard representation of $S_n$ and $\Dcal(\hfrak_n)$ the ring of differential operators on $\hfrak_n$. Introduced in \cite{etingof-ginzburg}, the rational Cherednik algebra $\Hrm_c(S_n,\hfrak_n)$ is a deformation of $\Dcal(\hfrak_n)\rtimes S_n$ at a complex parameter $c$. It turns out that only when $c=m/n$ for coprime $m,n$ does $\Hrm_c(S_n,\hfrak_n)$ have finite-dimensional representations. In that case, there is a unique irreducible one \cite{beg}, which we denote by $\Lrm_{m/n}$. As shown in \textit{loc. cit.}, the dimension of $\Lrm_{m/n}$, $\Lrm_{m/n}^{S_n}$, is $m^{n-1}$, resp. $\frac{1}{m+n}\binom{m+n}{n}$. 
	\subsection{The elliptic Hall algebra and a rational master formula}
	The Hall algebra of the category of coherent sheaves on an elliptic curve has a geometric action on the direct sum of $\C^*\times \C^*$-equivariant K-theory of $\Hilb^n(\C^2)$ \cite{feigin-tsy,SV_Hilbert} over all $n\geq 0$. 
	For $(m,n)$ as above, there is a distinguished element $Q_{m,n}$ in the elliptic Hall algebra. 
	On the other hand, localized at equivariant parameters, Bridgeland-King-Reid equivalence \cite{BKR} implies an isomorphism
	\begin{align}
		\label{iso: Lambda intro} \bigoplus_{n\geq 0} K^{\C^*\times\C^*}(\Hilb^n)_\loc\cong \Lambda,
	\end{align}
	where $\Lambda$ denotes the ring of symmetric functions over $\C(q,t)$. 
	Therefore, the action of $Q_{m,n}$ on $1\in \Lambda$ results in a symmetric function. 
	
	The main result of the paper is the following character formula, which proves \cite[Conjecture 5.5]{gn} of Gorsky and Negu\c{t}. Recall that there is a grading on $\Lrm_{m/n}$ induced by the Euler element $\hrm_{m/n}$ in $\Hrm_{m/n}$, leading to a decomposition $\Lrm_{m/n}=\bigoplus_{\ell\in \Z} \Lrm_{m/n}(\ell)$. Moreover, we will define a Hodge filtration $F$ on $\Lrm_{m/n}$. 
	\begin{maintheorem}\label{thm:gn}(= Theorem \ref{thm:ch Lc Q})
		For positive integer $m$ coprime to $n$, the bigraded Frobenius character of $\gr^F\Lrm_{m/n}$ with respect to the $\hrm_{m/n}$ and the filtration gradings (parametrized by $q$ and $t$ respectively) satisfies 
		\begin{align}
			\label{eqn:ch Lc}
			\ch_{S_n\times\C^*\times\C^*}( \gr^F(\Lrm_{m/n})\otimes \sign)=(Q_{m,n}\cdot 1)(z;q,q^{-1}t).
		\end{align}
	\end{maintheorem}
	When $m=n+1$, $\gr^F(\Lrm_{(n+1)/n})\otimes \sign\cong \DH_n$ (cf. \cite{Gordon_diagonal_harmonics}) and $Q_{n+1,n}\cdot 1= \nabla e_n$. Hence \eqref{eqn:ch Lc} generalizes \eqref{formula:master}. 
	Our proof focuses on computing the left hand side of \eqref{eqn:ch Lc}. The result is then matched with the right hand side studied using shuffle algebra techniques in \cite{negutshuffle}, \cite{negutwalgebra}, \cite{BHMPS} and many others. 
	\subsection{Cuspidal $D$-modules and Hodge filtrations}
	For every integer pair $(m,n)$ as above, there is a $D$-module $\Nbb_{m/n}$ on the Lie algebra $\mathfrak{sl}_n$ supported on the nilpotent cone, known as the cuspidal character $D$-module in the work of Lusztig \cite{lusztig4}. As shown by Calaque, Enriquez and Etingof \cite{cee}, the simple module $\Lrm_{n/m}$ of type $A_{m-1}$ can be realized as the multiplicity space of $(\C^n)^{\otimes m}$ in $\Nbb_{m/n}$ as a $\SL_n$-module. The $D$-module $\Nbb_{m/n}$ naturally has the structure of a Hodge module, according to Saito's theory \cite{saito}. In this way, we define a Hodge filtration on $\Lrm_{n/m}$. 
	
	We may take minimal extension of the integrable connection to the Springer resolution, which is across a simple normal crossing divisor. In particular, the Hodge filtration can be explicitly described in terms of Deligne's canonical lattice. Using Laumon's formula on associated graded of a direct image of a filtered $D$-module, we write the Hodge associated graded of $\Nbb_{m/n}$ as the pushforward of a DG module, supported on what we call a cuspidal DG commuting variety. This description allows us to compute the left hand side of \eqref{eqn:ch Lc} with $m$ and $n$ flipped.
	
	Such a computation of the Hodge associated graded of $\Nbb_{m/n}$ closely resembles that of the Harish-Chandra $D$-module in \cite{isospec}, through which the Procesi bundle is realized as the pushforward of a DG algebra. In particular, we also obtain a coherent sheaf $\Fcal_{m/n}$ on $\Hilb^n$, whose image under the isomorphism \eqref{iso: Lambda intro} is $q^{n-1}Q_{m+n,n}\cdot 1$.
	
	\subsection{Vanishing of higher cohomology}
	In contrast to Haiman's approach, vanishing of higher cohomology of $\Fcal_{m/n}$ is not necessary for proving Theorem \ref{thm:gn}. However, we do establish the following result using the ``quantization commutes with reduction" theorem of Halpern-Leistner \cite{HalpernLeistner_DerivedTheta}. Let $\Pcal^\vee$ be the dual of the Procesi bundle and for $0\leq j\leq n-1$, let $\Pcal^\vee_{\wedge^j\hfrak}=\Hom_{S_n}(\wedge^j\hfrak,\Pcal^\vee)$. 
	\begin{maintheorem}(=Theorem \ref{thm: vanishing})\label{thm B}
		For $0\leq j\leq n-1$, $H^i(\Hilb^n,\Fcal_{m/n}\otimes\Pcal^\vee_{\wedge^j\hfrak})=0$ for all $i>0$, and there is a bigraded isomoprhism
		\begin{align*}
			H^0(\Hilb^n,\Fcal_{m/n}\otimes\Pcal^\vee_{\wedge^j\hfrak})&\cong \Hom_{S_{n}}(\wedge^{j}\hfrak_{n},\gr^F\Lrm_{m/n}).
		\end{align*}
	\end{maintheorem}
	Beyond $\wedge^j\hfrak$, we conjecture the following: 
	\begin{conj}\label{conj:vanisihng}
		For $i>0$, $H^i(\Hilb^n,\Fcal_{m/n}\otimes\Pcal^\vee)=0$, and there is a bigraded $\C[\hfrak_n^2]\rtimes S_n$-module isomorphism
		\begin{align*}
			H^0(\Hilb^n,\Fcal_{m/n}\otimes\Pcal^\vee)&\cong \gr^F \Lrm_{m/n} 
		\end{align*}
	\end{conj}
	When $m\gg 0$, the vanishing of $H^{>0}(\Hilb^n, \Fcal_{m/n}\otimes\Pcal^\vee)$ is automatic. Under this assumption, we prove Conjecture \ref{conj:vanisihng} as Theorem \ref{thm:Lc F P} by showing the socle of $\gr^F  {\Lrm_{m/n}}$ as a $\C[\hfrak^2]$-module is $\gr^F \e {\Lrm_{m/n}}$. 
	\subsection{Applications of Theorem \ref{thm:gn}}
	\subsubsection{Khovanov-Rozansky homology of torus knots}
	To a braid $\gamma$ one associates its Rouquier complex of Soergel bimodules. Applying Hochschild homology termwise and then taking cohomology gives a triply graded vector space, with gradings arising from the Hochschild cohomological degree, internal degree and homological degree. After reduction and grading shifts depending on the braid, this defines the reduced Khovanov-Rozansky homology $\mathrm{HHH}(\overline\gamma)$ of the braid closure. We use the grading conventions of \cite[Section 1.5]{gors}, with the reduced unknot concentrated in degree $(0,0,0)$.
	
	For torus knot $T_{m,n}$, Gorsky, Oblomkov, Rasmussen and Shende proposed the following conjecture and proved a bigraded version without the filtration. 
	\begin{conj}\label{conj:gorsconj}
		(\cite[Conjecture 1.2]{gors}) For positive coprime integers $m,n$, set $g=(m-1)(n-1)/2$. There exists a filtration on $\Lrm_{m/n}$ compatible with the filtration on $\Hrm_{m/n}({S_n},\hfrak_n)$ and the $\hrm_{m/n}$-grading such that, for $0\leq i\leq n-1$ and $j,\ell\in\mathbb Z$, there exists an isomorphism of vector spaces
		\begin{align*}
			\mathrm{HHH}^{2g+2i,\,2j,\,\ell-j-g}(T_{m,n}) \cong \Hom_{S_n}\left( \wedge^i\hfrak,\gr^F_\ell\Lrm_{m/n}(j) \right).
		\end{align*}
	\end{conj}
	
	Elias-Hogancamp and Mellit \cite{Elias-Hogancamp, mellitknothomology} relate $\Hrm\Hrm\Hrm(T_{m,n})$ with the multiplicity of $\wedge^\bullet\hfrak$ on the right hand side of \eqref{eqn:ch Lc}. Thus, as a corollary of Theorem \ref{thm:gn}, we have: 
	\begin{maintheorem}\label{thm:gors}(= Theorem \ref{thm:HHH Lc})  Conjecture \ref{conj:gorsconj} holds with respect to the re-indexed Hodge filtration.
	\end{maintheorem}
	Moreover, Theorem \ref{thm B} provides a coherent sheaf realization of $\Hrm\Hrm\Hrm(T_{m,n})$. It would be interesting to compare our sheaf with the complex of coherent sheaves on $\Hilb^n$ constructed in \cite{oblomkov-rozansky} using matrix factorization. 
	\subsubsection{Filtrations and compactified Jacobians}
	In \cite{gors}, the filtration in Conjecture \ref{conj:gorsconj} admits several constructions. The coincidence of the inductive and algebraic filtrations has been shown in \cite{ma}. Here we show the Hodge filtration coincides with the $\afrak$-filtration (cousin of the algebraic filtration).
	As a corollary, using the shuffle theorem, we prove the main conjecture in \cite{oyjacobian} of Oblomkov and Yun as Theorem \ref{cor:oy}, relating associated graded of cohomology of the compactified Jacobian of the plane curve $y^m=x^n$ with respect to the perverse filtration and associated graded of $\Lrm_{m/n}^{S_n}$ with respect to the $\afrak$-filtration. 
	\subsection{Acknowledgements} I would like to express my deepest gratitude to my advisor, Victor Ginzburg, for suggesting this problem. This paper would not have been possible without his guidance and inspiration. I am very grateful to Andrei Negu\c{t}, especially for telling me about the reference \cite{negutwalgebra}. I thank Eugene Gorsky for valuable discussions and for giving me the opportunities to present this work while it was still in progress, Oscar Kivinen and Minh-T\^{a}m Trinh for helpful discussions at an early stage of this project, Misha Mazin for pointing out a sign error in a previous draft, Roman Bezrukavnikov for directing me towards the work of Halpern-Leistner and Ivan Losev for inspiring discussions. This work was partially supported by the the NSF grant DMS-2503504.

	ChatGPT was used for auditing numerous drafts. Besides cosmetic editing, it also disproved all of my attempts on Conjecture \ref{conj:vanisihng}, fixed normalizations in Section \ref{section: EHA} and assisted in the proof of Lemma \ref{lem:strata}. 
	\section{Cherednik algebras and mirabolic $D$-modules}\label{sec:basics}
	We work over $\C$ throughout. Fix an integer $n\geq2$. Let $\ol{G}$ denote $\GL_n$ with Lie algebra $\ol{\gfrak}:=\gl_n$ and $G$ denote $\SL_n$ with Lie algebra $\gfrak=\sln$. Let $\ol{T}$ be the maximal torus in $\ol{G}$ with Lie algebra $\ol{\hfrak}\subset \ol{\gfrak}$ consisting of diagonal matrices. Take $T=\ol{T}\cap G$ and $\hfrak=\ol{\hfrak}\cap \gfrak$.
	\subsection{Rational Cherednik algebras}\label{sec:omega}
	For any $c\in \C$, the rational Cherednik algebra $\Hrm_c:=\Hrm_c({S_n},\hfrak)$ is the $\C$-algebra generated by $\hfrak$, $\hfrak^*$ and $S_n$ subject to the following relations: 
	\begin{align*}
		&[x,x']=[y,y']=0, \quad wxw^{-1}=w(x), \quad wyw^{-1}=w(y)\\
		&[y,x]=x(y)-\sum_{s\in S}c\langle \alpha_s,y\rangle\langle \alpha_s^\vee, x\rangle s
	\end{align*}
	where $x,x'\in \hfrak^*$, $y,y'\in \hfrak$, $w\in {S_n}$ and $\langle-,-\rangle$ is the natural pairing between $\hfrak$ and $\hfrak^*$. Here $S\subset S_n$ is the set of reflections, $\alpha_s$ and $\alpha_s^\vee$ are the root and coroot associated with $s$, normalized by $\langle\alpha_s,\alpha_s^\vee\rangle=2$.
	
	For an irreducible ${S_n}$-representation $\mu$, we regard it as a $S(\hfrak)\rtimes {S_n}$-representation by letting the augmentation ideal $S(\hfrak)_+$ act trivially and define the Verma module $M_c(\mu)=\Hrm_c\otimes_{S(\hfrak)\rtimes {S_n}}\mu.$
	The Verma module $M_c(\mu)$ has a unique nonzero irreducible quotient $\Lrm_c(\mu)$.
	\subsection{Simple modules}
	\begin{thm}\label{thm:beg}
		(\cite[Theorem 1.2]{beg})   When $c=\frac{m}{n}>0$ for $m$ coprime to $n$, the unique irreducible finite-dimensional representation of $\Hrm_c$ is $\Lrm_c:=\Lrm_c(\triv)$. Moreover, for $n\geq2$ the algebra $\Hrm_c$ has finite-dimensional representations only when $c=\frac{m}{n}$ for a nonzero integer $m$ coprime to $n$.
	\end{thm}
	Let $\Ocal(\Hrm_c)$ denote category $\Ocal$ for {rational Cherednik algebras}, which is a full subcategory of $\Hrm_c\Mod$ whose objects are finitely generated over $\Hrm_c$ and such that {the augmentation ideal $S(\hfrak)_+$ acts locally nilpotently} \cite{ggor}. As $\mu$ runs over irreducible ${S_n}$-representations, the modules $\Lrm_c(\mu)$ give a complete list of irreducible objects in $\Ocal(\Hrm_c)$.
	
	Let $\e:=\frac{1}{n!}\sum_{w\in {S_n}}w$ be the symmetrizing idempotent in $\Hrm_c$. The spherical Cherednik algebra is $\Arm_c:=\e\Hrm_c \e$. By \cite[Corollary 4.2]{be_parabolic_ind_res}, for  all $c$ satisfying
	\begin{align*}
		c\notin \{\frac{a}{b}\in (-1,0)|a,b\in \Z,2\leq b\leq n\}
	\end{align*}
	there is a Morita equivalence $\Hrm_c\Mod\to \Arm_c\Mod$, $M\mapsto \e M.$
	Let $\Ocal(\Arm_c)$ denote the category $\Ocal$ of the spherical Cherednik algebra.
	
	For any smooth algebraic variety $X$, we denote the sheaf of differential operators on $X$ by $\Dcal_X$ and write $\Dcal(X):=\Gamma(X,\Dcal_X)$. Let $R^+$ be the set of positive roots, put $\delta=\prod_{\alpha\in R^+}\alpha$, and let $\hfrak_r=\{\delta\neq0\}\subset\hfrak$. The polynomial representation gives the Dunkl embedding
	\[
	\Hrm_a\hookrightarrow \Dcal(\hfrak_r)\rtimes {S_n}, \qquad y\mapsto D_{y,a}:=\partial_y-a\sum_{s\in S} \frac{\langle\alpha_s,y\rangle}{\alpha_s}(1-s).
	\]
	After localization, $(\Hrm_a)_\delta\cong\Dcal(\hfrak_r)\rtimes {S_n}$. For each $a\in\C$, let $\phi_a$ be the automorphism of this localized algebra characterized by $\phi_a(x)=x, \phi_a(D_{y,a})=D_{y,-a}, \phi_a(w)=\operatorname{sign}(w)w.$ By \cite[Lemma 5.6 and Proposition 5.8]{ggs},
	\begin{align}
		\label{iso:omega-algebra}
		\phi_{c+1}(\Arm_c)=\delta^{-1}\Arm_{-c-1}\delta.
	\end{align}
	Thus $u\mapsto\delta\phi_{c+1}(u)\delta^{-1}$ identifies $\Arm_c$ with $\Arm_{-c-1}$. We denote the induced equivalence by
	\[
	\Omega_c:\Arm_c\Mod\xrightarrow{\sim}\Arm_{-c-1}\Mod, \qquad \Omega_c^{-1}=\Omega_{-c-1}.
	\]
	\subsection{Mirabolic $D$-modules}
	%\subsubsection{Notations and conventions on $D$-modules}\label{sec:d-modules}
	
	Let $V=\C^n$ on which $\ol{G}$ acts by left multiplication and put $\Gfrak=\gfrak\times V$. Let $\tau:\ol{\gfrak}\to \Dcal(\Gfrak)$ denote the embedding induced by differentiating the diagonal action of $\ol{G}$ on $\Gfrak$. For any $f\in \Ocal(\Gfrak)$, an element $g\in\ol{G}$ acts by
	\begin{align}
		\label{formula:action} (g\cdot f)(x)=f(g^{-1}\cdot x).
	\end{align}
	In particular, writing $E_V=\sum_{i=1}^n v_i\partial_{v_i}$ for the Euler vector field on $V$, we have $\tau(\mathbf{1})=-E_V$.
	
	For a constant $d\in \C$, define the shifted embedding $\tau_d$ by
	\begin{align}
		\label{formula:tau} \tau_d: \ol{\gfrak}\to \Dcal(\Gfrak), \quad x\to \tau(x)-d\cdot\tr(x).
	\end{align}
	Thus $\tau_d(\mathbf{1})=-E_V-nd.$ 
	In Section \ref{sec:cee} we will take $d=-c$.
	
	Recall that every coherent $\Dcal_X$-module $M$ admits a good filtration in the sense that the associated sheaf $\wt{\gr}(M)$ on $T^*X$ is coherent. The support $SS(M)$ of $\wt{\gr}(M)$ in $T^*X$ is  called the singular support of $M$, which is independent of the choice of a good filtration.
	
	Identify $\gfrak\cong \gfrak^*$ via the trace form. Let $\Nscr\subset \gfrak$ denote the nilpotent cone. Inside $T^*\Gfrak\cong \gfrak\times\gfrak\times V\times V^*$, we have a Lagrangian subvariety $\mathcal Y:=\{(X,Y,i,j)\in \gfrak\times\Nscr\times V\times V^*| [X,Y]+ij=0\}_\red$. The subscript $_{\red}$ denotes the reduced structure here and below.
	\begin{defn}(\cite{finkelberg-ginzburg_mirabolic,ganginzburg})
		A finitely generated $\Dcal_\Gfrak$-module $M$ of regular singularities is called {mirabolic} if  $SS(M)\subset \mathcal Y$.
	\end{defn}
	Let $\Cscr(\Gfrak)$ denote the category of mirabolic $D$-modules.
	\subsection{Quantum Hamiltonian reduction}
	Fix a nonzero element $\vol\in \bigwedge^n V^*$. Following \cite[(5.3.2)]{Bezru-Finkelberg-Ginzburg}, consider the regular function on $\Gfrak$ defined by:
	\begin{align}
		\label{formula:s}s(x,v)=\langle \vol, v\wedge xv\wedge \dots\wedge x^{n-1}v\rangle.
	\end{align}
	Take $\Gfrak_\cyc:=\{(x,v)\in\Gfrak|s(x,v)\neq 0\}$, the locus where $v$ is a cyclic vector for  $x$, i.e. $\C[x]v=V$. By \cite[Lemma 5.3.3]{Bezru-Finkelberg-Ginzburg}, the composition of the projection $\Gfrak_\cyc\to \gfrak$ and the Chevalley map $\gfrak\to \gfrak/\!/ G\cong \hfrak/{S_n}$ is a principal $\ol{G}$-bundle and the diagonal $\ol{G}$-action on $\Gfrak_\cyc$ is free along the fibers. Therefore we have an algebra  isomorphism
	\begin{align*}
		\iota: \C[\Gfrak_\cyc]^{\ol{G}}\cong \C[\hfrak]^{S_n}.
	\end{align*}
	Our convention \eqref{formula:action} gives $g\cdot s=\det(g)^{-1}s$, hence $\tau(x)s=-\tr(x)s$. Consequently the multivalued section $s^{-d}$ is annihilated by $\tau_d(\ol{\gfrak})$. 
	
	By \cite[Proposition 2.1.2 and Remark 2.8.3]{bgnearbycycle}, any mirabolic $D$-module is locally finite as a $\tau(\ol{\gfrak})$-module. 
	In particular, its global sections decompose as $\Gamma(M)=\bigoplus_{a+\Z\in\C/\Z}\Gamma(M)_{a+\Z},$ where
	\[
	\Gamma(M)_{a+\Z} =\bigoplus_{k\in\Z}\{v\in\Gamma(M)\mid (\tau(\mathbf{1})-a-k)^Nv=0\text{ for }N\gg0\}.
	\]
	Given $d\in\C$, let $\Cscr_d(\Gfrak)$ be the full subcategory of $\Cscr(\Gfrak)$ consisting of modules $M$ on which $\tau(\mathbf1)$ acts semisimply and such that
	\begin{align*}
		\Gamma(M)=\Gamma(M)_{nd+\Z}.
	\end{align*}
	
	\begin{thm}\label{thm:gg}(\cite[Theorem 7.4.3]{ganginzburg}, \cite[Theorem 8.1]{ggs}) The radial part map
		\begin{align*}
			\Dcal(\Gfrak)^{\ol{G}}\to \Dcal(\hfrak/{S_n}),\quad u\mapsto \bigg[\C[\hfrak/{S_n}]\ni f\mapsto \iota\left(s^d u\left( s^{-d}\iota^{-1}(f)\right)\right)\bigg]
		\end{align*}
		induces an isomorphism
		\begin{align*}
			\Hscr_d: (\Dcal(\Gfrak)/\Dcal(\Gfrak)\tau_d(\ol{\gfrak}))^{\ol{G}}\cong \Arm_{d-1}.
		\end{align*}
		Thus there is a quantum Hamiltonian reduction functor
		\[
		\Hbb_d: \Cscr_d(\Gfrak)\to \Ocal(\Arm_{d-1}),\quad M\mapsto \Gamma(\Gfrak,M)^{\tau_d(\ol{\gfrak})}, 
		\]
		which is exact and the kernel is a Serre subcategory such that $\Cscr_d(\Gfrak)/\Ker(\Hbb_d)\cong \Ocal(\Arm_{d-1}).$
	\end{thm}
	\begin{rem}
		The shift by ``$-1$" in the subscript of $\Arm_{d-1}$ is consistent with the classical result of Harish-Chandra: For $u\in \Dcal(\gfrak)^G$, define $\mathrm{Rad}(u)$ to be the unique $S_n$-invariant differential operator on the regular semisimple locus $\hfrak_r$ such that $\mathrm{Rad}(u)(f|_{\hfrak_r})=(uf)|_{\hfrak_r}$ for every $f\in\C[\gfrak]^G$. Then the map 
		\[
		\Dcal(\gfrak)^G\to \Dcal(\hfrak/{S_n}),\quad u\mapsto \bigg[\C[\hfrak/{S_n}]\ni f\mapsto \delta \mathrm{Rad}(u)\left( \delta^{-1}f\right)\bigg]
		\]
		induces an isomorphism $\left(\Dcal(\gfrak)/\Dcal(\gfrak)\ad(\gfrak)\right)^G\cong\Dcal(\hfrak)^{S_n}$. 
	\end{rem}
	
	\subsection{Cuspidal $D$-modules}
	\subsubsection{Functors on $D$-modules}
	Let $f: X\to Y$ be a morphism of smooth algebraic varieties. We write $f_\dagger, f_!, f^\dagger$ to denote direct image, direct image with proper support and inverse image functors of holonomic $D$-modules respectively, all valued in bounded derived categories of holonomic $D$-modules \cite[Section 1.5]{htt}. In particular, let $\Dcal_{X\to Y}$ and $\Dcal_{Y\leftarrow X}$ be the transfer $D$-modules as defined in \cite[1.3]{htt}. We then have $f_\dagger:=Rf_*(\Dcal_{Y\leftarrow X}{\otimes_{\Dcal_X}^{\mathrm L} -)}$, $f_!=\Dbb_Y\circ f_\dagger\circ \Dbb_X$ and $f^\dagger={\Dcal_{X\to Y}\otimes_{f^{-1}\Dcal_Y}^{\mathrm L} f^{-1}(-)} $.
	
	When $f$ is a locally closed embedding, one defines the minimal extension functor $f_{!*}$ {by $f_{!*}M:=\operatorname{im}(\mathcal H^0(f_!M)\to\mathcal H^0(f_\dagger M))$ in the abelian category of holonomic $D$-modules}. For a locally closed subset  $U\xrightarrow{i}X$, an integrable connection on $U$ (i.e. {an $\Ocal_U$-coherent} $\Dcal_U$-module) is called clean if its minimal extension coincides with the extensions using $i_\dagger$ and $i_!$.
	
	%Since the local system we intend to take minimal extension of is clean, we will not get into details of minimal extensions but stick to the $\dagger$-extension throughout.
	\subsubsection{Integrable connections} \label{sec:loc-sys}
	The zero fiber of the fibration $\Gfrak_\cyc\to \hfrak/{S_n}$ is
	\begin{align*}
		U:=\{(x,v)\in\Nscr\times V|s(x,v)\neq 0\}.
	\end{align*}
	The diagonal $\ol{G}$-action on $U$ is transitive and free, and $\pi_1(U)=\Z$. Since $g\cdot s=\det(g)^{-1}s$, every rank-one regular integrable connection on $U$ with finite monodromy is generated by $s^a$ for some $a\in\Q$. We denote the corresponding integrable connection by $\ol{\Ecal}_a$. 
	
	Let $\Nscr_r\subset \Nscr$ denote the regular nilpotent orbit. Since $\pi_1(G)=1$, the fibration \[\Stab_G(x)\to G\to\Nscr_r\] induced by the action of $G$ at a regular nilpotent $x$ gives \[\pi_1(\Nscr_r)\cong\pi_0(\Stab_G(x))=\pi_0(Z(G))=\Z/n\Z.\] For $b\in\frac1n\Z,$ define the rank-one integrable connection ${\Ecal}_b$ on $\Nscr_r$ whose monodromy is $e^{2\pi ib}$. Both $\ol{\Ecal}_a$ and ${\Ecal}_b$ are $G$-equivariant.
	
	The projection $U\to \Nscr$ has image inside $\Nscr_r$ and the fibers of this projection are isomorphic to $ \Ffrak:=\C^{n-1}\times\C^*$. The fibration
	\[
	\xymatrix{ \Ffrak\ar[r]& U\ar[r] &\Nscr_r}
	\]
	induces an exact sequence
	\[
	1\to(\Z=\pi_1(\Ffrak))\xrightarrow{n}\left(\Z=\pi_1(U)\right)\to\left(\pi_1(\Nscr_r)=\Z/n\Z\right)\to 1.
	\]
	Thus, for $a\in\frac1n\Z$, the connection $\ol{\Ecal}_a$ is the pullback of $\Ecal_a$.
	\subsubsection{Cuspidal $D$-modules}\label{sec:cee}
	Let $c=\frac{m}{n}$ {for a positive integer $m$} coprime to $n$.	
	\begin{defn}
		\begin{itemize}
			\item The {cuspidal} character $\Dcal_\gfrak$-module $\Nbb_{c}$ is the minimal extension of $\Ecal_{c}$ to $\gfrak$.
			\item The {cuspidal} mirabolic $\Dcal_\Gfrak$-module $\ol{\Nbb}_c$ of parameter $c$ is the minimal extension of $\ol{\Ecal}_{c}$ to $\Gfrak$.
		\end{itemize}
	\end{defn}
	\begin{lem}\label{lem:mirabolic-cuspidal}
		As $G$-equivariant $\Dcal_\Gfrak$-modules, $\ol{\Nbb}_c\cong\Nbb_c \boxtimes \Ocal_V$. Moreover, 
		\begin{align}
			\label{formula:u} SS(\ol{\Nbb}_c)=\Cfrak_\nil\times V,\text{ for } \Cfrak_\nil:=\{(x,y)\in \Nscr\times \Nscr| [x,y]=0\}_\red.
		\end{align}
	\end{lem}
	\begin{proof}
		Along each fiber of $U\to\Nscr_r$, the function $s$ restricts to the $n$th power of a coordinate along $V$. Since $nc=m\in\Z$, the connection $\ol{\Ecal}_c$ has trivial monodromy along the fiber. Its minimal extension to $\Nscr_r\times V$ is therefore $\Ecal_c\boxtimes\Ocal_V$. Taking minimal extension to $\Gfrak$, we obtain the first statement. 
		
		By \cite[Theorem 5(b)]{Lusztig_fourier_cuspidal}, the Fourier transform of a cuspidal module has nilpotent support.
		Since $\Nbb_c$ is $G$-equivariant and has nilpotent singular support, we have $SS(\Nbb_c)\subset\Cfrak_\nil$.
		By \cite{Baranovsky}, $\Cfrak_\nil$ is irreducible of dimension $\dim\gfrak$. Since $\Nbb_c$ is nonzero, the inclusion $SS(\Nbb_c)\subset\Cfrak_\nil$ has to be an equality. 
	\end{proof}
	\subsection{Calaque-Enriquez-Etingof  construction}
	Write $N_c:=\Gamma(\gfrak,\Nbb_c)$. Below we write $\Hrm_a^{(r)}$, $\Arm_a^{(r)}$, and $\Lrm_a^{(r)}$ for the type $A_{r-1}$ objects, $V_n=\C^n$ and write $\e^{(r)}$ for the symmetrizing idempotent of $S_r$ when objects of different ranks appear at once.
	\begin{thm}[{\cite[Theorem 68]{cee}}]
		\label{thm:cee}
		There is a functor from $\Dcal(\sln)$-modules to $\Hrm_{n/m}^{(m)}$-modules that sends $M$ to $(M\otimes (V_n^*)^{\otimes m})^{\SL_n}$ such that 
		\begin{align}
			(N_{m/n}\otimes(V_n^*)^{\otimes m})^G\cong\Lrm_{n/m}^{(m)}.
			\label{iso: cee iso}
		\end{align}
		In particular, the symmetric group $S_m$ acts by permuting the tensor factors.
	\end{thm}
	As a result, applying $\Hbb_{-m/n}$ to $\ol{\Nbb}_{m/n}$, we obtain
	\begin{align}
		\Hbb_{-{m/n}}(\ol{\Nbb}_{m/n})
		&\cong(N_{m/n}\otimes\Sym^m(V_n^*))^G\cong \e^{(m)}\Lrm_{n/m}^{(m)}
		\label{eqn: cuspidal red}
	\end{align}
	as vector spaces. The left side carries an $\Arm_{-m/n-1}^{(n)}$-action. By Theorem \ref{thm:gg}, Hamiltonian reduction sends a simple object to a simple object or zero. 
	Applying $\Omega_c^{-1}$ and using the uniqueness of the finite-dimensional simple module for $c=m/n>0$ gives
	\[\Omega_c^{-1}\circ\Hbb_{-c}(\ol{\Nbb}_c)
	\cong\e^{(n)}\Lrm_c^{(n)}.\]
	
	More generally, let $\hfrak_m$ be the reflection representation of $S_m$. For $0\leq i<m$, Schur-Weyl duality and the Koszul differential give
	\begin{align}
		\Hom_{S_m}(\wedge^i\hfrak_m,(V_n^*)^{\otimes m}) \cong\ker\bigl(\Sym^{m-i}(V_n^*)\otimes\wedge^i(V_n^*) \rightarrow\Sym^{m-i+1}(V_n^*)\otimes\wedge^{i-1}(V_n^*)\bigr).
	\end{align}
	Consequently, we have an exact sequence
	\begin{align}
		\label{eqn: hook kernel}
		0\rightarrow \Hom_{S_m}(\wedge^i\hfrak_m,\Lrm_{n/m}^{(m)}) \rightarrow (N_{m/n}\otimes\Sym^{m-i}(V_n^*)\otimes\wedge^i(V_n^*))^{\SL_n} \rightarrow (N_{m/n}\otimes\Sym^{m-i+1}(V_n^*)\otimes\wedge^{i-1}(V_n^*))^{\SL_n}.
	\end{align}	
	\section{Hodge filtrations on cuspidal $D$-modules}\label{sec:hodge}
	\subsection{Hodge modules}
	In Saito's theory \cite{saitofrench,saito}, a polarizable Hodge module on a smooth algebraic variety $X$ consists of a regular holonomic $\Dcal_X$-module $M$, a good filtration $F_\bullet M$ compatible with the order filtration on $\Dcal_X$ called the Hodge filtration and a polarization, satisfying certain compatibilities. We use the complex version in the sense of Sabbah and Schnell \cite{Sabbah-Schnell}, in which the $\Q$-structure is omitted. 
	
	Let $Z\subset X$ be an irreducible closed subvariety. A polarizable pure Hodge module on $X$ with strict support $Z$ restricts, over a suitable smooth dense open subset $U\subset Z$, to a polarizable variation of Hodge structure (VHS). Conversely, Saito's structure theorem asserts that such a VHS extends uniquely to a polarizable pure Hodge module on $X$ with strict support $Z$.
	Here and below, ``strict support $Z$'' means that the Hodge module has no nonzero subobject or quotient supported on a proper closed subset of $Z$.
	
	The integrable connection $\Ecal_c$ on $\Nscr_r$ has finite monodromy and hence admits a flat positive definite Hermitian metric. We can therefore equip it with a polarizable complex VHS structure of type $(0,0)$ such that the Hodge filtration is
	\[
	F_k\Ecal_c=
	\begin{cases}
		0,&k<0,\\
		\Ecal_c,&k\geq0.
	\end{cases}
	\]
	A uniform shift of the filtration would amount to a Tate twist. Saito's structure theorem uniquely determines the Hodge module structures on the minimal extensions $\Nbb_c$ of weight $\dim\Nscr$.
	
	Our goal is to compute the associated graded of $\Nbb_c$ with respect to the Hodge filtration. Fix a Borel subgroup $B\subset G$ whose Lie algebra $\bfrak\subset\gfrak$ consists of upper triangular matrices. We will first describe explicitly the filtration of the restriction to $\bfrak$. All the statements in this section can be made analogously for $\ol{\Nbb}_c$ upon taking $\boxtimes \Ocal_V$. 
	\subsection{Deligne's canonical lattice}\label{sec:deligne}
	
	\begin{thm}[{\cite{deligneequation},  \cite[Theorem 5.2.17]{htt}}]
		\label{thm:deligne}
		Let $D\subset X$ be a simple normal crossing divisor, let $j:U:=X\setminus D\hookrightarrow X$ be the open embedding, and let $\Ecal$ be a regular integrable connection on $U$ with quasi-unipotent local monodromy. There is a unique locally free $\Ocal_X$-submodule $\Ecal^{>-1}\subset j_*\Ecal$ extending $\Ecal$, called Deligne's canonical lattice, with the following properties:
		\begin{enumerate}
			\item The connection preserves the lattice logarithmically: $\nabla(\Ecal^{>-1})\subset\Omega_X^1(\log D)\otimes\Ecal^{>-1}$.
			\item The eigenvalues of the residue along every irreducible component of $D$ lie in $(-1,0]$.
		\end{enumerate}
	\end{thm}
	For rank one case, if the lattice is locally generated by a function $f$ such that $x_i\partial_{x_i}f=\alpha_i f$, then the connection can be expressed as $d+\sum_i\alpha_i\,dx_i/x_i$ and the residue is $\alpha_i$. 
	
	Keep the setting of Theorem \ref{thm:deligne}, and suppose in addition that $\Ecal$ underlies a polarizable VHS with Hodge filtration $F_\bullet\Ecal$. The unique polarizable pure Hodge module with strict support $X$ extending this VHS has underlying $D$-module $j_{!*}\Ecal=\mathcal D_X\cdot\Ecal^{>-1}\subset j_*\Ecal.$ 
	By \cite[3.b]{saito} (see also \cite{popa}), the canonical lattice is filtered by locally free $\Ocal_X$-submodules such that $F_k\Ecal^{>-1}=\Ecal^{>-1}\cap j_*F_k\Ecal,$ and the Hodge filtration on the minimal extension is $F_k(j_{!*}\Ecal) =\sum_i F_i\Dcal_X\cdot F_{k-i}\Ecal^{>-1}.$
	
	\subsection{Hodge filtrations restricted to $\bfrak$}
	Set $\nfrak=[\bfrak,\bfrak]$ and $\nfrak_r=\nfrak\cap\Nscr_r$. We use coordinates on $\nfrak\times V$ of the form
	\begin{align}
		\label{formula:coordinates-x}
		\begin{pmatrix}
			0&x_1&*&*&*&*\\
			0&0&x_2&*&*&*\\
			&&&\cdots&x_{n-2}&*\\
			0&0&0&\cdots&0&x_{n-1}\\
			0&0&0&\cdots&0&0
		\end{pmatrix}
		\begin{pmatrix}v_1\\\vdots\\\vdots\\v_n\end{pmatrix}.
	\end{align}
	The complement $D:=\nfrak\setminus\nfrak_r$ is the simple normal crossing divisor defined by $x_1x_2\cdots x_{n-1}=0$.
	
	Recall the $\Dcal_U$-module $\ol{\Ecal}_c$ from Section \ref{sec:loc-sys}.
	The restriction of $s^c$ from \eqref{formula:s} to $\nfrak\times V$ is $x_1^c x_2^{2c}\cdots x_{n-1}^{(n-1)c}v_n^{nc}.$ Set $U_0:=U\cap(\nfrak\times V)$.  Then as a cyclic $\Dcal_{U_0}$-module with generator $s^c|_{U_0}$, $\ol{\Ecal}_c|_{U_0}$ can be identified as 
	\[
	\ol{\Ecal}_c|_{U_0}\cong \frac{\Dcal_{U_0}} {\displaystyle \sum_{i=1}^{n-1}\Dcal_{U_0}(x_i\partial_{x_i}-ic) +\Dcal_{U_0}(v_n\partial_{v_n}-nc) +\sum_{j>i+1}\Dcal_{U_0}\partial_{x_{ij}} +\sum_{j<n}\Dcal_{U_0}\partial_{v_j}}.
	\]
	Similarly,
	\[
	{\Ecal}_c^\circ:=\Ecal_c|_{\nfrak_r} \cong \frac{\Dcal_{\nfrak_r}} {\sum_{i=1}^{n-1}\Dcal_{\nfrak_r}(x_i\partial_{x_i}-ic) +\sum_{j>i+1}\Dcal_{\nfrak_r}\partial_{x_{ij}}}.
	\]
	with
	\[
	F_k\Ecal_c^\circ=
	\begin{cases}
		\Ecal_c^\circ,&k\geq0,\\
		0,&k<0.
	\end{cases}
	\]
	
	Write $i_0:\nfrak_r\hookrightarrow\nfrak$ and set $\Lbf_c^\circ:=(i_0)_{!*}\Ecal_c^\circ.$
	For $1\leq i\leq n-1$, the number $ic=im/n$ is not an integer by the coprime assumption. Thus the local monodromy of $\Ecal_c^\circ$ around every component $x_i=0$ is nontrivial, and the extension is clean: $\Lbf_c^\circ\cong(i_0)_\dagger\Ecal_c^\circ.$
	
	Put $s_0:=\prod_{i=1}^{n-1}x_i^i$ and 
	\begin{align*}
		\lceil s_0^c\rceil :=x_1^{\lceil c\rceil}x_2^{\lceil2c\rceil}\cdots x_{n-1}^{\lceil(n-1)c\rceil}.
	\end{align*}
	With respect to the generator $\lceil s_0^c\rceil^{-1}s_0^c$ of $\Ecal_c^\circ$, the connection is 
	\[d+\sum_{i=1}^{n-1}(ic-\lceil ic\rceil)\frac{dx_i}{x_i}. \]
	Consequently, the lattice $\Ocal_{\nfrak}\lceil s_0^c\rceil^{-1}s_0^c$ is preserved logarithmically by the connection. Its residue along $x_i=0$ is  $ic-\lceil ic\rceil\in(-1,0)$. So it is Deligne's canonical lattice with filtration
	\[
	F_k(\Ecal_c^\circ)^{>-1} =
	\begin{cases}
		\Ocal_\nfrak\lceil s_0^c\rceil^{-1}s_0^c,&k\geq0,\\
		0,&k<0.
	\end{cases}
	\]
	It follows that
	\begin{align}
		\label{iso:open-emb-hodge-fil}
		F_k\Lbf_c^\circ
		&=\sum_i F_i\Dcal_\nfrak\cdot
		F_{k-i}(\Ecal_c^\circ)^{>-1}
		=\begin{cases}
			(F_k\Dcal_\nfrak)\cdot \lceil s_0^c\rceil^{-1}s_0^c,&k\geq0,\\
			0,&k<0.
		\end{cases}
	\end{align}
	
	Let $i_\nfrak:\nfrak\hookrightarrow\gfrak$ be the closed embedding, and put $\Dcal_{\gfrak\leftarrow\nfrak} :=i_\nfrak^*\Dcal_\gfrak\otimes_{\Ocal_\nfrak} \omega_{\nfrak/\gfrak}.$  
	Set $r:=\operatorname{codim}_{\gfrak}\nfrak=\dim\bfrak.$
	For the direct image of left $D$-modules, the filtration on the transfer module is:
	\[
	F_j\Dcal_{\gfrak\leftarrow\nfrak}:=i_\nfrak^*F_{j-r}\Dcal_\gfrak\otimes_{\Ocal_\nfrak}\omega_{\nfrak/\gfrak}.
	\]
	The direct image of $(\Lbf_c^\circ,F_\bullet)$ has underlying $\Dcal_\gfrak$-module
	\[
	\Lbf_c:=(i_\nfrak)_\dagger\Lbf_c^\circ =i_{\nfrak*}\left(\Dcal_{\gfrak\leftarrow\nfrak}\otimes_{\Dcal_\nfrak}\Lbf_c^\circ\right).
	\]
	with Hodge filtration
	\begin{align}
		\label{iso:closed-emb-hodge-fil}
		F_k\Lbf_c =\operatorname{im}\left((i_{\nfrak})_*\left( \sum_j F_j\Dcal_{\gfrak\leftarrow\nfrak} \otimes F_{k-j}\Lbf_c^\circ\right) \rightarrow (i_\nfrak)_\dagger\Lbf_c^\circ \right).
	\end{align}
	Equivalently, taking a trivialization of $\omega_{\nfrak/\gfrak}$ and normal coordinates $t_1,\ldots,t_r$, we have $F_k\Lbf_c =\sum_{\nu\in\mathbb N^r} F_{k-r-|\nu|}\Lbf_c^\circ\otimes\partial_t^\nu.$
	\subsection{Associated graded of $\Lbf_c$}
	
	Recall the standard $\mathfrak{sl}_2$-triple in $\gfrak$:
	\begin{align}
		\label{formula:efh}
		\ul{e}&=
		\begin{pmatrix}
			0&1&\cdots&0&0\\
			&&\cdots&\\
			0&0&\cdots&1&0\\
			0&0&\cdots&0&1\\
			0&0&\cdots&0&0
		\end{pmatrix},
		&
		\ul{f}&=
		\begin{pmatrix}
			0&0&\cdots&0&0\\
			n-1&0&\cdots&0&0\\
			0&2(n-2)&\cdots&0&0\\
			&&\cdots&\\
			0&0&\cdots&n-1&0
		\end{pmatrix},
		&
		\ul{h}&=[\ul{e},\ul{f}].
	\end{align}
	Define the Hessenberg space $\mfrak:=\overline{B\cdot(\ul{f}+\operatorname{stab}_\gfrak(\ul{e}))};$ equivalently, $\mfrak$ is the sum of $\bfrak$ and the negative simple-root spaces.
	
	\begin{rem}
		The affine space $\ul{f}+\operatorname{stab}_\gfrak(\ul{e})$ is the Kostant slice. The Whittaker reduction of $G\times^B\mfrak$ is the relative compactification of the universal centralizer for $G/Z(G)$ \cite{balibanu}.
	\end{rem}
	
	Using the coordinates \eqref{formula:coordinates-x} and
	\begin{align}
		\label{formula:coordinates-y}
		\begin{pmatrix}
			*&*&\cdots&*&*&*\\
			y_1&*&\cdots&*&*&*\\
			0&y_2&\cdots&*&*&*\\
			&&\cdots&&\\
			0&0&\cdots&y_{n-2}&*&*\\
			0&0&\cdots&0&y_{n-1}&*
		\end{pmatrix},
	\end{align}
	define
	\begin{align}
		\label{eqn: Y0}
		\Yfrak_0 :=\left\{(x,y)\in\nfrak\times\mfrak \mathrel{\big|}x_i y_i=0\text{ for }1\leq i\leq n-1\right\}.
	\end{align}
	Let $i_{\Yfrak_0}:\Yfrak_0\hookrightarrow T^*\gfrak$ denote the resulting closed embedding.
	
	Let $\Ocal(\bfrak_-)_{+}$ and $S([\nfrak,\nfrak])_{+}$ denote the augmentation ideals. There is an identification
	\[
	\Gamma(\gfrak,\Lbf_c) \cong \frac{\Dcal(\gfrak)} {\displaystyle \Dcal(\gfrak)\cdot\Ocal(\bfrak_-)_{+} +\sum_{i=1}^{n-1}\Dcal(\gfrak)(x_i\partial_{x_i}-ic) +\Dcal(\gfrak)\cdot S([\nfrak,\nfrak])_{+}}.
	\]
	Taking associated graded with respect to \eqref{iso:open-emb-hodge-fil} and \eqref{iso:closed-emb-hodge-fil} gives $\wt{\gr}(\Lbf_c) \cong(i_{\Yfrak_0})_*\Ocal_{\Yfrak_0}.$ It remains to identify the equivariant structure. 
	Similar to \eqref{formula:tau}, define
	\[
	\tau'_{-c}:\ol{\gfrak}\rightarrow\Dcal(\gfrak), \qquad x\mapsto\tau'(x)+c\operatorname{tr}(x).
	\]
	Let $\ol B\subset \ol{G}$ be the upper-triangular Borel subgroup.
	Although $\Lbf_c$ is only $B$-equivariant, its associated graded  carries a natural $\ol{B}$-equivariant structure:
	
	\begin{lemconst}\label{lem:const-equivariance}
		The $\ol{\bfrak}$-action on $\wt{\gr}(\Lbf_c)$ induced by $\tau'_{-c}$ integrates to a $\ol{B}$-action whose restriction to $B$ agrees with the original $B$-equivariant structure.
	\end{lemconst}
	\begin{proof}
		Let $u$ be the cyclic generator corresponding to $s_0^c$ and a constant trivialization of $\omega_{\nfrak/\gfrak}$. Under $\tau'_{-c}$, $u$ is annihilated by $\nfrak$ and has integral weight $\lambda=(0,\ldots,0,m)-2\rho$. Equip $\Dcal(\gfrak)\otimes\C_\lambda$ with the $\ol{B}$-action given by conjugation and the character $\lambda$. The surjection
		\[\Dcal(\gfrak)\otimes\C_\lambda
		\to\Gamma(\gfrak,\Lbf_c),\qquad P\otimes1\mapsto Pu,\]
		is $\ol{\bfrak}$-equivariant. Since $\ol{B}$ is connected, its kernel is $\ol{B}$-stable, so the action descends to $\Lbf_c$.
		
		The generator $\lceil s_0^c\rceil^{-1}u$ is a $\ol{B}$-semi-invariant, and \eqref{iso:open-emb-hodge-fil} and \eqref{iso:closed-emb-hodge-fil} give
		\[F_k\Lbf_c=(F_{k-r}\Dcal_{\gfrak})\lceil s_0^c\rceil^{-1}u.\]
		Since conjugation preserves the order filtration, the action preserves $F_\bullet\Lbf_c$ and descends to the associated graded. Its restriction to $B$ agrees with the original action because their differentials agree and $B$ is connected.
	\end{proof}
	
	Define the weight
	\begin{align*}
		\lceil\mu_c\rceil
		&:=\bigl(\lceil c\rceil,\lceil2c\rceil-\lceil c\rceil,\ldots,
		\lceil nc\rceil-\lceil(n-1)c\rceil\bigr),
	\end{align*}
	Give $\Ocal_{\Yfrak_0}$ the natural $\ol{B}$-equivariant structure. For a character $\lambda$ of $\ol{T}$, let $\C_\lambda$ denote the corresponding $1$-dimensional representation of $\ol{B}$.  Let $\rho$ be the half sum of positive roots.
	Since $\omega_{\nfrak/\gfrak}=\det(\gfrak/\nfrak)$ and $\gfrak/\nfrak\cong\hfrak\oplus\nfrak^-$ as $\ol T$-modules, the character of $\omega_{\nfrak/\gfrak}$ is $\det(\nfrak^-)=-2\rho$.
	
	For a $\C^*$-equivariant sheaf $\Fcal$, or equivalently a grading $\Fcal=\oplus_{k\in \Z} \Fcal_k$, let $\Fcal\langle \ell\rangle $ denote the shift such that 
	\begin{align}
		\label{eqn: grading shift}
		(\Fcal\langle \ell\rangle)_k=\Fcal_{k-\ell}.
	\end{align}
	\begin{lem}\label{lem:l}
		With the equivariant structure of Lemma-Construction \ref{lem:const-equivariance} and Hodge filtration, there is a $\ol{B}\times\C^*$-equivariant isomorphism 
		\[
		\wt{\gr}(\Lbf_c) \cong(i_{\Yfrak_0})_*\Ocal_{\Yfrak_0} \otimes\C_{\lceil\mu_c\rceil-2\rho}\langle\dim\bfrak\rangle.
		\]
		where the $\C^*$-equivariance comes from filtration grading. 
	\end{lem}
	
	\begin{proof}
		We have already identified $\wt{\gr}(\Lbf_c)$ with $(i_{\Yfrak_0})_*\Ocal_{\Yfrak_0}$ after forgetting equivariance. Moreover, the lowest degree is $\gr_r(\Lbf_c) = \omega_{\nfrak/\gfrak} \otimes\lceil s_0^c\rceil^{-1}s_0^c.$
		Under the action in \eqref{formula:action}, $\ol{B}$-acts on $\lceil s_0^{c}\rceil^{-1}$ through the quotient $\ol{B}\twoheadrightarrow \ol{T}$ by the  character
		\[
		\bigl(\lceil c\rceil ,\lceil 2c\rceil-\lceil c\rceil,\ldots,\lceil (n-1)c\rceil-\lceil (n-2)c\rceil,-\lceil (n-1)c\rceil\bigr).
		\]
		On the other hand, the shifted action $\tau'_{-c}$ contributes the character $(0,\dots,0,m)$. The sum of these two characters is $\lceil\mu_c\rceil$. Finally, $\omega_{\nfrak/\gfrak}$ contributes the character $-2\rho$ and the grading shift $\langle \dim \bfrak\rangle$.
	\end{proof}
	
	\subsection{Functors on equivariant Hodge modules}\label{sec:ind}
	
	We recall some facts about equivariant Hodge modules that will be used to express $\wt{\gr}(\Nbb_c)$ in terms of $\wt{\gr}(\Lbf_c)$. Let  $\HM^G(X)$ denote the abelian category of $G$-equivariant Hodge modules on $X$. Associated graded in this section is taken with respect to Hodge filtrations. 
	
	Let $X$ and $Y$ be two smooth varieties with actions of an affine group $G$. Let $f:X\to Y$ be a $G$-equivariant morphism. We have the cotangent correspondence
	\begin{align}
		\label{diag:cotangent correspondence}
		T^*X \xleftarrow{\ \rho_f\ } X\times_YT^*Y \xrightarrow{\ \varpi_f\ } T^*Y
	\end{align}
	where $\rho_f$ is induced by the differential of $f$ and $\varpi_f$ is the natural projection, both $G$-equivariant. 
	
	\subsubsection{Smooth inverse image}
	
	\begin{lem}[{\cite[Theorem 30.1]{schnell}}]
		\label{lem:gr-pullback}
		Suppose $f:X\to Y$ is smooth. For every $M\in\HM^G(Y)$, there is an isomorphism of $G$-equivariant coherent sheaves on $T^*X$
		\[\wt{\gr}(f^\dagger M) \cong (\rho_f)_*\varpi_f^*\wt{\gr}(M).\]
	\end{lem}
	
	\subsubsection{Direct image}
	Suppose $f$ is projective and take $M\in\HM^G(X)$. Saito's direct image theorem \cite[Theorem 2.14]{saito} implies that the filtered direct image $f_\dagger M$ is strict and that $R^if_\dagger M$ is a Hodge module on $Y$ for every $i$. The equivariant structures on these direct images are supplied by \cite[Theorem 8.5]{achar}.
	\begin{thm}(\cite[Theorem 2.4]{popa-schnell}, \cite[Section 2.3.2]{laumon})
		\label{thm:laumon-thm}
		Suppose $f$ is projective. For every $M\in\HM^G(X)$, there is a isomorphism of $G$-equivariant coherent sheaves on $T^*Y$
		\begin{align*}
			\wt{\gr}(R^if_\dagger M) \cong R^i(\varpi_f)_*\left( \wt{\omega}_{X/Y}\otimes {L\rho_f^*\wt{\gr}(M)} \right).
		\end{align*}
	\end{thm}
	Note that the $\C^*$-grading of $\wt{\omega}_{X/Y}$ is $\dim(Y)-\dim(X)$. 
	
	\subsubsection{Induction}
	Let $H\subset G$ be a closed subgroup and $X$ a smooth $H$-variety. Let $H$ act freely on $G\times X$ by $h\cdot(g,x)=(gh^{-1},hx),$ and write $G\times^H X$ for the quotient. Consider the natural maps
	\begin{align*}
		X\xleftarrow{\operatorname{pr}_X}G\times X \xrightarrow{\pi}G\times^H X.
	\end{align*}
	Equivariant descent gives the induction equivalence $\Ind_H^G:\HM^H(X)\xrightarrow{\sim}\HM^G(G\times^H X)$ \cite[Section~6]{achar}, characterized by the isomorphism
	\begin{align}
		\label{formula:induction-descent}
		\pi^\dagger\Ind_H^G L\cong\operatorname{pr}_X^\dagger L
	\end{align}
	in $\HM^{G\times H}(G\times X)$.
	
	Let $\mu_H:T^*X\to\operatorname{Lie}(H)^*$ be the moment map, with scheme-theoretic zero fiber $\mu_H^{-1}(0)$.
	Strong equivariance implies that the infinitesimal action preserves the Hodge filtration, $\wt{\gr}(L)$ is supported on $\mu_H^{-1}(0)$ scheme-theoretically.
	
	Let $\iota:G\times\mu_H^{-1}(0)\hookrightarrow G\times T^*X$ be the natural inclusion. For $(x,\xi)\in T^*X$ such that $\mu_H(x,\xi)=0$, the covector $(d\operatorname{pr}_X)_{(g,x)}^*\xi$ annihilates the tangent space to the $H$-orbit through $(g,x)$, hence lies in the image of $\rho_\pi$. Since $\rho_\pi$ is a closed embedding identifying its source with the annihilator of $\ker(d\pi)$, the composite $\rho_{\operatorname{pr}_X}\circ\iota$ factors through a closed embedding
	\begin{align*}
		\wt{s}:G\times\mu_H^{-1}(0) \hookrightarrow (G\times X)\times_{G\times^H X}T^*(G\times^H X),
	\end{align*}
	satisfying $\rho_{\operatorname{pr}_X}\circ\iota=\rho_\pi\circ\wt{s}$. Let $s:G\times^H\mu_H^{-1}(0)\to T^*(G\times^H X)$ be the descent of $\wt{s}$ modulo $H$. 
	
	Writing $r:G\times\mu_H^{-1}(0)\to\mu_H^{-1}(0)$ for projection, define coherent-sheaf induction by the equivariant descent identity $q^*\ind_H^G\Fscr\cong r^*\Fscr$.
	\begin{prop}\label{prop:induction-thm}
		For $L\in\HM^H(X)$, there is an isomorphism
		\begin{align*}
			\wt{\gr}(\Ind_H^G L)\cong s_*\ind_H^G\wt{\gr}(L)
		\end{align*}
		in $\Coh^{G\times\mathbb C^*}(T^*(G\times^H X))$.
		Consequently, $SS(\Ind_H^G L)\subset s(G\times^H\mu_H^{-1}(0))$.
	\end{prop}
	
	\begin{proof}
		Put $M=\Ind_H^G L$ and regard $\Fscr=\wt{\gr}(L)$ as a sheaf on $\mu_H^{-1}(0)$.
		Lemma~\ref{lem:gr-pullback} and \eqref{formula:induction-descent} give
		\begin{align}
			\label{formula:pullback}
			(\rho_\pi)_*\varpi_\pi^*\wt{\gr}(M) \cong(\rho_{\operatorname{pr}_X})_* \varpi_{\operatorname{pr}_X}^*\wt{\gr}(L).
		\end{align}
		Here the last associated graded is viewed on $T^*X$.
		Flat pullback along $\varpi_{\operatorname{pr}_X}$ identifies the right-hand side with
		\begin{align*}
			(\rho_{\operatorname{pr}_X}\circ\iota)_*r^*\Fscr \cong(\rho_\pi)_*\wt{s}_*q^*\ind_H^G\Fscr.
		\end{align*}
		Since $\rho_\pi$ is a closed embedding, its pushforward is fully faithful. Cancelling it and applying flat base change gives
		\begin{align*}
			\varpi_\pi^*\wt{\gr}(M) \cong\wt{s}_*q^*\ind_H^G\Fscr\cong\varpi_\pi^*(s_*\ind_H^G\Fscr).
		\end{align*}
		These natural isomorphisms are $G\times H\times\mathbb C^*$-equivariant, so descent along the principal $H$-bundle $\varpi_\pi$ proves the claim. 
	\end{proof}
	
	\subsection{Springer cuspidal $D$-modules}
	\label{sec:springer-cuspidal-d-modules}
	
	Let $\Bcal$ denote the flag variety. Since every regular nilpotent element $x$ is contained in a unique Borel subalgebra, there are locally closed embeddings of $\Nscr_r$ into $\Bcal\times \gfrak\cong G\times^B\gfrak$ such that the following diagram commutes:
	\[
	\xymatrix{ \Nscr_r\ar@{^{(}->}[r]\ar@{_{(}->}[rd]_{\wt i}
		&\gfrak
		\\
		&G\times^B \gfrak\ar[u]_-p
	}
	\]
	Recall the integrable connection $\Ecal_c$ from Section \ref{sec:loc-sys}. Define $\Mbb_c:=(\wt i)_{!*}\Ecal_c$. %$\ol{\Mbb}_c:=(\wt{\ol i})_{!*}\ol{\Ecal}_c.$
	By cleanness of $\Ecal_c$, we have $\Mbb_c\cong(\wt i)_\dagger\Ecal_c.$ 
	Functoriality of direct image and cleanness give $p_\dagger\Mbb_c\cong\Nbb_c$.
	%On the other hand, the same argument as Lemma \ref{lem:mirabolic-cuspidal} gives $\ol{\Mbb}_c\cong\Mbb_c\boxtimes\Ocal_V$. Therefore $\ol p_\dagger\ol{\Mbb}_c\cong\ol{\Nbb}_c$ as well. Both direct images are concentrated in cohomological degree zero.
	\begin{lem}\label{lem:m}
		As $G$-equivariant Hodge modules, $\Ind_B^G \Lbf_c \cong \Mbb_c$. 
	\end{lem}
	
	\begin{proof}
		Consider the following commutative diagram, whose left and middle squares are Cartesian:
		\[
		\xymatrix{ \nfrak_r\ar[d]_{i_{0,r}}
			&G\times\nfrak_r\ar[l]_{p_0}\ar[r]^-{\pi_0}\ar[d]^{i_r}
			&G\times^B\nfrak_r\ar@{=}[r]\ar[d]^{\wt i}
			&\Nscr_r\ar[d]^i\\
			\gfrak
			&G\times\gfrak\ar[l]^-{\operatorname{pr}_\gfrak}\ar[r]_-\pi
			&G\times^B\gfrak\ar[r]_-p
			&\gfrak.
		}
		\]
		On the regular orbit, equivariant descent gives $\pi_0^\dagger\Ecal_c\cong p_0^\dagger\Ecal_c^\circ.$
		Using smooth base change, we obtain
		\begin{align*}
			\pi^\dagger\Mbb_c
			&\cong
			\pi^\dagger(\wt i)_\dagger\Ecal_c \cong(i_r)_\dagger\pi_0^\dagger\Ecal_c
			\notag\\
			&\cong
			(i_r)_\dagger p_0^\dagger\Ecal_c^\circ \cong \operatorname{pr}_\gfrak^\dagger (i_{0,r})_\dagger\Ecal_c^\circ \cong \operatorname{pr}_\gfrak^\dagger\Lbf_c.\qedhere
		\end{align*}
	\end{proof}
	
	Recall that $SS(\Lbf_c)=\Yfrak_0$. Using the coordinates from \eqref{formula:coordinates-x} and \eqref{formula:coordinates-y}, if $x\in\nfrak$ and $y\in\mfrak$, then $[x,y]\in\bfrak$, and its diagonal entries are
	\[
	x_1y_1,\quad x_2y_2-x_1y_1,\quad\ldots,\quad x_{n-1}y_{n-1}-x_{n-2}y_{n-2},\quad -x_{n-1}y_{n-1}.
	\]
	It follows that $[x,y]\in\nfrak$ is equivalent to $x_i y_i=0$ for $1\leq i\leq n-1.$ 
	Consequently, there is a closed embedding
	\[
	\Yfrak_0\hookrightarrow\gfrak\times\nfrak\times\gfrak, \qquad (x,y)\mapsto(x,[x,y],y).
	\]
	
	Set $\Yfrak:=G\times^B\Yfrak_0$. It fits into the following Cartesian diagram, 
	\begin{align}
		\label{diagram:y-into-cotangent}
		\xymatrix{ \Yfrak_0\ar[r]\ar[d]
			&\gfrak\times\nfrak\times\gfrak\ar[d]\\
			G\times^B\Yfrak_0\ar[r]
			&G\times^B(\gfrak\times\nfrak\times\gfrak)
			\ar[r]^-\sim
			&T^*(G\times^B \gfrak).
		}
	\end{align}
	Let $i_\Yfrak:\Yfrak\hookrightarrow T^*(G\times^B\gfrak)$ be the resulting closed embedding.
	
	For a weight $\lambda$, let $\Lcal_\lambda:=\Ind_B^G\C_\lambda=G\times^B\C_\lambda$ be the corresponding $\ol{G}$-equivariant line bundle on $\Bcal$, and let $\wt{\Lcal}_\lambda$ be its pullback to $T^*(G\times^B \gfrak)$.  With this convention, we have $\omega_\Bcal\cong\Lcal_{2\rho}.$ 
	Give $(i_\Yfrak)_*\Ocal_\Yfrak$ its natural $\ol{G}$-equivariant structure.
	
	Using the shifted embedding $\tau'_{-c}$, we integrate the $\ol{\gfrak}$-action on $\wt{\gr}(\Mbb_c)$, as in Lemma-Construction \ref{lem:const-equivariance}.
	
	\begin{coro}\label{cor:grm}
		With respect to the Hodge filtration, there is a $\ol G\times\C^*$-equivariant isomorphism:
		\begin{align}
			\wt{\gr}(\Mbb_c)
			&\cong
			\wt{\Lcal}_{\lceil\mu_c\rceil-2\rho} \otimes(i_\Yfrak)_*\Ocal_\Yfrak\otimes\C\langle\dim\bfrak\rangle,
			\label{formula:gr-mc}
		\end{align}
	\end{coro}
	
	\begin{proof}
		Proposition \ref{prop:induction-thm} and Lemma \ref{lem:m} give
		\begin{align}
			\wt{\gr}(\Mbb_c) \cong \ind_B^G\wt{\gr}(\Lbf_c).
			\label{iso:ml}
		\end{align}
		%The induction functor on coherent sheaves sends the support $\Yfrak_0\subset T^*\gfrak$ to $G\times^B\Yfrak_0=\Yfrak$, embedded in $T^*(G\times^B \gfrak)$ by \eqref{diagram:y-into-cotangent}. 
		By Lemma \ref{lem:l}, the restriction of the right hand side of \eqref{iso:ml} to the fiber over the standard Borel is $(i_{\Yfrak_0})_*\Ocal_{\Yfrak_0} \otimes\C_{\lceil\mu_c\rceil-2\rho}\otimes\C\langle\dim\bfrak\rangle.$
		Descent therefore identifies the right hand side of \eqref{iso:ml} with $\wt{\Lcal}_{\lceil\mu_c\rceil-2\rho} \otimes(i_\Yfrak)_*\Ocal_\Yfrak\otimes\C\langle\dim\bfrak\rangle$ as a $G\times\C^*$-equivariant sheaf.
		
		It remains to check the extension from $G$- to $\ol{G}$-equivariance.
		By construction, the $\ol{\gfrak}$-action on $\wt{\gr}^{\Hrm}(\Mbb_c)$ is obtained from $\tau'_{-c}$. Its restriction to the fiber over the standard Borel is precisely the $\ol{B}$-action on $\wt{\gr}^{\Hrm}(\Lbf_c)$ used in Lemma \ref{lem:l}. Hence the isomorphism \eqref{iso:ml} intertwines the center action as well, and is therefore $\ol{G}$-equivariant. 
	\end{proof}
	
	\subsection{Cuspidal DG modules}
	\label{sec:cuspidal-dg-modules}
	
	We now mimic \cite[Section 4]{isospec} to compute $\wt{\gr}^{\Hrm}(\Nbb_c)$. 
	
	Put $\Xfrak_{\nfrak,\mfrak}:=G\times^B(\nfrak\times\mfrak),$ and define
	\[
	q_\bfrak:\Xfrak_{\nfrak,\mfrak}\rightarrow G\times^B\bfrak, \qquad (g,x,y)\mapsto(g,[x,y]).
	\]
	Recall the definition of $\Yfrak_0$ from \eqref{eqn: Y0}. Since the monomials $x_1y_1,\ldots,x_{n-1}y_{n-1}$ form a regular sequence, we have 
	\begin{align*}
		\Yfrak\cong \Xfrak_{\nfrak,\mfrak}\times^R_{G\times^B \bfrak} (G\times^B \nfrak).
	\end{align*}
	
	On $\Bcal$, let $\ul{\bfrak}$ and $\ul{\nfrak}$ denote the vector bundles associated with the $B$-modules $\bfrak$ and $\nfrak$. Their total spaces are $G\times^B\bfrak$ and $G\times^B\nfrak$. Let
	\[
	\iota_\bfrak:\Bcal\hookrightarrow G\times^B\bfrak, \qquad \iota_\nfrak:\Bcal\hookrightarrow G\times^B\nfrak
	\]
	be their zero sections, and let $\pi_\bfrak$ and $\pi_\nfrak$ be the projections to $\Bcal$. Then there is a quasi-isomorphism 
	\begin{align}
		(\iota_\nfrak)_*\Ocal_\Bcal
		&\simeq\left(
		\wedge^\bullet\pi_\nfrak^*\ul{\nfrak^*}, \partial_\nfrak \right),
		\label{formula:koszul-0-section}
	\end{align}
	where the differential is contraction with the tautological section. We place the structure sheaf in cohomological degree zero and $\wedge^k$ in degree $-k$.
	
	Let $q_\nfrak:\Yfrak\to G\times^B\nfrak$ be the restriction of $q_\bfrak$. Pull \eqref{formula:koszul-0-section} back to $\Yfrak$ and define $\Ascr:=Lq_\nfrak^*(\iota_\nfrak)_*\Ocal_\Bcal$ and 
	\begin{align}
		\Acal:=Lq_\bfrak^*(\iota_\bfrak)_*\Ocal_\Bcal \simeq \left( \wedge^\bullet \pi_{\Xfrak_{\nfrak,\mfrak}}^*\ul{\bfrak^*}, q_\bfrak^*\partial_\bfrak \right).
		\label{formula:A}
	\end{align}
	
	At a point $(g,x,y)\in\Xfrak_{\nfrak,\mfrak}$, the differential in \eqref{formula:A} is contraction with $[x,y]\in\bfrak$; thus
	\[
	\partial(v_1\wedge\cdots\wedge v_k) =\sum_{i=1}^k(-1)^{i-1}v_i([x,y]) v_1\wedge\cdots\wedge\widehat{v_i}\wedge\cdots\wedge v_k.
	\]
	Both $\Ascr$ and $\Acal$ are DG algebras, with multiplication induced by the wedge product.
	
	In all, the aforementioned spaces can be combined into Figure \eqref{diag:big} with all 4 squares being derived Cartesian. 
	\begin{figure}[ht]
		\centering
		\[
		\begin{tikzcd}
			&& \mathrm{Spec}_{\Yfrak}(\Ascr) && \Bcal
			\\
			& \Yfrak && G\times^B\nfrak
			\\
			\Xfrak_{\nfrak,\mfrak} && G\times^B\bfrak
			&& \Bcal\times T^*\gfrak
			\\
			&&& T^*(G\times^B \gfrak)
			\arrow["j",from=2-2,to=3-1]
			\arrow[from=2-4,to=3-3]
			\arrow["q_\nfrak",from=2-2,to=2-4]
			\arrow["q_\bfrak",from=3-1,to=3-3]
			\arrow[from=1-3,to=2-2]
			\arrow["\iota_\nfrak",from=1-5,to=2-4]
			\arrow[from=1-3,to=1-5]
			\arrow[from=4-4,to=2-4]
			\arrow["\pi_{\Bcal\times T^*\gfrak}",from=3-5,to=1-5]
			\arrow[from=3-5,to=4-4]
			\arrow[bend right,from=2-2,to=4-4,crossing over]
			\arrow[bend right,from=1-3,to=3-5,crossing over]
		\end{tikzcd}
		\]
		\caption{}
		\label{diag:big}
	\end{figure}
	
	Below, $\pi_X$ denotes the projection of a variety $X$ occurring above to $\Bcal$.
	
	\begin{defn}
		The cuspidal DG module of slope $c$ is
		\begin{align}
			\Acal_c
			&:=
			\Acal\otimes \pi_{\Xfrak_{\nfrak,\mfrak}}^* \Lcal_{\lceil\mu_c\rceil},
			\label{formula:a}
		\end{align}
		Similarly, put $\Ascr_c:=\Ascr\otimes\pi_\Yfrak^*\Lcal_{\lceil\mu_c\rceil}.$
	\end{defn}
	Projection formula plus base change implies that 
	\begin{align}
		Rj_*\Ascr_c\simeq\Acal_c.
		\label{formula:push-twisted-ascr}
	\end{align}
	\subsection{Associated graded of $\Nbb_c$}
	Recall our notations for the cotangent correspondence \eqref{diag:cotangent correspondence} in the case of $p:G\times^B \gfrak\to\gfrak$
	\[
	T^*(G\times^B \gfrak) \xleftarrow{\rho_p} \Bcal\times T^*\gfrak \xrightarrow{\varpi_p} T^*\gfrak.
	\]
	\begin{align}
		\label{diagram:cotangent-maps}
		\xymatrix{ \Yfrak\ar[rr]^{i_\Yfrak}\ar[rrd]_{p_\Yfrak}
			&&T^*(G\times^B \gfrak)
			\ar[d]^{\wt{p}}
			&&\Bcal\times T^*\gfrak
			\ar[ll]_{\rho_p} \ar[lld]^{\varpi_p}
			\\
			&&T^*\gfrak.
		}
	\end{align}
	In diagram \eqref{diagram:cotangent-maps}, take $p_\Yfrak$ to be the restriction of
	\[
	a:\Xfrak_{\nfrak,\mfrak}\rightarrow T^*\gfrak\simeq\gfrak\times\gfrak, \qquad (g,x,y)\mapsto(gx,gy),
	\]
	to $\Yfrak$ and the other arrows to be natural projections and closed embeddings.
	
	In particular, $p_\Yfrak =\wt{p}\circ i_\Yfrak =a\circ j.$ 
	
	Endow $\Acal$ with the natural $\ol{G}\times \C^*$-equivariance as a derived subscheme of $\Xfrak_{\nfrak,\mfrak}$ with $\C^*$ acting in the $\mfrak$-direction. Recall the grading shift \eqref{eqn: grading shift}. 
	\begin{thm}\label{thm:gr N Ac}
		With respect to Hodge filtrations, there is a $\ol G\times\C^*$-equivariant isomorphism
		\begin{align}
			\wt{\gr}(\Nbb_c)
			&\cong  Ra_*\Acal_c\langle n-1\rangle,
			\label{iso: gr Nc}
		\end{align}
		In particular, $Ra_*\Acal_c$ is concentrated in cohomological degree zero.
	\end{thm}
	
	\begin{proof}
		Since $p_\dagger\Mbb_c\cong\Nbb_c$, Theorem \ref{thm:laumon-thm} supplies a $G$-equivariant isomorphism
		\begin{align}
			\wt{\gr}(\Nbb_c) \cong R(\varpi_p)_* \left( L\rho_p^* \wt{\gr}(\Mbb_c) \otimes \pi_{\Bcal\times T^*\gfrak}^*\omega_\Bcal \right).
			\label{iso:n-m}
		\end{align}
		This isomorphism is $\ol{G}$-equivariant as we use Lemma-Construction \ref{lem:const-equivariance} for both sides. Moreover, the relative canonical bundle carries $\C^*$-grading $-\dim\Bcal$, as specified in Theorem \ref{thm:laumon-thm}. Corollary \ref{cor:grm} contributes a grading shift of $\dim\bfrak$. Hence the overall grading shift is $\dim\bfrak-\dim\Bcal=n-1.$ 
		
		Now substituting Corollary \ref{cor:grm} into \eqref{iso:n-m} gives
		\begin{align}
			\wt{\gr}(\Nbb_c) \cong R(\varpi_p)_* \Bigl( L\rho_p^* \left( \wt{\Lcal}_{\lceil\mu_c\rceil-2\rho} \otimes(i_\Yfrak)_*\Ocal_\Yfrak \right) \otimes \pi_{\Bcal\times T^*\gfrak}^*\omega_\Bcal \Bigr)\langle n-1\rangle.
			\label{formula:yz}
		\end{align}
		
		If $X$ is smooth and $i_Y:Y\hookrightarrow X$, $i_Z:Z\hookrightarrow X$ are closed embeddings, then projection formula gives 
		\begin{align}
			(i_Y)_*Li_Y^*(i_Z)_*\Ocal_Z
			&\cong
			(i_Y)_*\Ocal_Y \otimes^L_{\Ocal_X} (i_Z)_*\Ocal_Z \cong (i_Z)_*Li_Z^*(i_Y)_*\Ocal_Y.
			\label{formula:closed-embedding}
		\end{align}
		
		Apply \eqref{formula:closed-embedding} to the case when $X=T^*(G\times^B \gfrak)$ with $Y=\Bcal\times T^*\gfrak$ and $Z=\Yfrak$. Using the projection formula and the cancellation $\Lcal_{-2\rho}\otimes\omega_\Bcal\cong\Ocal_\Bcal$, the right hand side of \eqref{formula:yz}, forgetting the grading shift, becomes
		\begin{align}
			&R\wt{p}_*
			R(\rho_p)_* \Bigl( L\rho_p^* \left( \wt{\Lcal}_{\lceil\mu_c\rceil-2\rho} \otimes(i_\Yfrak)_*\Ocal_\Yfrak \right)
			\notag\otimes
			\pi_{\Bcal\times T^*\gfrak}^*\omega_\Bcal \Bigr)
			\notag\\
			\cong& R\wt{p}_* \left( \wt{\Lcal}_{\lceil\mu_c\rceil} \otimes(i_\Yfrak)_*\Ocal_\Yfrak \otimes^L (\rho_p)_* \Ocal_{\Bcal\times T^*\gfrak} \right)
			\notag\\
			\cong& R\wt{p}_* R(i_\Yfrak)_* Li_\Yfrak^* \left( \wt{\Lcal}_{\lceil\mu_c\rceil} \otimes (\rho_p)_* \Ocal_{\Bcal\times T^*\gfrak} \right)
			\notag\\
			\cong& R(p_\Yfrak)_* Li_\Yfrak^* \left( \wt{\Lcal}_{\lceil\mu_c\rceil} \otimes (\rho_p)_* \Ocal_{\Bcal\times T^*\gfrak} \right).
			\label{formula:laumon-formula}
		\end{align}
		
		Tracing through diagram \ref{diag:big} above and using the projection formula give  
		\begin{align*}
			Li_\Yfrak^* \left( \wt{\Lcal}_{\lceil\mu_c\rceil} \otimes (\rho_p)_* \Ocal_{\Bcal\times T^*\gfrak} \right) \cong\Ascr_c.
		\end{align*}
		Thus \eqref{formula:laumon-formula} equals $R(p_\Yfrak)_*\Ascr_c$. Finally, $p_\Yfrak=a\circ j$ and \eqref{formula:push-twisted-ascr} yield
		\begin{align*}
			R(p_\Yfrak)_*\Ascr_c
			&\cong
			Ra_*Rj_*\Ascr_c \cong Ra_*\Acal_c.\qedhere
		\end{align*}
	\end{proof}
	\section{Elliptic Hall algebras and bigraded characters}	
	\subsection{Bigradings}
	Put $A=(\C^*)^2$. On $\gfrak\times\gfrak$ we have the $A$-action 
	\begin{align}
		\label{eq: A action}
		(u,v)\cdot(x,y)=(u^{-1}x,v^{-1}y)
	\end{align}
	so that  {the coordinates on the first and second factors have} $A$-weights $q$ and $t$, respectively. 
	
	We renormalize the Hodge filtration by $F_j^0N_c:=F_{j+n-1}^{\Hrm}N_c$ and $\gr^0N_c:=\gr^{F^0}N_c$. Theorem \ref{thm:gr N Ac} identifies $\gr^0N_c$ with $R\Gamma(\Acal_c)$ with the $\C^*_t$-equivariance. Further, the filtration $F^0$ on $N_c$ induces 
	\begin{align*}
		F^0_j\Lrm_{n/m}^{(m)}
		&:=\left(F^0_jN_{m/n}\otimes(V^*)^{\otimes m}\right)^{\ol G}.
	\end{align*}    
	As for the $\C^*_q$-equivariance, let $x_j$ be an orthonormal basis of $\gfrak$ and in $\Dcal(\gfrak)$ consider 
	\[
	H_{\gfrak}:=\sum_j x_j\partial_{x_j}+\frac{\dim\gfrak}{2}.
	\]
	First of all, $\sum_j x_j\partial_{x_j}$ is the differential of the dilation action $\C^*\times \gfrak\to \gfrak$, $(t,x)\mapsto t^{-1}x$, under which the integrable connection $\Ecal_c$ is twisted equivariant. This action hence preserves the corresponding weight-$0$ VHS and hence the Hodge filtration on its minimal extension. 
	Therefore $H_\gfrak$ acts on $\gr^0N_c$.
	Each $G$-multiplicity space of $N_c$ is a finite-dimensional $\mathfrak{sl}_2$-module with $H_{\gfrak}$ as the semisimple element by \cite[Example 63]{cee}.
	Thus $H_{\gfrak}$ defines a $\C^*$-action on $\gr^0(N_c)$ corresponding to the infinitesimal action of the one-parameter subgroup $\{(u,u^{-1}),u\in \C^*\}\subset A$ up to a shift.

	On the other hand, inside $\Hrm_{n/m}^{(m)}$ take 
	\[\hrm_c=\tfrac12\sum_{a=1}^{m-1}(x_ay_a+y_ax_a).\]
	Then by \cite{cee}, the isomorphism \eqref{iso: cee iso} intertwines the action of $H_\gfrak$ and that of $\hrm_{n/m}$. 
	\subsection{Polynomial characters}
	For every dominant weight $\mu=(\mu_1,\dots,\mu_n)$, let $S_\mu V$ denote the irreducible rational representation of $\GL_n$ labeled by $\mu$. Here we extend the Schur functor notation by determinant twists. In particular, $S_\mu V$ is a  polynomial representation precisely when $\mu_n\geq0$. Let $\chi_\mu\in \C[z_1^\pm,\dots,z_n^\pm]^{S_n}$ be the character of $S_\mu V$.  Set
	\begin{align*}
		\left(\sum_\mu a_\mu\chi_\mu\right)_{\pol} :=\sum_{\mu_n\geq0}a_\mu s_\mu
	\end{align*}
	where $s_\mu$ denotes the Schur function labeled by $\mu$.
	
	For a finite-dimensional $A$-module $E=\bigoplus_{i,j}E_{i,j}$, put 
	\begin{align*}
		\ch_A(E)&:=\sum_{i,j}q^it^j\dim E_{i,j}.
	\end{align*}
	For an $A\times S_m$-module $K$, its bigraded Frobenius characteristic is
	\begin{align*}
		\ch_{S_m\times A}(K)=\sum_{\lambda\vdash m} \ch_A\Hom_{S_m}(V_\lambda,K)\,s_\lambda,
	\end{align*}
	where $V_\lambda$ is the Specht module indexed by $\lambda$.
	
	\begin{lem}\label{lem: char Lc as pol}
		With respect to the $\hrm_{n/m}$, resp. $H_\gfrak$, and filtration gradings, 
		\begin{align*}
			{\ch_{S_m\times A}(\gr^0\Lrm_{n/m}^{(m)})} =\left(\ch_{A\times\ol G}\gr^0N_{m/n}\right)_{\pol}.
		\end{align*}
	\end{lem}
	\begin{proof}
		Recall that $V:=\C^n$ is the standard representation of $\GL_n$. By Theorem \ref{thm:cee}, $\Lrm_{n/m}^{(m)}\cong (N_{m/n}\otimes(V^*)^{\otimes m})^{\SL_n}.$ 
		With respect to Lemma-Construction \ref{lem:const-equivariance}, the center of $\GL_n$ acts on $N_{m/n}$ by the character $z\mapsto z^m$. Thus setting $E_\mu=\Hom_{\GL_n}(S_\mu V,N_{m/n})$, we have 
		\begin{align*}
			N_{m/n}=\bigoplus_{\substack{\mu_1\geq\cdots\geq\mu_n\\
					\mu_1+\cdots+\mu_n=m}} E_\mu\otimes S_\mu V,
		\end{align*}
		both as $\GL_n$-modules and $\SL_n$-modules. 
		
		Moreover, Schur-Weyl duality gives
		\begin{align*}
			(V^*)^{\otimes m} =\bigoplus_{\substack{\lambda\vdash m\\\ell(\lambda)\leq n}} (S_\lambda V)^*\otimes V_\lambda,
		\end{align*}
		as $\GL_n\times S_m$-modules. Therefore,
		\[\Lrm_{n/m}^{(m)}\cong (N_{m/n}\otimes(V^*)^{\otimes m})^{\SL_n}\cong \bigoplus_{\substack{\mu_1\geq\cdots\geq\mu_n\\
				\mu\vdash m}} E_\mu\otimes V_\mu. \]
		Taking $\ch_{S_m\times A}$ of left and right hand sides, we obtain the lemma.  
	\end{proof}
	
	\subsection{Character of cuspidal DG modules}
	For $f\in \C(q,t)(z_1,\dots,z_r)$, define
	\[\Sym_r(f)=\sum_{w\in S_r} w(f).\]
	%Thus $d_i=b_{n+1-i}$ and both sequences have sum $m$.
	For an integer sequence $b=(b_1,\dots,b_n)$, write $z^b:=z_1^{b_1}\cdots z_n^{b_n}$. Set
	\begin{align}
		\Gamma(x)=\frac{1-qtx}{(1-qx)(1-tx)(1-x^{-1})}.
		\label{def: Gamma}
	\end{align}
	
	\begin{prop}
		\label{prop: ch Nc}
		Put \begin{align*}
			\Csf_c
			=\left(\frac{1-qt}{1-t}\right)^{n-1}
			\Sym_n\left(
			\frac{z^{\lceil\mu_c\rceil}}{\prod_{i=1}^{n-1}(1-tz_i/z_{i+1})}
			\prod_{i<j}\Gamma(z_j/z_i)\right).
		\end{align*}
		With respect to $H_\gfrak$ and filtration gradings, there exists some integer $a$ such that \[\ch_{A\times\ol G}(\gr^0N_c)=q^a\Csf_c(z,q,q^{-1}t).\] 
	\end{prop}
	
	\begin{proof}
		Write $\ch:=\ch_{A\times \ol{G}}$ for simplification. By Theorem \ref{thm:gr N Ac}, it suffices to compute $\ch R\Gamma(\Acal_c)$. The fiber over the fixed Borel $\bfrak$ is the Koszul complex 
		\begin{align}
			K_c^\bullet:=\C_{\lceil\mu_c\rceil}\otimes\Sym(\nfrak^*\oplus\mfrak^*) \otimes\wedge^\bullet\bfrak^*.
			\label{eqn: Ae}
		\end{align}
		Since $\nfrak$ is strictly upper triangular, the weights of its dual are $z_j/z_i$, for $i<j$. The Hessenberg space $\mfrak$ is the sum of $\bfrak$ and the negative simple-root spaces. Thus the $A\times\ol T$ weights of the three factors are, respectively,
		\begin{align*}
			q\nfrak^*&:\quad qz_j/z_i &(i<j),\\
			t\mfrak^*&:\quad t\ \ (n-1\text{ copies}),\quad
			tz_j/z_i\ \ (i<j),\quad tz_i/z_{i+1}\ \ &(i<n),\\
			qt\bfrak^*&:\quad qt\ \ (n-1\text{ copies}),\quad qtz_j/z_i\ \ &(i<j).
		\end{align*}
		For a representation $E$ with weights $w_1,\ldots,w_s$, the character of $\Sym(E)$ is $\prod_a(1-w_a)^{-1}$, and the alternating character of $\wedge^\bullet E$ is $\prod_a(1-w_a)$. Therefore the character of \eqref{eqn: Ae} is
		\begin{align}
			\left(\frac{1-qt}{1-t}\right)^{n-1} \frac{z^{\lceil\mu_c\rceil}\prod_{i<j}(1-qtz_j/z_i)} {\prod_{i<j}(1-qz_j/z_i)(1-tz_j/z_i) \prod_{i=1}^{n-1}(1-tz_i/z_{i+1})}.
			\label{eq:ch Koszul}
		\end{align}
		
		For our upper-triangular Borel and the convention $\Lcal_\lambda=\ol G\times^{\ol B}\C_\lambda$, we have
		\begin{align}
			\ch R\Gamma(\ol G/\ol B,\ol G\times^{\ol B}K_c^\bullet) =\Sym_n\left( \frac{\ch K_c^\bullet}{\prod_{i<j}(1-z_i/z_j)}\right).
			\label{eq:ind}
		\end{align}
		Indeed, additivity reduces the assertion to a one-dimensional weight. 
		Borel-Weil-Bott and the Weyl character formula give \eqref{eq:ind}. Equivalently, the cotangent weights at the standard flag are $z_i/z_j$, for $i<j$. Applying \eqref{eq:ind} to \eqref{eq:ch Koszul} proves the proposition.
	\end{proof}
	
	Following \cite{negutflag}, let $\Xfrak_{\nfrak,\nfrak}:=G\times^B(\nfrak\times\nfrak)$ with a commutator map to $G\times^B[\nfrak,\nfrak]$. Let $i:\Bcal\to G\times^B [\nfrak,\nfrak]$ be the zero section.
	Put
	\[		\lfloor\mu_c\rfloor
	:=\bigl(\lfloor c\rfloor,\lfloor2c\rfloor-\lfloor c\rfloor,\ldots,
	\lfloor nc\rfloor-\lfloor(n-1)c\rfloor\bigr)\] and take the DG module 
	\[\Acal_c':=\Ocal_{\Xfrak_{\nfrak,\nfrak}}
	\otimes^L_{q^{-1}\Ocal_{G\times^B [\nfrak,\nfrak]}}
	q^{-1}Ri_*\Lcal_{\lfloor\mu_c\rfloor}.\]Define 
	\begin{align}
		\label{eqn: DG n n}
		\Hcal_c:=\Sym_n\left( \frac{z^{\lfloor\mu_c\rfloor}}{\prod_{i=1}^{n-1}(1-qtz_{i+1}/z_{i})} \prod_{i<j}\Gamma(z_j/z_i)\right).
	\end{align}
	Then analogous to Proposition \ref{prop: ch Nc}, we have $\ch_{A\times \ol{G}} R\Gamma(\Xfrak_{\nfrak,\nfrak},\Acal_c')=\Hcal_c$.   
	
	\begin{prop} \cite[(2.34) and (2.35)]{negutwalgebra}
		\label{prop: cuspidal catalan}
		For $\gcd(m,n)=1$, one has
		\[\Csf_c=q^{n-1}\Hcal_c.\]
	\end{prop}
	
	\subsection{Polynomial representation of elliptic Hall algebra}\label{section: EHA}
	Take the ring of symmetric functions  $\Lambda$ over $k:=\C(q,t)$.
	
	We use the positive part $\Ecal^+$ of the elliptic Hall algebra in the conventions of \cite[Sections 3.2-3.5]{BHMPS}. For each coprime pair $(u,v)$ with $u>0$, define the alphabet $X^{u,v}=(x^{u,v}_1,x^{u,v}_2, \dots)$. For every such $(u,v)$, the algebra $\Ecal^+$ contains a copy of $\Lambda(X^{u,v})$ and is generated by them as an algebra. 
	For coprime  {$u>0$, $v\in\Z$}, define 
	\[Q_{u,v}:=p_1[-(1-q)(1-t)X^{u,v}] \in \Lambda(X^{u,v}),\]
	where $p_1=\sum z_i$. Here and below $[-]$ denotes the plethystic substitution. 
	
	For every $\eta\in k(z_1,\ldots,z_r)$, put
	\begin{align*}
		\mathscr H_r(\eta)=\Sym_r\left( \eta(z)\prod_{i<j}\Gamma(z_i/z_j)\right).
	\end{align*}
	Note the difference with \eqref{eqn: DG n n}. This choice is compatible with Proposition \ref{prop: ch C psi} later. Take $\Ker(\mathscr H_r)$ inside $k[z_1^{\pm1},\ldots,z_r^{\pm1}]$. Theorem 10.1 in \cite{SV_Hilbert} is an algebra isomorphism 
	\begin{align}
		\psi: \mathcal S:=\bigoplus_{r\geq 0} (k[z_1^{\pm1},\ldots,z_r^{\pm1}]/\ker\mathscr H_r) \xrightarrow{\sim}\Ecal^+, \qquad  {[z_1^a]}\mapsto  Q_{1,a},
		\label{iso:psi EHA}
	\end{align}
	The product on $\Scal$ is via concatenation: $(f\cdot g)(x_1,\dots, x_{k+\ell})=f(x_1,\dots, x_k)g(x_{k+1},\dots, x_{k+\ell})$. In particular, $\Ecal^+$ is generated by $\{Q_{1,a}|a\in \Z\}$ as an algebra. Here we use the reformulation in \cite[Proposition 3.5.1]{BHMPS}.
	
	For a sequence $b=(b_1,\dots,b_n)$, though the rational function $\phi_b(z)=\frac{z^b}{\prod_{i=1}^{n-1}(1-qtz_i/z_{i+1})}$ does not lie inside $k[z_1^\pm,\dots,z_n^\pm]$, but according to \cite[Section 3.6]{BHMPS}, there exists a $\eta_b\in k[z_1^\pm,\dots,z_n^\pm]$ such that
	\begin{align}
		\label{eqn: eta}
		\mathscr H_n(\eta_b)=\mathscr H_n(\phi_b).
	\end{align}
	Let $w_0$ be the longest element in $S_n$. Since  $\lceil\mu_c\rceil=w_0\lfloor\mu_c\rfloor$, we have $\Hcal_c=\mathscr H_n(\phi_{\lceil\mu_c\rceil})$. Define $D_{\lceil\mu_c\rceil}=\psi([\eta_{\lceil\mu_c\rceil}])$. 
	%The isomorphism \eqref{iso:psi EHA} makes this definition independent of the choice of $\eta_{\lceil\mu_c\rceil}$. 
	
	By the proof of Theorem 3.1 in \cite{SV_Hilbert} (cf. \cite[Propositions 3.3.4]{BHMPS}), the algebra $\Ecal^+$ acts on $\Lambda$ by the vertex formula
	\begin{align}
		Q_{1,a}F[X]=(-1)^a\mathrm{Coeff}_{u^a}\bigl(F[X+(1-q)(1-t)/u]\Omega[-uX]\bigr), \qquad \Omega[X]=\exp\left(\sum_{j\geq1}\frac{p_j[X]}j\right).
		\label{eq:vertex}
	\end{align}
	Here $\mathrm{Coeff}_{u^a}$ denotes taking the coefficient of $u^a$. 
	\begin{prop}
		\label{prop:Hb pol Q}
		We have an equality $(\Hcal_c)_{\pol}=\omega(Q_{n,m}\cdot 1).$ 
		
		Here $\omega$ is the involution of $\Lambda$ sending $s_\lambda$ to $s_{\lambda^t}$ where $\lambda^t$ denotes the transpose. 
	\end{prop}
	
	\begin{proof}
		Proposition 3.6.1 of \cite{BHMPS} states that  $D_{\lceil\mu_c\rceil}=Q_{n,m}$.
		By \cite[Corollary 3.7.2]{BHMPS}, one has $(\Hcal_c)_{\pol}=\omega(D_{\lceil\mu_c\rceil}1)$.
	\end{proof}
	
	\subsection{Bigraded characters}
	\begin{thm}(\cite[Theorem I.3]{glwx})\label{lem:Q specialization}
		For coprime positive integers $m,n$, one has
		\begin{align}
			Q_{m,n}1\big|_{t=q^{-1}} =\frac{q^{-(m-1)(n-1)/2}}{[m]_q}\,e_n[X[m]_q], \qquad \text{ where } [m]_q:=1+q+\cdots+q^{m-1}.
			\label{formula:pmn-t-inverse-q}
		\end{align}
		Moreover, the highest $q$-degree of $Q_{m,n}1\big|_{t=q^{-1}}$ is $(m-1)(n-1)/2$. 
	\end{thm}
	
	\begin{proof}
		\eqref{formula:pmn-t-inverse-q} is \cite[Theorem I.3]{glwx}. For the second assertion, note that the $q$-degree of $e_n[X[m]_q]$ is $n(m-1)$ while the $q$-degree of $[m]_q$ is $m-1$. Altogether we have the $q$-degree of $Q_{m,n}1\big|_{t=q^{-1}}$ is \[n(m-1)-(m-1)-(m-1)(n-1)/2=(m-1)(n-1)/2.\qedhere\]
	\end{proof}
	\begin{thm}\label{thm:ch Lc Q}
		With respect to $\hrm_{n/m}$ and filtration gradings, we have 
		\[\ch_{S_m\times A}(\gr^0\Lrm_{n/m}^{(m)})
		=\omega(Q_{n,m}\cdot 1)(z;q,q^{-1}t).\]
	\end{thm}
	\begin{proof}
		By Proposition \ref{prop: ch Nc},  {Proposition \ref{prop: cuspidal catalan},} identity Lemma \ref{lem: char Lc as pol} and Proposition \ref{prop:Hb pol Q}, we have \[\ch_{S_m\times A}(\gr^0\Lrm_{n/m}^{(m)})
		=q^a\omega(Q_{n,m}\cdot 1)(z;q,q^{-1}t)\] 
		for some $a\in \Z$. To check $a=0$, it suffices to compare both sides when setting $t=1$: the highest $q$-degree of the right hand side upon specialization is $(m-1)(n-1)/2$ according to Theorem \ref{lem:Q specialization} which equals the highest $\hrm_{n/m}$-weight of $\Lrm_{n/m}^{(m)}$ by \cite[Theorem 1.6]{beg}. 
	\end{proof}

	\subsection{Khovanov-Rozansky link homology}
	Throughout this subsection, $m,n>0$ are coprime and $g=(m-1)(n-1)/2$. Let $T_{m,n}$ be the $(m,n)$ torus knot. 
	Let $\Pcal^{\mathrm{Mel}}_{m,n}(a,q,t)$ denote its reduced homology invariant under the conventions of \cite{mellitknothomology}.
	\begin{thm}[{\cite[Corollary 3.5]{mellitknothomology}}]
		\label{thm:mellit}
		One has
		\begin{align*}
			\Pcal^{\mathrm{Mel}}_{m,n}(a,q,t) = \left(a(qt)^{-1/2}\right)^g \sum_{i=0}^{n-1}(-a)^i \bigl(Q_{m,n}1,s_{(i+1,1^{n-i-1})}\bigr).
		\end{align*}
	\end{thm}
	
	Let $\Hrm\Hrm\Hrm(T_{m,n})$ denote the reduced triply graded Khovanov-Rozansky homology of the $(m,n)$ torus knot with grading convention as in \cite[Section 1.5]{gors}. 
	Write $\Lrm_{m/n}=\Lrm_{m/n}^{(n)}$, and let $\Lrm_{m/n}(j)$ denote its $\hrm_{m/n}$-eigenspace of weight $j$.
	Define the re-indexed Hodge filtration by
	\begin{align*}
		F_\ell\Lrm_{m/n} := \bigoplus_{j\in\mathbb Z} F^0_{\lfloor(\ell-j)/2\rfloor}\Lrm_{m/n}(j).
	\end{align*}
	
	\begin{thm}\label{thm:HHH Lc}
		For $0\leq i\leq n-1$ and $j,\ell\in\mathbb Z$, there exists an isomorphism of vector spaces
		\begin{align*}
			\mathrm{HHH}^{2g+2i,\,2j,\,\ell-j-g}(T_{m,n}) \cong \Hom_{S_n}\left( \wedge^i\hfrak,\gr^F_\ell\Lrm_{m/n}(j) \right).
		\end{align*}
		Moreover, $F$ is preserved by $\hrm_{m/n}$ and is compatible with the filtration on $\Hrm_{m/n}$ in which $\hfrak$ and $\hfrak^*$ have degree one and $S_n$ has degree zero.
	\end{thm}
	
	\begin{proof}
		Set $d_{i,j,k} := \dim\Hom_{S_n}\left( \wedge^i\hfrak,\gr^0_k\Lrm_{m/n}(j) \right).$ 
		By Theorem \ref{thm:ch Lc Q}, with $m,n$ interchanged,
		\begin{align*}
			\operatorname{ch}_{S_n}\bigl(\gr^0\Lrm_{m/n}\bigr) = \omega(Q_{m,n}1)(q,q^{-1}t).
		\end{align*}
		The Frobenius characteristic of $\wedge^i\hfrak$ is $s_{(n-i,1^i)}$, and its image under $\omega$ is $s_{(i+1,1^{n-i-1})}$. Taking the corresponding Schur coefficient and replacing $t$ by $qt$ gives
		\begin{align*}
			\bigl(Q_{m,n}1,s_{(i+1,1^{n-i-1})}\bigr) = \sum_{j,k}d_{i,j,k}q^{j+k}t^k.
		\end{align*}
		Consequently,
		\begin{align*}
			\Pcal_{m,n}(\alpha,q,t):= \sum \alpha^i\bigl(Q_{m,n}1,s_{(i+1,1^{n-i-1})}\bigr) = \sum_{i,j,k}d_{i,j,k}a^iq^{j+k}t^k.
		\end{align*}
		
		Use $A,Q,H$ to record the three gradings in the conventions of \cite[Section 1.5]{gors}.  Then the convention comparison between \cite{mellitknothomology} and \cite{gors} is that $a=-A^2$, $q=Q^2$, $t=Q^{-2}H^2$, and $(qt)^{1/2}= -H$. Therefore Theorem \ref{thm:mellit} gives
		\begin{align*}
			\sum_{\alpha,\beta,\gamma} A^\alpha Q^\beta H^\gamma \dim\mathrm{HHH}^{\alpha,\beta,\gamma}(T_{m,n})= A^{2g}H^{-g} \Pcal_{m,n}(A^2,Q^2,Q^{-2}H^2).
		\end{align*}
		Comparing coefficients yields
		\begin{align*}
			\dim\mathrm{HHH}^{2g+2i,\,2j,\,2k-g}(T_{m,n}) =d_{i,j,k}.
		\end{align*}
		
		By the definition of $F$,
		\begin{align*}
			\gr^F_\ell\Lrm_{m/n}(j) \cong
			\begin{cases}
				\gr^0_{(\ell-j)/2}\Lrm_{m/n}(j),
				&\ell-j\in2\mathbb Z,\\
				0,&\ell-j\notin2\mathbb Z.
			\end{cases}
		\end{align*}
		Taking $\ell=j+2k$ for an integer $k$ gives the desired equality of dimensions. When $\ell-j$ is odd, both sides vanish.

		Finally, by the compatibility of $F^0$ with the order filtration, we have that, for $x\in\hfrak^*$ and $y\in\hfrak$,
		\begin{align*}
			xF^0_k\Lrm_{m/n}(j) \subseteq F^0_k\Lrm_{m/n}(j+1), \qquad yF^0_k\Lrm_{m/n}(j) \subseteq F^0_{k+1}\Lrm_{m/n}(j-1).
		\end{align*}
		Hence
		\begin{align*}
			xF_\ell\Lrm_{m/n}\subseteq F_{\ell+1}\Lrm_{m/n}, \qquad yF_\ell\Lrm_{m/n}\subseteq F_{\ell+1}\Lrm_{m/n}.
		\end{align*}
		The $S_n$-action preserves $F$. Since $F$ is defined as a direct sum of subspaces of $\mathrm h_{m/n}$-eigenspaces, it is also preserved by $\mathrm h_{m/n}$. 
	\end{proof}
	
	\section{Localization on Hilbert schemes of points}
	\subsection{Hilbert scheme of points}
	Let $\ol{\Hilb}^n$ be the Hilbert scheme of ideals of colength $n$ in $\C[x,y]$.  It is a smooth quasi-projective variety of dimension $2n$.
	The Hilbert-Chow morphism
	\begin{align*}
		\xi:\ol{\Hilb}^n\rightarrow \Sym^n(\C^2)=(\C^2)^n/S_n
	\end{align*}
	sends an ideal to the cycle associated with its support, counted with multiplicities.
	
	As $(\C^2)^n\cong \ol{\hfrak}^2$, we can take $\Hilb^n:=\xi^{-1}(\hfrak^2/S_n)$ . 
	Translation by the center of mass gives an isomorphism $\ol{\Hilb}^n\cong \Hilb^n\times\C^2.$ 
	%In particular, $\Hilb^n$ is smooth of dimension $2n-2$.
	
	%The centered Hilbert scheme also has the following ADHM presentation.  
	Set
	\begin{align*}
		\wt{\Hilb^n} :=\bigl\{(X,Y,v)\in\gfrak\times\gfrak\times V \mathrel{\big|}[X,Y]=0,\ \C[X,Y]v=V\bigr\}_{\red}.
	\end{align*}
	The diagonal action of $\ol G$ on $\wt{\Hilb^n}$ is free, and the geometric quotient $\wt{\Hilb^n}/\ol G$ is isomorphic to $\Hilb^n$ \cite{nakajima}, which inherits an $A$-action from \eqref{eq: A action}. 
	Let
	\begin{align*}
		\wt{\Hilb_0^n}:=(\Nscr\times\Nscr\times V)\cap\wt{\Hilb^n}.
	\end{align*}
	This is an open subvariety of \eqref{formula:u}.  Its quotient $\Hilb_0^n:=\wt{\Hilb_0^n}/\ol G$ is the punctual Hilbert scheme, namely the fiber of $\xi$ over the cycle $n(0,0)$. 
	
	\subsection{The BKR equivalence for $\Hilb^n$}
	The isospectral Hilbert scheme is the reduced fiber product
	\begin{align*}
		\IHilb^n:=\bigl(\Hilb^n\times_{\hfrak^2/S_n}\hfrak^2\bigr)_{\red}.
	\end{align*}
	\begin{align*}
		\xymatrix{
			\IHilb^n\ar[r]^-\beta\ar[d]_\alpha&\hfrak^2\ar[d]\\
			\Hilb^n\ar[r]&\hfrak^2/S_n. }
	\end{align*}
	In \cite{haiman_macdonald}, Haiman shows that $\alpha$ is finite and flat of degree $n!$ .  Consequently, $\Pcal:=\alpha_*\Ocal_{\IHilb^n}$ is an $S_n\times A$-equivariant vector bundle of rank $n!$, known as the Procesi bundle.  
	Hence, as a result of Bridgeland-King-Reid \cite{BKR}, the functor 
	\begin{align}
		\Phi: D_A^b\Coh(\Hilb^n)&\xrightarrow{\ \sim\ } D_{S_n\times A}^b\Coh(\hfrak^2),
		\label{iso:bkr}\\
		\Fcal&\mapsto R\beta_*L\alpha^*\Fcal\cong R\Gamma\bigl(\Hilb^n,\Pcal\otimes^L\Fcal\bigr).\nonumber
	\end{align}
	is an equivalence.
	
	We compose \eqref{iso:bkr} with the bigraded Frobenius characteristic. Localized over $\C(q,t)$, $\bigoplus_{n\geq 0} K^{S_n\times A}(\hfrak^2)$ can be identified with $ \Lambda$. Therefore, we have an isomorphism
	\begin{align}
		\kappa: \bigoplus_{n\geq0} K^A(\Hilb^n)_\loc
		&\xrightarrow{\ \sim\ }\Lambda,
		\label{iso: K Lambda}\\
		[\Fcal]&\mapsto \ch_{S_n\times A} R\Gamma\bigl(\Hilb^n,\Pcal\otimes^L\Fcal\bigr).\nonumber
	\end{align}
	
	\subsection{Geometric realization of $\Ecal^+$}
	Let 
	\begin{align*}
		\Cfrak:=\{(x,y)\in\gfrak^2\mid[x,y]=0\}_{\red}. 
	\end{align*}
	Similarly define the commuting  variety $\ol{\Cfrak}$ inside $\ol{\gfrak}^2$. Let 
	\begin{align*}
		p:\wt{\Hilb^n}\rightarrow\Cfrak, \qquad \varpi:\wt{\Hilb^n}\rightarrow\Hilb^n
	\end{align*}
	be the maps of forgetting the cyclic vector and taking the quotient respectively.

	%This construction agrees with restriction to the cyclic locus and descent of $\wt{\gr}^{\Hrm}(\ol{\Nbb}_c)\langle1-n\rangle$

	The direct sum of $K$-group $\bigoplus_{r\geq0}K^{\GL_r\times A}(\ol{\Cfrak}_r)_\loc$ carries the Hall convolution product $*$ defined by parabolic induction \cite[Sections 7.2-7.5]{SV_Hilbert}.
	Inside it, consider the subalgebra $K$ generated by $\{z^a[\Ocal_{\ol{\Cfrak}_1}],a\in \Z\}$ where $z$ is the $\GL_1$-equivariant parameter. Then \cite[Theorem 7.14 and Corollary 7.15]{SV_Hilbert} give an algebra isomorphism 
	\begin{align*}
		C:\Ecal^{+,\mathrm{op}}\xrightarrow{\sim}K , \qquad C(Q_{1,a})=(1-q)(1-t)z^a[\Ocal_{\ol{\Cfrak}_1}]. 
	\end{align*}

	Let $\Ecal$ be the full elliptic Hall algebra, generated by $Q_{u,v}=p_1[-MX^{u,v}]$ for coprime pairs $(u,v)\neq (0,0)$. Let $\alpha:\Ecal\to\End_k(\Lambda)$ be its polynomial representation, whose restriction to $\Ecal^+$ is defined by the vertex formula \eqref{eq:vertex}.
	We have an antihomomorphism of $\Ecal$ restricted to
	\begin{align*}
		\tau:\Ecal^{+,\mathrm{op}}\rightarrow\Ecal, \qquad Q_{1,a}\mapsto Q_{a,1} \quad \text{ for } a\in\Z.
	\end{align*}
	
	Denote by $\rho$ the action of $K$ on $\Lambda$ defined in \cite[Proposition 7.9]{SV_Hilbert} given by convolution. 
	Define $\nabla$ to be the operator on $\Lambda$ defined by $\nabla\wt H_\lambda=q^{n(\lambda^t)}t^{n(\lambda)}\wt H_\lambda$, where $n(\lambda)=\sum_i(i-1)\lambda_i$.
	\begin{prop}\label{prop:actions on Lambda}
		For $Q\in\Ecal^+$,
		\begin{align*}
			\rho(C(Q))
			=\nabla\,\alpha(\tau(Q))\,\nabla^{-1}\end{align*}
		Consequently, for coprime positive $m,n$,
		\begin{align*}
			\rho(C( Q_{n,m}))1 =\nabla(Q_{m,n}1)=Q_{m+n,n}1.
		\end{align*}
	\end{prop}
	\begin{proof}
		For the first equality, both sides are algebra homomorphisms from $\Ecal^{+,\mathrm{op}}$ to $\End_k(\Lambda)$.
		Since the elements $Q_{1,a}$, $a\in\Z$, generate $\Ecal^{+,\mathrm{op}}$, it suffices to compare their images.
		
		Put $\theta_a=z^a[\Ocal_{\ol{\Cfrak}_1}]$. Let $\Vcal_r$ be the rank-$r$ tautological vector bundle on $\ol{\Hilb}^r$ whose fiber at an ideal $I$ is $\C[x,y]/I$. 
		Let $Z_{r+1,r}$ be the nested Hilbert scheme parametrizing ideals $I_{r+1}\subset I_r$ of colengths $r+1$ and $r$.
		Let $p_+$ and $p_-$ be its projections to $\ol{\Hilb}^{r+1}$ and $\ol{\Hilb}^r$, respectively. The nested tautological line bundle $\Lcal$ is defined by
		\begin{align*}
			0\rightarrow\Lcal \rightarrow p_+^*\Vcal_{r+1} \rightarrow p_-^*\Vcal_r \rightarrow0.
		\end{align*}
		Its fiber at $(I_{r+1},I_r)$ is $I_r/I_{r+1}$.
		
		Proposition 7.11 in \cite{SV_Hilbert} shows that, for $\xi\in K^A(\ol{\Hilb}^r)_{\mathrm{loc}}$,
		\begin{align*}
			\rho(\theta_a)(\xi) =R(p_+)_*\bigl([\Lcal^{\otimes a}] \otimes^L Lp_-^*\xi\bigr) =:T_a(\xi).
		\end{align*}
		Moreover, \cite[Lemma 4.3 and Corollary 4.4]{SV_Hilbert} give
		\begin{align*}
			(1-q)(1-t)T_a=\nabla^{a+1}p_1^\bullet\nabla^{-(a+1)}
		\end{align*}
		where $p_1^\bullet$ denotes multiplication by $p_1$. Indeed, $(1-q)(1-t)T_{-1}=p_1^\bullet$, and increasing the tautological-line exponent by one conjugates the operator by $\nabla$.
		
		On the other hand, by \cite[Proposition 3.3.1 and Lemma 3.4.1]{BHMPS},
		\begin{align*}
			\alpha(Q_{a,1})=\nabla^a p_1^\bullet\nabla^{-a}.
		\end{align*}
		Consequently,
		\begin{align*}
			\rho(C(Q_{1,a})) =(1-q)(1-t)\rho(\theta_a)=(1-q)(1-t)T_a =\nabla^{a+1}p_1^\bullet\nabla^{-(a+1)}=\nabla\,\alpha(\tau(Q_{1,a}))\,\nabla^{-1}.
		\end{align*}
		
		The second statement follows from $\nabla^{-1}1=1$ and $ {\nabla\alpha(Q_{m,n})\nabla^{-1}=\alpha(Q_{m+n,n})}$.
	\end{proof}
	Recall the isomorphism $\psi$ from \eqref{iso:psi EHA}. Let $\Hscr:=\bigoplus_r \Hscr_r$.
	\begin{prop}\label{prop: ch C psi}
		For $\eta\in \C(q,t)[z_1^\pm,\dots,z_r^\pm]$, we have $\ch_{\GL_r\times A}( {C(\psi([\eta]))})=\Hscr(\eta)$.
	\end{prop}
	\begin{proof}
		By linearity, it suffices to consider $\eta=z_1^{a_1}\cdots z_r^{a_r}$. Recall that $\theta_{a_i}=z^{a_i}[\Ocal_{\ol{\Cfrak}_1}]$. 
		By the definitions of $\psi$ and $C$,
		\begin{align*}
			C(\psi([\eta])) =(1-q)^r(1-t)^r\theta_{a_r}*\cdots*\theta_{a_1}.
		\end{align*}
		
		Take $\ol G=\GL_r$ with $\ol B$, $\ol\bfrak$, $\ol{T}$ and $\nfrak$ as before. Put $\Xfrak_{\ol{\bfrak},\ol{\bfrak}}:=\ol G\times^{\ol B}(\ol\bfrak\times\ol\bfrak)$, with commutator map $q$ to $\ol G\times^{\ol B}\nfrak$, and $i$ is the zero section. For a character $\ell$ of $\ol{T}$, recall the line bundle $\Lcal_\ell=\ol G\times^{\ol B}\C_\ell$ on $\ol G/\ol B$. 
		By \cite[Section 7.10 and Proposition 7.16(a)]{SV_Hilbert}, $\theta_{\ell_1}*\cdots*\theta_{\ell_r}$ is represented by the pushforward to $\ol{\Cfrak}_r$ of 
		\begin{align}
			\label{eqn: DG b b}
			\Ocal_{\Xfrak_{\ol{\bfrak},\ol{\bfrak}}} \otimes^L_{q^{-1}\Ocal_{\ol{G}\times^{\ol{B}} \nfrak}} q^{-1}Ri_*\Lcal_\ell
		\end{align}
		where $\ell=(\ell_1,\dots,\ell_r)$.
		
		Similar computation as Proposition \ref{prop: ch Nc} gives 
		\begin{align*}
			\ch_{\GL_r\times A}(\theta_{\ell_1}*\cdots*\theta_{\ell_r})
			%&=M^{-r}\Sym_r\left[z^\ell\prod_{i<j}\frac{1-qtz_j/z_i}{(1-qz_j/z_i)(1-tz_j/z_i)(1-z_i/z_j)}\right]\\
			&=
			(1-q)^{-r}(1-t)^{-r}\Sym_r\left[ z^\ell\prod_{i<j}\Gamma(z_j/z_i) \right].
		\end{align*}
		Applying this formula with $\ell_i=a_{r+1-i}$ yields
		\begin{align*}
			\ch_{\GL_r\times A}(C(\psi([\eta])))
			&=
			\Sym_r\left[ z_1^{a_r}\cdots z_r^{a_1} \prod_{i<j}\Gamma(z_j/z_i)
			\right]\\
			&=
			\Sym_r\left[ z_1^{a_1}\cdots z_r^{a_r} \prod_{i<j}\Gamma(z_i/z_j) \right] =\mathscr H_r(\eta),
		\end{align*}
		where the second equality reverses the variables.
	\end{proof}
	\subsection{K-theory of the cuspidal sheaf}
	By Lemma \ref{lem:mirabolic-cuspidal}, $\wt{\gr}^0(\Nbb_c)$ is set-theoretically supported on $\Cfrak_\nil$. Furthermore, 
	\begin{prop}
		\label{prop: support gr Nc}
		The coherent sheaf $\wt{\gr}^0(\Nbb_c)$ is scheme-theoretically supported on $\Cfrak_\nil$. Moreover, $F_{-1}^0N_c=0$. 
	\end{prop}
	\begin{proof}
		
		Take $R=\C[T^*\gfrak]$ and $M=\wt{\gr}^0 N_c$. By Lemma \ref{lem:mirabolic-cuspidal}, $\Supp_R(M)=\Cfrak_\nil$ for $\Cfrak_\nil = \{(x,y)\in\Nscr\times\Nscr\mid[x,y]=0\}_{\red}.$
		By \cite{Baranovsky}, $\Cfrak_\nil$ is irreducible. Let $\pfrak\subset R$ be the prime ideal defining $\Cfrak_\nil$.
		
		Let $i:\Nscr_r\hookrightarrow\gfrak_r$ be the closed embedding, and put $\Nbb_c^\circ=\Nbb_c|_{\gfrak_r}$. Then $\Nbb_c^\circ=i_\dagger\Ecal_c$. Write $P_k=\tr(X^k)\in\C[\gfrak]^G$ for $2\leq k\leq n$. The functions $P_2,\dots,P_n$ are transverse coordinates for the embedding $i$. Locally, $F_j^0\Nbb_c^\circ$ can be identified with
		\begin{align}
			\label{eqn: F transverse}
			\sum_{ {a_2,\ldots,a_n\geq0,\ a_2+\cdots+a_n\leq j}} i_*\left( \Ecal_c\otimes\omega_{\Nscr_r/\gfrak_r} \right) \partial_{P_2}^{a_2}\cdots\partial_{P_n}^{a_n}.
		\end{align}
		Consequently, $\wt{\gr}^0(\Nbb_c^\circ)$ is locally free of positive rank on $\Cfrak_\nil^\circ = \{(x,y)\in\Nscr_r\times\Nscr\mid[x,y]=0\}.$
		In particular,
		\begin{align}
			\label{formula:generic-annihilator}
			\Ann_{R_{\pfrak}}(M_{\pfrak})=\pfrak R_{\pfrak}.
		\end{align}
		
		By \cite[Lemma 5.1.13]{saitofrench}, $M$ is a Cohen-Macaulay $R$-module of dimension $\dim\gfrak$. It therefore has no embedded associated primes, and hence $\Ass_R(M)=\{\pfrak\}$.
		
		Put $I=\Ann_R(M)$. Since $M$ is finite and faithful over $R/I$, choosing finitely many generators gives an embedding $R/I\hookrightarrow M^{\oplus r}$. Therefore $\Ass_R(R/I)\subset\Ass_R(M)=\{\pfrak\}.$
		Since $\sqrt I=\pfrak$, it follows that $\Ass_R(R/I)=\{\pfrak\}$, so $I$ is $\pfrak$-primary.
		
		By \eqref{formula:generic-annihilator}, $I_{\pfrak}=\pfrak R_{\pfrak}$. If $I\neq\pfrak$, then $\pfrak/I$ is a nonzero submodule of $R/I$, and hence $\Ass_R(\pfrak/I)=\{\pfrak\}$. On the other hand, $(\pfrak/I)_{\pfrak} = \pfrak R_{\pfrak}/I_{\pfrak} = 0,$ a contradiction. Thus $I=\pfrak$. 
		
		For $j<0$, the formula \eqref{eqn: F transverse} implies that $\gr^0_jN_c$ vanishes on $\gfrak_r$. Therefore, the $R$-submodule $K=R\gr^0_jN_c\subset M$ is supported over $\Nscr\setminus \Nscr_r$ but also satisfies $\Ass_R(K)\subset \Ass_R(M)=\{\pfrak\}$. Thus $K$ has to be 0. Since $F^0$ is bounded below, we conclude that $F_{-1}^0N_c=0$.
	\end{proof}
	Proposition \ref{prop: support gr Nc} justifies the following definition.  
	\begin{defn}\label{def:Fc}
		We take $\Scal_c$ to be the non-derived restriction of $\gr^0(\Nbb_c)$ to $\Cfrak$ and define
		\[\Fcal_c:=(\varpi_*p^*\mathcal S_c)^{\ol G}.\]
		{We give $\Scal_c$ and $\Fcal_c$ the $A$-equivariance from the Koszul construction \eqref{eqn: Ae}.}
	\end{defn}
	Recall that $\gr^0(\Nbb_c)\cong Ra_*\Acal_c$. In terms of $Ra_*\Acal_c$, the definition of $\Fcal_c$ makes sense for any integer $m$. The support of $\Fcal_c$ is $\Hilb_0^n$ and it satisfies $\Fcal_c\otimes \Ocal(k)\cong \Fcal_{c+k}$ for any integer $k$. 
	We have the following characterization of its $K$-theory class. 
	\begin{prop}\label{prop:Fc and P}
		For $c=m/n$ for an integer $m$ coprime to $n$, we have 
		\[\kappa([\Fcal_{c}])=q^{n-1}(Q_{m+n,n}1).\]
	\end{prop}
	\begin{proof}
		By Proposition \ref{prop:actions on Lambda}, it suffices to show $C(q^{n-1}Q_{n,m})=[\Scal_c]$, where we view $[\Scal_c]$ as a class in $K$ by extending by $0$ from $\Cfrak$ to $\ol{\Cfrak}$. 
		
		Since $Q_{n,m}=D_{\lceil\mu_c\rceil}$, the identity $\ch_{\GL_r\times A}\circ C\circ\psi={\mathscr H}$ from Proposition \ref{prop: ch C psi} applying to $\eta_{\lceil\mu_c\rceil}$ in \eqref{eqn: eta} gives
		\begin{align*}
			\ch_{\GL_n\times A}(C(Q_{n,m}))=\Hcal_c.
		\end{align*}
		On the other hand, Proposition \ref{prop: ch Nc} and Proposition \ref{prop: cuspidal catalan} gives
		\begin{align*}
			\ch_{\GL_n\times A}([\Scal_c])=\Csf_c=q^{n-1}\Hcal_c.
		\end{align*}
		By injectivity of $\ch_{\GL_n\times A}$ on localized equivariant $K$-theory \cite[2.5.1]{Varagnolo-Vasserot}, the equality $C(q^{n-1}Q_{n,m})=[\Scal_c]$ follows.
	\end{proof}
	Next, we apply localization theorem to $\kappa([\Fcal_c])$ to obtain an explicit formula of $Q_{m+n,n}\cdot 1$. 
	\subsection{Modified Macdonald polynomials and fixed points}
	For a partition $\lambda\vdash n$, let $\lambda^t$ denote its transpose.
	The modified Macdonald polynomial $\wt H_\lambda[Z;q,t]$ is uniquely characterized by
	\begin{align*}
		\wt H_\lambda[(1-q)Z;q,t]
		&\in\C(q,t)\langle s_\mu\mid\mu\geq\lambda\rangle,
		\\
		\wt H_\lambda[(1-t)Z;q,t]
		&\in\C(q,t)\langle s_\mu\mid\mu\geq\lambda^t\rangle,
		\\
		\wt H_\lambda[1;q,t]&=1.
	\end{align*}
	Here $[-]$ denotes plethystic substitution and $\geq$ is the dominance order on partitions.
	
	On the other hand, the $A$-fixed points of $\Hilb^n$ are the monomial ideals of colength $n$.
	They are indexed by partitions $\lambda\vdash n$. Denote the fixed point associated with $\lambda$ by $I_\lambda$.
	
	\begin{prop}[{\cite[Theorem 4.1.5 and Proposition 5.4.1]{haiman_survey}}]
		Under the isomorphism \eqref{iso: K Lambda}, one has
		\begin{align*}
			\kappa([\Ocal_{\{I_\lambda\}}])=\wt H_\lambda[Z;q,t].
		\end{align*}
	\end{prop}
	
	\subsection{The localization formula}
	For a box $u$ in the Young diagram of a partition $\lambda$, let $a(u)$ and $\ell(u)$ denote its arm and leg, and let $a'(u)$ and $\ell'(u)$ denote its coarm and coleg, as illustrated below.
	\begin{figure}[ht]
		\centering
		\begin{tikzpicture}
			\draw (0,0)--(0,3)--(2,3)--(2,2)--(4,2)--(4,0)--(0,0);
			\draw (1,1)--(1,1.5)--(1.5,1.5)--(1.5,1)--(1,1);
			\draw [<->,>=stealth] (1.5,1.25)--(4,1.25);
			\draw [<->,>=stealth] (0,1.25)--(1,1.25);
			\draw [<->,>=stealth] (1.25,1.5)--(1.25,3);
			\draw [<->,>=stealth] (1.25,0)--(1.25,1);
			\draw (2.7,1) node {$a(u)$};
			\draw (0.5,1) node {$a'(u)$};
			\draw (0.9,2) node {$\ell(u)$};
			\draw (1.8,0.5) node {$\ell'(u)$};
		\end{tikzpicture}
		\caption{Arm, leg, coarm, and coleg of a box $u$.}
		\label{fig:arm}
	\end{figure}
	
	For a finite-dimensional $A$-representation $E$, write
	\begin{align*}
		\Lambda(E)&:=\sum_{k=0}^{\dim E}(-1)^k \ch_A(\wedge^kE).
	\end{align*}
	
	Define $g_\lambda:=\Lambda\bigl(T^*_{I_\lambda}\Hilb^n\bigr)$.  By \cite{ellingsrud-stromme,haimandiscrete}, one has 
	\begin{align}
		g_\lambda=\frac{1}{(1-q)(1-t)} \prod_{u\in\lambda} \bigl(1-q^{a(u)+1}t^{-\ell(u)}\bigr) \bigl(1-q^{-a(u)}t^{\ell(u)+1}\bigr).
		\label{formula:g}
	\end{align}
	
	Let $i_\lambda:\{I_\lambda\}\hookrightarrow\Hilb^n$ be the inclusion.
	For every $[\Fcal]\in K^A(\Hilb^n)$, equivariant localization gives 
	\begin{align}
		\kappa([\Fcal]) =\sum_{\lambda\vdash n} \frac{\wt H_\lambda[Z;q,t]}{g_\lambda} \ch_A\bigl(Li_\lambda^*\Fcal\bigr).
		\label{formula:localization}
	\end{align}
	%    Formula \eqref{formula:localization} will be applied to the coherent sheaf $\Fcal_c$ next.
	\subsection{Principal nilpotent pairs}
	We recollect some results of \cite{ginzburg_principal_nilpotent,isospec}, which are valid for any semisimple Lie group $G$ with Lie algebra $\gfrak$.
	
	%Let $\Cfrak:=\{(x,y)\in\gfrak\times\gfrak\to \gfrak,[x,y]=0\}_\red$ be the commuting variety with the reduced structure.
	%Let $\Cfrak$, resp. $\Ifrak\Cfrak$ be the normalization of $\Cfrak_\gfrak$, resp, $\Ifrak\Cfrak_\gfrak$. One of the main results in \textit{op. cit.} is that $\Ifrak\Cfrak$ is Gorenstein and thus $\Cfrak$ is Cohen-Macaulay. Let $p:\Ifrak\Cfrak\to \Cfrak$ be the lift of $p_\gfrak$.
	{Write $\Cfrak_r$ for the locus of commuting pairs whose common centralizer has dimension $\operatorname{rk}\gfrak$.} The subvariety $\Cfrak_r\subset \Cfrak$ is smooth and open dense.
	
	%   \xcom{Consider $\Cfrak_r\xrightarrow{j} \Cfrak_{norm}\xrightarrow{p} \Cfrak$. The annihilator of the CM-module $\gr(\Nbb)$ is generically reduced and hence $\gr(\Nbb)$ is a $\Ocal_\Cfrak$-module. Moreover, by \cite[Lemma 2.6.1]{isospec}, $(p\circ j)_*(p\circ j)^*\gr \Nbb\cong \gr(\Nbb)$. This implies that $\gr(\Nbb)$ is a $p_*\Ocal_{\Cfrak_{norm}}$-module. not true.. not codim 2}
	
	\label{C*actions}
	\begin{defn}(\cite{ginzburg_principal_nilpotent})
		An element $\ebf=(x_1,x_2)$ in $\Cfrak_r$ is called a principal nilpotent pair if for any $(t_1,t_2)\in A$, there exists some $g\in G$ such that $(t_1x_1,t_2x_2)=(\Ad(g)x_1,\Ad(g) x_2)$.
	\end{defn}
	It is shown in \cite[Theorem 1.2]{ginzburg_principal_nilpotent} that for every principal nilpotent pair $(e_1,e_2)$, there exists an associated semisimple pair $\hbf=(h_1,h_2)$, unique up to conjugation by the unipotent radical of the centralizer of $ {\ebf}$ in $G$, such that $\hbf$ is regular {, $[h_1,h_2]=0$,} and
	\[
	[h_i,e_j]=\delta_{ij}e_j, \quad i,j=1,2.
	\]
	The adjoint action of $(h_1,h_2)$ decomposes $\gfrak$ into weight spaces
	\[
	\gfrak=\bigoplus_{ {(a,b)\in \Z^2}}\gfrak_{a,b},\quad \text{where }\ad(h_1)x=ax, \ad(h_2)x=bx, \forall x\in \gfrak_{a,b}.
	\]
	For every fixed principal nilpotent pair $\ebf$ with an associated semisimple pair $\hbf$, let 
	\begin{align}\label{action: rho}
		\rho: A\to G/Z(G)
	\end{align}
	be the 2-parameter subgroup with differential at the identity being $\C^2\to \gfrak$: $(1,0)\mapsto h_1$, $(0,1)\mapsto h_2$. We define an $A$-action on $\gfrak$ by $\Ad(\rho)$ and an $A$-action $\gfrak\times\gfrak$ by
	\begin{align}
		\label{C*action_on_gtimesg} (t_1,t_2)(x,y)=(t_1^{-1}\Ad(\rho(t_1,t_2))x,t_2^{-1}\Ad(\rho(t_1,t_2))y)
	\end{align}
	such that $\ebf$ is a fixed point under this action. The weight space $(\gfrak\oplus\gfrak)_{a,b}$ under the action (\ref{C*action_on_gtimesg}) equals $\gfrak_{a+1,b}\oplus \gfrak_{a,b+1}$.
	%By taking Hodge associated graded of the Harish-Chandra $D$-module, in \cite{isospec} the coherent sheaf $\Rcal$ is expressed as the pushforward of a certain DG algebra. This reinterpretation allows \cite{bgcombo} to compute the bigraded dimension of the stalk of the coherent sheaf $\Rcal$ at a principal nipotent pair. These computation results matches exactly those in \cite{garsia-haiman-pieri}.

	\subsection{Principal nilpotent pairs for $\sln$}
	
	{The principal nilpotent pairs in $\sln$, modulo simultaneous conjugation and simultaneous matrix transposition, are} parametrized by partitions $\lambda\vdash n$ \cite[Theorem 5.6(b)]{ginzburg_principal_nilpotent} . These correspond to the fixed points $\{I_\lambda\}$ of $\Hilb^n$ under the GIT construction. 
	To see this, choose a basis vector of $V=\C^n$ for each box of the Young diagram of $\lambda$.  The corresponding pair $\ebf_\lambda=(e_1,e_2)$ is defined by letting $e_1$ move a box one step to the right and $e_2$ move a box one step up, with the convention that a vector is sent to zero when the required box is absent. 
	\begin{defn}
		A bijective labeling of the boxes of $\lambda$ by $1,\ldots,n$ is called an \emph{almost standard Young tableau} if the labels increase from left to right along each row, and if, for two vertically adjacent boxes, the upper label is larger than the lower label except that a box labeled $a-1$ is allowed to lie immediately above the box labeled $a$.
	\end{defn}
	Thus a standard Young tableau is an almost standard Young tableau with no exceptional vertical adjacency.  We write $\ASYT_\lambda$ and $\SYT_\lambda$ for the sets of almost standard and standard Young tableaux of shape $\lambda$, respectively. Almost standard tableaux appear in Negu\c{t}'s study of eccentric correspondences \cite[Section 4.5]{negutthesis}. The related eccentric flag commuting stack is also discussed in \cite[Section 2.3]{gn_commuting_idempotent}.
	In the present setting, the ASYT condition records precisely which negative simple-root directions are allowed in the Hessenberg space $\mfrak$.
	
	Every tableau $\sigma$ of shape $\lambda$ determines a realization of an associated semisimple pair for $\ebf_\lambda$ \cite[Section 4.4]{bgcombo}.  If the box labeled $k$, denoted by $\square_k$, has coordinates
	\begin{align*}
		a'(\square_k)=a_k, \ell'(\square_k)=b_k,
	\end{align*}
	{Assign the box labeled $k$ to the $(n-k+1)$-st coordinate vector and} define
	\begin{align}
		h_1&=\operatorname{diag}(a_n,\ldots,a_1) -\frac{1}{n}\left(\sum_{k=1}^n a_k\right)I_n,
		\nonumber\\
		h_2&=\operatorname{diag}(b_n,\ldots,b_1) -\frac{1}{n}\left(\sum_{k=1}^n b_k\right)I_n.
		\label{formula:h-12}
	\end{align}
	Then $h_1,h_2\in\sln$ and $[h_i,e_j]=\delta_{ij}e_j.$
	Moreover, $(h_1,h_2)$ is regular.  For example, if $(e_1,e_2)=(\ul e,0)$ in the notation of \eqref{formula:efh}, formula \eqref{formula:h-12} gives $(h_1,h_2)=(\frac12\ul h,0)$.

	Fix $\ebf_\lambda$ and the associated pair $\hbf_\lambda=(h_1,h_2)$ above.
	Because $\hbf_\lambda$ is regular, its common centralizer is a Cartan subalgebra.  The Borel subalgebras containing $h_1$ and $h_2$ are therefore
	indexed by ${S_n}$; write them as $\bfrak_w$, for $w\in {S_n}$. \cite[Lemma 4.4.1]{bgcombo} establishes a bijection \[\bigl\{\bfrak_w\mid e_1,e_2\in\nfrak_w\bigr\}
	\mathrel{\mathop{\longleftrightarrow}^{\sim}}\SYT_\lambda.\] 
	Similarly, let $\nfrak_w$ and $\mfrak_w$ be the conjugates of $\nfrak$ and $\mfrak$ by $w$.  
	\begin{lem}\label{lem:asyt}
		There is a bijection
		\begin{align*}
			\bigl\{\bfrak_w\mid e_1\in\nfrak_w, e_2\in\mfrak_w\bigr\} \mathrel{\mathop{\longleftrightarrow}^{\sim}}\ASYT_\lambda.
		\end{align*}
	\end{lem}
	
	\begin{proof}
		As in the proof of \cite[Lemma 4.4.1]{bgcombo}, a Borel containing $h_1,h_2$ is equivalent to an ordering of their joint eigenlines, hence to a labeling of the boxes of $\lambda$.  The condition $e_1\in\nfrak_w$ forces the labels to increase to the right.  Since $\mfrak_w$ is obtained from $\bfrak_w$ by adjoining the negative simple-root spaces, the condition $e_2\in\mfrak_w$ forces the labels to increase upward, except that a vertical arrow may lie in a negative simple-root space.  Such a root joins two consecutive labels, giving precisely the exceptional adjacency $a-1$ above $a$.  These conditions are exactly the definition of an almost standard tableau.
	\end{proof}
	
	\subsection{Bigraded characters}\label{subsec:bigraded-characters}
	The result in this section can be compared with a special case of \cite[Proposition 5.5]{negutflag}.
	
	For a tableau $\sigma$, define the weight of the box labeled by $i$ to be  $\chi_i=q^{a'(\square_i)}t^{\ell'(\square_i)}$. For a Laurent polynomial $F=\sum a_{ij}q^it^j$, define
	\begin{align*}
		\Omega(F)=\prod_{i,j}(1-q^it^j)^{a_{ij}},\qquad \Omega^0(F)=\Omega(F-[q^0t^0]F).
	\end{align*}
	Also write
	\begin{align*}
		\omega(x)=\frac{(1-x)(1-qtx)}{(1-qx)(1-tx)}.
	\end{align*}
	\begin{prop}\label{prop: Qmn loc}
		For every $m>0$ coprime to $n$,  {put $c=m/n$.} We have that
		\begin{align}
			q^{n-1}Q_{m+n,n}\cdot 1=\kappa([\Fcal_c]) =\left(\frac{1-qt}{1-t}\right)^{n-1}\sum_{\lambda\vdash n}{\wt H_\lambda} \sum_{\sigma\in\ASYT_\lambda} \frac{\Xi_\sigma\prod_{i=1}^n \chi_{n-i+1}^{\lceil\mu_c\rceil(i)}} {\widehat{\prod}_{i=1}^{n-1}(1-t\chi_{i+1}/\chi_i)}.
			\label{eq:cusp loc}
		\end{align}
		where 
		\begin{align}\label{formula:xi}
			\Xi_\sigma:=\widehat{\prod}_{i=1}^n(1-\chi_i^{-1})^{-1} \widehat{\prod}_{i<j}\omega(\chi_i/\chi_j).
		\end{align}
		Here and below, a hatted product means omitting  every linear factor $1-q^at^b$ with $(a,b)=(0,0)$, in both numerator and denominator.
	\end{prop}
	\begin{proof}
		The first equality is Proposition \ref{prop:Fc and P}. We show the second equality using localization formula. Set $\Fscr_c:=(Ra_*\Acal_c)|_{\Cfrak_r}$. Write $\ebf:=\ebf_\lambda$ with an embedding $i_\ebf: \{\ebf\}\to \Cfrak_r$.
		By the definition of $\Fcal_c$, we have an $A$-equivariant isomorphism
		\begin{align*}
			Li_\lambda^*\Fcal_c
			&\cong Li_{\ebf}^*\Fscr_c.
		\end{align*}
		Here the action on the right is given by \eqref{C*action_on_gtimesg}.
		
		Write $X:=G\times^B(\nfrak\times \mfrak)$.  {Set $U:=\{(x,y)\in\gfrak^2:\dim\zfrak(x,y)\leq\operatorname{rk}\gfrak\}$. 
			Then $U\cap\Cfrak=\Cfrak_r$. Take $X_U:=X\times_{\gfrak^2}U$, which is smooth, being open in $X$.} Under the $A$-action of \eqref{C*action_on_gtimesg}, the fixed locus decomposes as
		\begin{align*}
			{X_U^A} =\bigsqcup_{w\in {S_n}} {X_U^{w,A}}, \text{ where } {X_U^{w,A}} :=\{\bfrak_w\}\times \Bigl((\nfrak_w\oplus\mfrak_w) \cap {U^A}\Bigr).
		\end{align*}
		The commutator map gives $[\nfrak_w,\mfrak_w]=\bfrak_w$, but we will keep writing $[\nfrak_w,\mfrak_w]$ to record the $A$-action.  {On this target the action is $(t_1,t_2)\cdot z=t_1^{-1}t_2^{-1}\Ad(\rho(t_1,t_2))z$, so $[\nfrak_w,\mfrak_w]^A=[\nfrak_w,\mfrak_w]\cap\gfrak_{1,1}$. We index the contributing $w$ by $\ASYT_\lambda$ using Lemma \ref{lem:asyt}, and put $i_{w,\ebf}:\{(\bfrak_w,\ebf)\}\hookrightarrow X_U$.}  Define the quotient
		\begin{align}
			\Rfrak_w :=[\nfrak_w,\mfrak_w]/[\nfrak_w,\mfrak_w]^A.
			\label{formula:rw}
		\end{align}
		Parallel to \cite[Proposition 3.8.6]{bgcombo}, one has 
		\begin{align}
			\ch_A\bigl(Li_{\ebf_\lambda}^*\Fscr_c\bigr)
			&={\Lambda}\left(
			Li_{\ebf_\lambda}^*T^*_{\Cfrak_r^A}\Cfrak_r\right) \sum_{w\in\ASYT_\lambda} {\Lambda}\left( {Li_{w,\ebf}^*T^*_{X_U^{w,A}}X_U}\right)^{-1}{\Lambda}(\Rfrak_w^*) \ch_A(\C_{\lceil\mu_c\rceil}).
			\label{eqn:bg}
		\end{align}
		
		We evaluate the terms in \eqref{eqn:bg}. For the line bundle, take the lift of \eqref{action: rho} defined by: 
		\[\wt\rho(t_1,t_2)=\operatorname{diag}(t_1^{a_n}t_2^{b_n},\ldots,t_1^{a_1}t_2^{b_1})\in\GL_n
		\]
		which fixes the cyclic vector corresponding to the box $(0,0)$. 
		Under the action of $\Ad(\wt\rho(A))$, we have 
		\begin{align*}
			\ch_A(\C_{\lceil\mu_c\rceil}) =\prod_{i=1}^n \chi_{n-i+1}^{\lceil\mu_c\rceil(i)}.
		\end{align*}
		Write $\ebf=\ebf_\lambda$ and $\zfrak(\ebf):=\Stab_\gfrak(\ebf)$.  By definition,
		\begin{align*}
			{\Lambda}(\Rfrak_w^*) ={\Lambda}\left( ([\nfrak_w,\mfrak_w]/[\nfrak_w,\mfrak_w]^A)^*\right).
		\end{align*}
		Lemma 3.9.1 and Lemma 4.6.1 of \cite{bgcombo}, with $\bfrak_w$ replaced by the Hessenberg space $\mfrak_w$ in the second identity, give 
		\begin{align*}
			{\Lambda}\left( Li_{\ebf}^*T^*_{\Cfrak_r^A}\Cfrak_r\right)
			&=g_\lambda
			{\Lambda}\left( (\gfrak/(\zfrak(\ebf)\oplus\hfrak))^*\right),
			\nonumber\\
			{\Lambda}\left( {Li_{w,\ebf}^*T^*_{X_U^{w,A}}X_U}\right)
			&={\Lambda}(\nfrak_w)
			{\Lambda}\left( \left( \frac{\nfrak_w\oplus\mfrak_w} {(\gfrak\oplus\gfrak)^A \cap(\nfrak_w\oplus\mfrak_w)} \right)^*\right),
		\end{align*}
		where $g_\lambda$ is as in \eqref{formula:g}. 
		
		Let $R$ and $R^+$ be the sets of roots and positive roots determined by $\bfrak_w$, and let $S\subset R^+$ be the simple roots.  Set
		\begin{align*}
			\ol\Sigma:=\sum_{1\leq i<j\leq n}\chi_i\chi_j^{-1}, \qquad \ul\Sigma:=\sum_{1\leq i<j\leq n}\chi_i^{-1}\chi_j.
		\end{align*}
		Then under our convention, 
		\begin{align*}
			\sum_{\alpha\in R^+} q^{\alpha(h_1)}t^{\alpha(h_2)} =\ul\Sigma,\qquad \sum_{\alpha\in S} q^{\alpha(h_1)}t^{\alpha(h_2)} =\sum_{i=1}^{n-1}\chi_{i+1}\chi_i^{-1}.
		\end{align*}
		Furthermore,
		\begin{align*}
			\ch_A(\hfrak^*)=n-1,\qquad \ch_A(\gfrak^*)=n-1+\ol\Sigma+\ul\Sigma,\qquad \ch_A(\zfrak(\ebf)^*)= \sum_{i=1}^n\chi_i^{-1}-1.
		\end{align*}
		Consequently, ${\Lambda}\left( (\gfrak/(\zfrak(\ebf)\oplus\hfrak))^*\right) =\Omega(\ol\Sigma+\ul\Sigma) \Omega^0\left(\sum_{i=1}^n\chi_i^{-1}\right)^{-1}.$ 
		Also,
		\begin{align*}
			\ch_A(\nfrak_w)=\ul\Sigma,\quad {\Lambda}(\nfrak_w)=\Omega(\ul\Sigma).
		\end{align*}
		The $A$-action of \eqref{C*action_on_gtimesg} yields
		\begin{align*}
			\ch_A(\nfrak_w^*\oplus\{0\})
			&=q\ol\Sigma,
			\nonumber\\
			\ch_A(\{0\}\oplus\mfrak_w^*)
			&=(n-1)t+t\ol\Sigma
			+t\sum_{i=1}^{n-1}\chi_{i+1}\chi_i^{-1}.
		\end{align*}
		It follows that
		\begin{align*}
			{\Lambda}\left( \left( \frac{\nfrak_w\oplus\mfrak_w} {(\gfrak\oplus\gfrak)^A \cap(\nfrak_w\oplus\mfrak_w)} \right)^*\right)= \Omega^0\left( (n-1)t+(q+t)\ol\Sigma +t\sum_{i=1}^{n-1}\chi_{i+1}\chi_i^{-1} \right).
		\end{align*}
		Finally, $\ch_A([\nfrak_w,\mfrak_w]^*) =(n-1)qt+qt\ol\Sigma$ implies that 
		\begin{align*}
			{\Lambda}(\Rfrak_w^*) =\Omega^0\bigl((n-1)qt+qt\ol\Sigma\bigr).
		\end{align*}
		Putting all pieces together, one obtains \eqref{eq:cusp loc}.
	\end{proof}

	\subsection{Comparison with other DG modules}
	
	\subsubsection{Catalan DG modules}
	Recall the DG module \eqref{eqn: DG n n}. We believe that the action map from $\mathrm{Spec}_{\Xfrak_{\nfrak,\nfrak}}\Acal'_0$ to $\gfrak^2$ factors through a map $a':\mathrm{Spec}_{\Xfrak_{\nfrak,\nfrak}}\Acal'_0\to\Cfrak$. Under this assumption, define $\Fcal'_c:=(R\varpi_*Lp^*Ra'_*\Acal'_c)^{\ol G}$. 
	Then parallel to Propositions \ref{prop:Fc and P} and \ref{prop: Qmn loc}, one can compute
	\begin{align*}
		\kappa([\Fcal'_c])=Q_{m+n,n}\cdot 1 =\sum_{\lambda\vdash n}{\wt H_\lambda} \sum_{\sigma\in\SYT_\lambda} \frac{ \Xi_\sigma \prod_{i=1}^n \chi_{n-i+1}^{\lfloor\mu_c\rfloor(i)}} {\widehat{\prod}_{i=1}^{n-1} \left(1-qt\frac{\chi_i}{\chi_{i+1}}\right)}.
	\end{align*}
	There is expected to be a $\ol G\times A$-equivariant isomorphis in the derived category of coherent sheaves on $\Cfrak$: 
	\begin{align*}
		q^{1-n}\Scal_c\simeq Ra'_*\Acal'_c
	\end{align*}
	A nonderived version of this conjecture appears in \cite{gn_commuting_idempotent}. 
	We note that it is not known a priori that the pushforward of $\Acal'_c$ to $\gfrak^2$ is concentrated in a single cohomological degree.  By contrast, $Ra_*\Acal_c$ is concentrated in one degree because it is the associated graded of $\Nbb_c$.  %The involution $T^*\gfrak\rightarrow T^*\gfrak$, $(x,y)\mapsto(y,-x)$, visibly preserves the Catalan model while a similar  invariance of the cuspidal DG module only follows from invariance of $\Nbb_c$ under the Fourier transform. 
	
	%%%%%%%%%%%%%%%%%%%%%%%%
	\subsubsection{The Procesi bundle}\label{sec:procesi-bundle}
	
	The diagonal actions of $G$ on $\Cfrak$ and ${S_n}$ on $\hfrak^2$ give a Chevalley-type isomorphism
	\begin{align*}
		\C[\Cfrak]^G&\cong\C[\hfrak^2]^{S_n}.
	\end{align*}
	Let $\Ifrak\Cfrak_{\norm}$ be the normalization of $\bigl(\Cfrak\times_{\hfrak^2/{S_n}}\hfrak^2\bigr)_{\red}$, let $\alpha:\Ifrak\Cfrak_{\norm}\to\Cfrak$ be the projection, and set $\Rcal:=\alpha_*\Ocal_{\Ifrak\Cfrak_{\norm}}.$ 
	Recall the quotient map $\varpi:\wt{\Hilb^n}\to\Hilb^n$.  According to \cite[Corollary 3.2.5]{isospec}, 
	\begin{align}
		(\varpi_*p^*\Rcal)^{\ol G} \cong \Pcal
		\label{formula:r-procesi}. 
	\end{align} 
	
	Theorem 1.3.3 and 4.3.2 of \textit{loc. cit.} realize $\Rcal$ as a pushforward of $\Ocal_{\Xfrak_{\bfrak,\bfrak}} \otimes^L_{q^{-1}\Ocal_{G\times^B \nfrak}} q^{-1}Ri_*\Ocal_\Bcal$, a specialization of \eqref{eqn: DG b b} at $\ell=0$ in the trace free setting, by taking Hodge associated graded of the Harish-Chandra $D$-module. As a result, \cite[Theorem 4.5.1 and Proposition 4.8.2]{bgcombo} compute
	\begin{align}
		\ch_A\bigl(Li_{\ebf_\lambda}^*\Rcal\bigr) =\frac{g_\lambda}{(1-q)^{n-1}(1-t)^{n-1}} \sum_{\sigma\in\SYT_\lambda}\Xi_\sigma.
		\label{formula:procesi}
	\end{align}
	By \eqref{formula:r-procesi}, this is also the character of the Procesi fiber at $I_\lambda$ and coincides with the Garsia-Haiman formula derived from the Macdonald Pieri rule \cite{garsia-haiman-pieri_rule}.
	\section{Vanishing of higher cohomology via window categories}\label{sec:vanishing}
	The goal of this section is to prove Theorem \ref{thm: vanishing}. Throughout this section, all functors between derived categories are derived unless otherwise specified.
	\subsection{KN stratification}
	\label{subsec:kn-stratification} We follow the definitions from \cite{DanielHL-GIT}. 
	Let $G$ be reductive and fix a $G$-equivariant embedding $X\hookrightarrow\Pbb^n\times\Abb^m$ to define $\Lcal:=\Ocal_X(1).$ Fix a positive-definite norm on the space of real cocharacter invariant under the Weyl group action. For a one-parameter subgroup $\lambda$ and a connected component $Z\subset X^\lambda$, the associated attracting locus (blade) is
	\[
	Y_{\lambda,Z} := \left\{ x\in X \,\middle|\, \lim_{t\to 0}\lambda(t)\cdot x\in Z \right\}.
	\]
	Taking limit defines a contracting morphism $\pi:Y_{\lambda,Z}\rightarrow Z.$
	
	Let $P_\lambda\subset G$ be the parabolic subgroup consisting of those $p\in G$ for which $\lim_{t\to 0}\lambda(t)p\lambda(t)^{-1}$ exists, and let $L_\lambda=Z_G(\lambda)$ be its Levi subgroup.  If $X$ is smooth along $Z$, then $\pi$ is a $P_\lambda$-equivariant locally trivial affine bundle.
	For a KN stratum $S=G\cdot Y_{\lambda,Z}$, one moreover has $S\simeq G\times^{P_\lambda}Y_{\lambda,Z}$, and the restriction to $Z$ of the conormal bundle $\Ical_S/\Ical_S^2$ has strictly positive $\lambda$-weights.
	
	The Kempf–Ness stratification is constructed inductively.  At the $i$-th step one chooses, up to $G$-conjugacy, a pair $(\lambda_i,Z_i^\ast)$ maximizing $\mu(\lambda,Z) := -\frac{\Wt_\lambda(\Lcal|_Z)}{\lVert\lambda\rVert}$ among pairs for which $Z$ is not contained in the union of the previously constructed strata and for which $\mu(\lambda,Z)>0$.  One then removes from $Z_i^\ast$ the previously constructed strata, obtaining an open subset $Z_i\subset Z_i^\ast$, and sets $Y_i:=\pi_i^{-1}(Z_i)$ to be the attracting locus and $S_i:=G\cdot Y_i.$
	
	Thus $X$ is obtained from its semistable locus by successively adjoining the unstable strata $S_i$.  We write $L_i=Z_G(\lambda_i)$ and $P_i :=P_{\lambda_i}$.
	\subsection{The non-commutative Hilbert scheme}
	\label{subsec:cyclic-triples}
	Let $\ol G=\GL(V)$ act diagonally on $X=\gfrak^2\times V$. For the character $\det^{-1}$, semistability is precisely the condition that $v$ generate $V$ under $x$ and $y$:
	
	\begin{align*}
		X^{\mathrm{ss}}&=\{(x,y,v):\C\langle x,y\rangle v=V\}.
	\end{align*}
	The action on this cyclic locus is free, and $X^{\mathrm{ss}}/\ol G$ is smooth and irreducible. The equation $[x,y]=0$ cuts out $\Hilb^n$  {inside $X^{\mathrm{ss}}/\ol G$}.
	
	The unstable strata are indexed by the dimension $i<n$ of the cyclic subspace. To describe the stratum with index $i$, choose a decomposition $V=V_i^0\oplus V_i^-$ with $\dim V_i^0=i$, and use the one-parameter subgroup
	\[
	\lambda_i(t)=\operatorname{diag}(1_{V_i^0},t^{-1}1_{V_i^-}).
	\]
	Its Levi subgroup is $L_i=\GL(V_i^0)\times\GL(V_i^-)$, and the corresponding parabolic is $P_i=\Stab_{\ol G}(V_i^0)$.
	Write $\lfrak_i=\operatorname{Lie}(L_i)$ and $\pfrak_i=\operatorname{Lie}(P_i)$. The centers and blades are
	\begin{align*}
		Z_i&=\{(x,y,v)\in(\lfrak_i\cap\gfrak)^2\times V_i^0:
		\C\langle x,y\rangle v=V_i^0\},\\
		Y_i&=\{(x,y,v)\in(\pfrak_i\cap\gfrak)^2\times V_i^0:
		\C\langle x,y\rangle v=V_i^0\},\\
		S_i&\cong\ol G\times^{P_i}Y_i.
	\end{align*}
	\subsection{Derived $\Theta$-stratification}\label{subsec:induced-strata}
	We first recall the content of \cite[Definition 1.2.2 and Lemma 1.2.3]{HalpernLeistner_DerivedTheta}.
	Set $\Theta=[\Abb^1/\Gbb_m]$ with the scaling action. For a derived algebraic stack $\mathcal{U}$, write
	\[
	\Filt(\mathcal{U})=\operatorname{Map}(\Theta,\mathcal{U}), \qquad \Grad(\mathcal{U})=\operatorname{Map}(B\Gbb_m,\mathcal{U}).
	\]
	Evaluation at $1$ gives the forgetful map $\operatorname{ev}_1:\Filt( \Ucal)\to \Ucal$. 
	Restriction to $[\{0\}/\Gbb_m]$ gives the associated-graded map $\operatorname{ev}_0:\Filt( \Ucal)\to \Grad( \Ucal)$, and $\Theta\to B\Gbb_m$ induces the split-filtration section $s:\Grad( \Ucal)\to\Filt( \Ucal)$.
	
	\begin{defn}
		\label{def: theta}
		Let $\Xcal$ be a derived algebraic stack locally almost of finite presentation over $\C$, with affine diagonal. A derived $\Theta$-stratum in an open substack $ \Ucal\subset\Xcal$ is an open and closed substack $\Scal\subset\Filt( \Ucal)$ for which $j=\operatorname{ev}_1|_\Scal$ is a closed immersion.
		Its center is $\Zcal:=s^{-1}(\Scal)$, and the restriction of $s$ gives the split-filtration section $\zeta:\Zcal\to\Scal$.
		
		A finite derived $\Theta$-stratification consists of open substacks
		\begin{align*}
			\Xcal=\mathcal U_0\supset\mathcal U_1\supset\cdots \supset\mathcal U_N=\Xcal^{\mathrm{ss}}
		\end{align*}
		and for each $1\leq a\leq N$, a derived stratum satisfying
		\begin{align*}
			j_a:\Scal_a&\hookrightarrow\Ucal_{a-1},\quad \mathcal U_a=\Ucal_{a-1}\setminus j_a(\Scal_a).
		\end{align*}
		
	\end{defn}
	\subsection{Induced strata on the DG flag Hilbert scheme}
	For a quotient stack, given a one-parameter subgroup, the quotient of the blade defines the filtration, and the quotient of the center defines the associated graded. 
	In the setting of interest, fix the upper triangular Borel $\ol B\subset\ol G=\GL_n$, identify $G/B=\ol G/\ol B$, and set
	\[
	X_{\nfrak,\mfrak}:=\ol G\times^{\ol B}(\nfrak\times\mfrak\times V).
	\]
	For $0\leq i<n$, let $S_{i,n-i}$ be the minimal representatives of the quotient $(S_i\times S_{n-i})\backslash S_n$. Writing $\ol B_\sigma:=\sigma\ol B\sigma^{-1}$, the fixed point locus  {decomposes} as
	\begin{align*}
		(\ol{G}/\ol{B})^{\lambda_i}
		&=\bigsqcup_{\sigma\in S_{i,n-i}}L_i/(L_i\cap\ol B_\sigma).
	\end{align*}
	For each $E=\nfrak,\mfrak,\bfrak$, put $E_\sigma:=\Ad_\sigma E$.
	We have the zero and nonnegative $\lambda_i$-weight spaces:
	\begin{align*}
		E_{\sigma,i}^0&=E_\sigma\cap\lfrak_i,\\
		E_{\sigma,i}^+&=E_\sigma\cap\pfrak_i.
	\end{align*}
	
	Taking the cyclic condition into account, for $\epsilon=0,+$, we define
	\begin{align*}
		W_{\sigma,i}^{\epsilon,\circ}
		&:=\{(x,y,v)\in\nfrak_{\sigma,i}^\epsilon
		\times\mfrak_{\sigma,i}^\epsilon\times V_i^0:
		\C\langle x,y\rangle v=V_i^0\},\\
		H_{i,\sigma}^0&:=L_i\cap\ol B_\sigma,\quad H_{i,\sigma}^+:=P_i\cap\ol B_\sigma.
	\end{align*}
	Then the classical center and blade are
	\begin{align*}
		X_{\mfrak,\sigma}^{\lambda_i,\circ}
		&=L_i\times^{H_{i,\sigma}^0}W_{\sigma,i}^{0,\circ},\\
		X_{\mfrak,i,\sigma}^+
		&=P_i\times^{H_{i,\sigma}^+}W_{\sigma,i}^{+,\circ}.
	\end{align*}
	
	Recall the commutator map $X_{\nfrak,\mfrak}\rightarrow\wt\gfrak=G\times^B\bfrak$, $(g,x,y,v)\mapsto (g,[x,y])$. 
	{Define} $M=X_{\nfrak,\mfrak}\times^R_{\wt\gfrak}\Bcal$ and $\Xcal=[M/\ol G]$. Being a derived zero fiber of the  {commutator map} in a smooth scheme modulo smooth affine groups, $\Xcal$ satisfies the assumptions of Definition \ref{def: theta}. The derived center and blade are respectively 
	\begin{align*}
		M_{\sigma}^{\lambda_i,\circ}
		&=L_i\times^{H_{i,\sigma}^0}
		\left(W_{\sigma,i}^{0,\circ}\times^R_{\bfrak_{\sigma,i}^0}\{0\}\right),\\
		M_{i,\sigma}^{+}
		&=P_i\times^{H_{i,\sigma}^+}
		\left(W_{\sigma,i}^{+,\circ}\times^R_{\bfrak_{\sigma,i}^+}\{0\}\right).
	\end{align*}
	Here the maps to $\bfrak_{\sigma,i}^{\epsilon}$ are the restricted commutator maps. Define
	\[
	\Zcal_{i,\sigma}=[M_{\sigma}^{\lambda_i,\circ}/L_i],\qquad \Scal_{i,\sigma}=[M_{i,\sigma}^{+}/P_i].
	\]
	\begin{lem}
		\label{lem:strata}
		Order the pairs $(i,\sigma)$ such that $(i,\sigma)< (j,\tau)$ if $i<j$ or $i=j$ and $\sigma< \tau$ under the Bruhat order. Choose a linear extension of this order {, denoted again by $<$}. Then the stacks $\Scal_{i,\sigma}$ with centers $\Zcal_{i,\sigma}$, form a derived $\Theta$-stratification of $\Xcal$.
	\end{lem}
    \begin{proof}
		Let superscript ${}^\cl$ denote the classical truncation of a derived stack. 
		Denote the complement of $\left(\bigcup_{(j,\tau)< (i,\sigma)}\Scal_{j,\tau}\right)$ in $\Xcal$ by $\Ucal_{i,\sigma}$. The condition $\dim\mathbb C\langle x,y\rangle v\leq i$ is a closed condition. Moreover, for a fixed rank $i$, the condition $\tau< \sigma$ is the Bruhat closure condition and hence is also closed. Together, we have that $\left(\bigcup_{(j,\tau)< (i,\sigma)}\Scal^\cl_{j,\tau}\right)$ is closed in $\Xcal^\cl$. Similarly $\Scal^\cl_{i,\sigma}$ is closed in $\Ucal^\cl_{i,\sigma}$. Moreover, $\Ucal_{i,\sigma}$ is locally almost of finite presentation and has affine automorphism groups as it is open inside $\Xcal$ of these properties. We proceed to verify the conditions of Definition \ref{def: theta} for $\Scal_{i,\sigma}$ inside $\Ucal_{i,\sigma}$. 
		
		For a group $H$, let $\Brm H$ denote the classifying stack. We have the identification
		\begin{align*}
			\Xcal\cong [(\nfrak_\sigma\oplus\mfrak_\sigma\oplus V)/\overline B_\sigma] \times^R_{[\bfrak_\sigma/\overline B_\sigma]}\mathrm B\overline B_\sigma.
		\end{align*}
		Write $\Filt_{i,\sigma}$ and $\Grad_{i,\sigma}$ for the substacks consisting of maps associated to the one-parameter subgroup $\lambda_i$ and relative position $\sigma$ in the direction of $\ol{G}/\ol{B}$. Applying $\Map(\Theta,-)$ is compatible with derived fiber products. Therefore computing termwise, we have that 
		\[
		\Scal_{i,\sigma}\cong \Filt_{i,\sigma}(\Ucal_{i,\sigma}),\quad \Zcal_{i,\sigma}\cong \Grad_{i,\sigma}(\Ucal_{i,\sigma}).
		\]
		In particular, $\Scal_{i,\sigma}$ is open and closed in $\Filt(\Ucal_{i,\sigma})$. 
		
		Let $U\subset M^{\cl}$ be the inverse image of $\Ucal_{i,\sigma}^{\cl}$. Consider the closed incidence subscheme $I\subset U\times\gr(i,V)$ of tuples $(\bfrak',x,y,v,E)$ such that $v\in E,x(E)\subset E,y(E)\subset E$ and that the relative position of $E$ and $\bfrak'$ is at most $\sigma$. By the definition of $U$, $\dim \C[x,y]v\geq i$. 
		Since $\C[x,y]v$ lies in in $E$ of dimension $i$, we have $E=\C[x,y]v$ and the relative position of $E$ and $\bfrak'$ is exactly $\sigma$. Therefore we have 
		\[
		I\cong\ol G\times^{P_i}(M_\sigma^+)^{\cl}.
		\]
		The projection $I\to U$ is proper. It is also a monomorphism: over any base scheme, the sections obtained by applying words in $x,y$ of length at most $n-1$ to $v$ lie in $E$ and span every fiber of $E$.
		By Nakayama's lemma, these sections generate $E$ as a sheaf, so $E$ is determined as their image in the trivial bundle $V$. Thus $I\to U$ is a proper monomorphism, hence a closed immersion. 
		Finally, $I\to U$ is the base change of $\Scal_{i,\sigma}^{\cl}\to\Ucal_{i,\sigma}^{\cl}$ along $U\to\Ucal_{i,\sigma}^{\cl}$. Thus $\Scal_{i,\sigma}^{\cl}\to\Ucal_{i,\sigma}^{\cl}$ is a closed immersion and so is $\Scal_{i,\sigma}\to\Ucal_{i,\sigma}$. 
	\end{proof}
	\subsection{Quantization commutes with reduction}
	Write $\zeta_{i,\sigma}:\Zcal_{i,\sigma}\to\Scal_{i,\sigma}$ and $j_{i,\sigma}:\Scal_{i,\sigma}\to \Xcal$ for the embeddings. 
	The $\lambda_i$-action gives the weight decomposition on the center. Following \cite[Lemma 1.5.3]{HalpernLeistner_DerivedTheta}, let $\operatorname{QC}(\Zcal)^{\geq w}$, resp. $\operatorname{QC}(\Zcal)^{<w}$ be the derived category of quasi-coherent complexes on $\Zcal$ of weight at least $w$, resp. less than $w$. 
	\begin{prop}\label{prop:window}
		Let $\Fcal,\Gscr$ be perfect complexes on $\Xcal$. Suppose that for each stratum there is an integer $w_{i,\sigma}$ such that
		\begin{align*}
			\zeta_{i,\sigma}^*j_{i,\sigma}^*\Fcal
			&\in\operatorname{QC}(\Zcal_{i,\sigma})^{\geq w_{i,\sigma}},\\
			\zeta_{i,\sigma}^*j_{i,\sigma}^!\Gscr
			&\in\operatorname{QC}(\Zcal_{i,\sigma})^{<w_{i,\sigma}}.
		\end{align*}
		Then restriction induces an isomorphism
		\[
		R\Hom_{\Xcal}(\Fcal,\Gscr) \xrightarrow{\sim}R\Hom_{\Xcal^{\mathrm{ss}}} (\Fcal|_{\Xcal^{\mathrm{ss}}},\Gscr|_{\Xcal^{\mathrm{ss}}}).
		\]
	\end{prop}
	\begin{proof}
		Apply \cite[Proposition 2.1.3]{HalpernLeistner_DerivedTheta} successively in the order of Lemma \ref{lem:strata}, noting that Lemma 1.5.4 of \textit{loc. cit.} reduces its weight conditions to the two conditions in the statement.
	\end{proof}
	
	\subsection{Weight bound}
	\label{subsec:upper-weight-bound}
	To apply Proposition \ref{prop:window} to the DG module $\Acal_c$, we must bound the weights. 
	Set $\Kcal_{i,\sigma}:=j_{i,\sigma}^!\Ocal_\Xcal$. For the perfect complexes used below, the projection formula gives
	\begin{align}
		\label{eqn:!}
		j_{i,\sigma}^!(-)&\simeq j_{i,\sigma}^*(-)\otimes\Kcal_{i,\sigma}.
	\end{align}
	The ambient stack and strata are quasi-smooth. Thus \cite[Proposition 3.1.5 and Lemma 1.3.2]{HalpernLeistner_DerivedTheta} give 
	\begin{align}\label{eqn: weight eta}
		\Wt_{\lambda_i}(\zeta_{i,\sigma}^*\Kcal_{i,\sigma}) =-\eta_{i,\sigma}, \text{ where }\eta_{i,\sigma}:=\Wt_{\lambda_i}\det_\Z (\zeta_{i,\sigma}^*j_{i,\sigma}^*\Lbb_\Xcal)_{>0}.
	\end{align}
	Here $\Lbb_\Xcal$ is the cotangent complex of $\Xcal$, $_{>0}$ selects the positive weight spaces, and $\det_\Z$ is the alternating tensor product of the determinants of the terms.
	Write $\Lcal_{\lambda,\Xcal}$ for the pullback of $\Lcal_\lambda$ from $\Bcal$ to $\Xcal$.
	
	\begin{lem}
		\label{lem:window}
		Let $c>0$. For $0\leq j\leq n$, $0\leq i<n$, and $\sigma\in S_{i,n-i}$,
		\[
		\zeta_{i,\sigma}^*j_{i,\sigma}^! \bigl(\Lcal_{\lceil\mu_c\rceil,\Xcal}\otimes\wedge^jV^*\bigr) \in \mathrm{QC}\bigl(\Zcal_{i,\sigma}\bigr)^{<0}.
		\]
	\end{lem}
	\begin{proof}
		From
		\begin{align*}
			\Xcal&\simeq \left[(\nfrak\times\mfrak\oplus V) \times^R_{\bfrak}\{0\}\,\middle/\,\ol B\right]
		\end{align*}
		we obtain
		\begin{align*}
			[\Lbb_{\Xcal}]
			&=[\nfrak_\sigma^*]+[\mfrak_\sigma^*]+[V^*]
			-[\bfrak_\sigma^*]-[\ol\bfrak_\sigma^*],\\
			\eta_{i,\sigma}
			&=\dim(\nfrak_\sigma)_{-1}+\dim(\mfrak_\sigma)_{-1}
			+(n-i)-2\dim(\bfrak_\sigma)_{-1}.
		\end{align*}
		The extra central line in $\ol\bfrak_\sigma$ has weight zero.
		Writing $\mfrak=\bfrak\oplus\ufrak_-$ gives
		\begin{align*}
			(\nfrak_\sigma)_{-1}&=(\bfrak_\sigma)_{-1},\\
			(\mfrak_\sigma)_{-1}
			&=(\bfrak_\sigma)_{-1}\oplus(\ufrak_{-,\sigma})_{-1}.
		\end{align*}
		Hence $ \eta_{i,\sigma}=(n-i)+\dim(\ufrak_{-,\sigma})_{-1}$. 
		
		We now compute the weights of $\zeta_{i,\sigma}^*j_{i,\sigma}^*(\Lcal_{\lceil\mu_c\rceil,\Xcal}\otimes\wedge^jV^*)$. 
		A weight of $\nu=(\nu_a)$ of $\wedge^jV^*$ satisfies $\nu_a\in\{0,-1\}^n$ and $\sum_a\nu_a=-j$. Any  weight of $\zeta_{i,\sigma}^*j_{i,\sigma}^* (\Lcal_{\lceil\mu_c\rceil,\Xcal}\otimes\wedge^jV^*)$ can be expressed by $\langle\lambda_i,\sigma(\lceil\mu_c\rceil+\nu)\rangle$.
		Since $c>0$, every $\lceil ic\rceil-\lceil(i-1)c\rceil$ is nonnegative, and the first entry is $\lceil c\rceil\geq1$.
		The entries of $\lambda_i$ are $0$ and $-1$, so 
		\[\langle\lambda_i,\sigma(\lceil\mu_c\rceil+\nu)\rangle\leq  {n-i}\]
		If $(\mathfrak u_{-,\sigma})_{-1}=0$, the first entry of $\sigma^{-1}\lambda_i$ must be $-1$, improving this bound to $n-i-1$.
		Since $\eta_{i,\sigma}=(n-i)+\dim(\mathfrak u_{-,\sigma})_{-1}$, we obtain
		\[\max_{\nu}\langle\lambda_i,\sigma(\lceil\mu_c\rceil {+}\nu)\rangle \leq\eta_{i,\sigma}-1.\]
		The lemma now follows from \eqref{eqn:!} and \eqref{eqn: weight eta}.
	\end{proof}
	\subsection{Vanishing of higher cohomology}
	The natural $A$-equivariant structures on $\gr(\Lrm_c)$ and $\Fcal_c$ differ by the change of variables $(q,q^{-1}t)$ and the shift $q^{1-n}$, as we have seen in Theorem \ref{thm:ch Lc Q} versus Proposition \ref{prop:Fc and P}.  {To simplify notation,} for the rest of the paper we write $\Fcal_c^0:=q^{1-n}\Fcal_c$ and when we say $A$-equivariant isomorphism we mean up to the change of variable $(q,q^{-1}t)$, which we  {denote} by $\backsimeq $. 
	
    For any simple $S_n$-representation $V_\mu$, write $\Pcal^\vee_\mu:=\Hom_{S_n}(V_\mu,\Pcal^\vee)$. 
	\begin{thm}\label{thm: vanishing}
		Let $c=m/n>0$ with $\gcd(m,n)=1$. For $0\leq j\leq n-1$, there is a bigraded isomorphism  {after shifting geometric weights by $q^{1-n}$ and substituting $(q,t)\mapsto(q,q^{-1}t)$ on the left} 
		\begin{align*}
			H^i(\Hilb^n,\Fcal^0_c\otimes\Pcal^\vee_{\wedge^j\hfrak})\backsimeq\begin{cases}
				0 &\text{ when } i> 0\\
				\Hom_{S_{n}}(\wedge^j\hfrak_{n},\gr^0\Lrm_{c}) &\text{ when }i=0
			\end{cases}
		\end{align*}
		
	\end{thm}
	\begin{proof}
		Recall that
		\begin{align*}
			\Hom_{S_n}(\wedge^jV,\Pcal^\vee) \cong\wedge^j\Vcal^\vee.
		\end{align*}
		where $\Vcal$ is the rank-$n$ tautological vector bundle on $\Hilb^n$. 
		Moreover, $\Vcal$ decomposes as $\Ocal\oplus\Vcal'$ and $\Pcal^\vee_{\wedge^j\hfrak}\cong\wedge^j(\Vcal')^\vee$. 
		
		Let $\Xcal^{\mathrm{ss}}$ be the semistable locus.
		For $0\leq  {j}\leq n$, apply Proposition \ref{prop:window} to $\Ocal_{\Xcal}$, $\Lcal_{\lceil\mu_c\rceil,\Xcal}\otimes\wedge^jV^*$, and $w_{i,\sigma}=0$, as Lemma \ref{lem:window} supplies weight bound condition: 
		\begin{align}\label{iso: window M}
			R\Gamma(\Xcal,\Lcal_{\lceil\mu_c\rceil,\Xcal}\otimes\wedge^jV^*)
			&\xrightarrow{\sim}
			R\Gamma(\Xcal^{\mathrm{ss}},\Lcal_{\lceil\mu_c\rceil,\Xcal}\otimes\wedge^jV^*).
		\end{align}
		
		Theorem \ref{thm:gr N Ac} identifies the pushforward of $\Lcal_{\lceil\mu_{c}\rceil,\Xcal}\boxtimes \Ocal_V$ to $X$ with $\gr^0(\ol\Nbb_c)$. Therefore \eqref{iso: window M} gives 
		\begin{align}\label{iso: window Fc}
			R\Gamma(X^{\mathrm{ss}},\gr^0(\ol\Nbb_c)\otimes\wedge^jV^*)^{\ol G}
			&\cong
			\Gamma(X,\gr^0(\ol\Nbb_c)\otimes\wedge^jV^*)^{\ol G}.
		\end{align}
		using that $X$ is affine and $\ol G$ is reductive.
		The left hand side of \eqref{iso: window Fc} is isomorphic to $R\Gamma(\Hilb^n,\Fcal_c\otimes\wedge^j\Vcal^\vee)$,  {by the scheme-theoretic support in Proposition \ref{prop: support gr Nc}}.
		Consequently, $H^{>0}(\Hilb^n,\Fcal_c\otimes\wedge^j\Vcal^\vee)=0$ and  \[\Gamma(\Hilb^n,\Fcal_c\otimes\wedge^j\Vcal^\vee)\cong \Gamma(X,\gr^0(\ol\Nbb_c)\otimes\wedge^jV^*)^{\ol G}.\] 
		
		Contraction with $\tau$ gives a split short exact sequence
		\begin{align*}
			0\rightarrow\Pcal^\vee_{\wedge^j\hfrak} \rightarrow\wedge^j\Vcal^\vee \xrightarrow{\iota_\tau} \Pcal^\vee_{\wedge^{j-1}\hfrak} \rightarrow0.
		\end{align*}
		The bundle $\Pcal^\vee_{\wedge^j\hfrak}$ is a direct summand of $\wedge^j\Vcal^\vee$, so $H^{>0}(\Hilb^n,\Fcal_c\otimes\Pcal^\vee_{\wedge^j\hfrak})=0$. 
		Moreover, composing $\iota_\tau$ with the inclusion $\Pcal^\vee_{\wedge^{j-1}\hfrak}\rightarrow\wedge^{j-1}\Vcal^\vee$ gives that 
		\begin{align*}
			\Gamma(\Hilb^n,\Fcal^0_c\otimes\Pcal^\vee_{\wedge^j\hfrak}) \cong&\Ker\bigl( \Gamma(\Hilb^n,\Fcal^0_c\otimes\wedge^j\Vcal^\vee) \to \Gamma(\Hilb^n,\Fcal^0_c\otimes\wedge^{j-1}\Vcal^\vee)
			\bigr)\\
			\backsimeq &\Ker\bigl( \Gamma(X,\gr^0(\ol\Nbb_c)\otimes\wedge^jV^*)^{\ol G} \to \Gamma(X,\gr^0(\ol\Nbb_c)\otimes\wedge^{j-1}V^*)^{\ol G}
			\bigr)\\
			\cong &\Hom_{S_m} \bigl(\wedge^j\hfrak_m,\gr^0\Lrm_{n/m}^{(m)}\bigr)
		\end{align*}
		using \eqref{eqn: hook kernel}.
		
		Finally, \cite[Proposition 1.5]{gors} supplies a bigraded isomorphism 
		\[\Hom_{S_{m}}(\wedge^j\hfrak_{m},\gr^\afrak\Lrm_{n/m}^{(m)})\cong \Hom_{S_{n}}(\wedge^j\hfrak_{n},\gr^\afrak\Lrm_{m/n}^{(n)}).\]
		Here $\gr^\afrak$ is taken with respect  {to} the filtration defined later in Section \ref{sec:alg} and shown to coincide with $F^0$ in Proposition \ref{prop:hodge-alg}, whose proof does not use anything in this section, so this forward reference is not circular.
	\end{proof}
	\section{Filtrations and full modules}
	In this section, we identify the Hodge filtration with the  $\afrak$-filtration, whose definition is more explicit. This allows us to prove Conjecture \ref{conj:vanisihng} for $c\gg 0$. 
	\subsection{The $\afrak$-filtration}\label{sec:alg}
	On the Cherednik algebra $\Hrm_c$, one has the Fourier transform 
	\begin{align*}
		\Ffrak_c(x_i)=y_i,\qquad\Ffrak_c(y_i)=-x_i,\qquad\Ffrak_c(w)=w, 
	\end{align*}
	which defines the Dunkl form on the polynomial representation $\C[\hfrak]$: 
	\begin{align*}
		(f,g)_c&:=\bigl[\Ffrak_c(f)g\bigr]|_{x_i=0}. 
	\end{align*}
	For $c=m/n>0$ with $m,n$ coprime, the radical $I_c$ of $(-,-)_c$ is nonzero, and $\C[\hfrak]/I_c\cong\Lrm_c$ \cite[Proposition 2.34]{dunklopdam}.
	On $\Lrm_c$ one has the induced Fourier transform $\ol\Ffrak_c$ \cite[(11)]{ma}. For a nonzero highest weight vector $\beta_c$, it is given by $\ol\Ffrak_c(\phi)=\Ffrak_c(\phi)\beta_c$ \cite[Lemma 3.3]{ma}.
	Let $\perp$ denote orthogonal complement for the Dunkl form. 
	Take the ideal $\afrak:=(\C[\hfrak]^{S_n}_+)\subset\C[\hfrak]$. 
	\begin{defn}[{\cite[4.2]{gors}}]
		The $\afrak$-filtration is $F_i^\afrak\Lrm_c=\ol{\Ffrak}_c\bigl((\afrak^{i+1})^\perp\bigr)$ for $i\geq 0$ and $0$ otherwise. 
		The powers of $\afrak$ are taken in $\Lrm_c$.
	\end{defn}
	Alternatively, we have the following characterization.
	\begin{lem}\label{lem:alg ann}
		For $i\geq0$, $F_i^\afrak\Lrm_c=\{\phi\in\Lrm_c\mid\afrak^{i+1}\phi=0\}$.
	\end{lem}
	\begin{proof}
		Define $\langle\phi,\psi\rangle_{\mathrm{Fr}} := \bigl(\phi,\ol{\Ffrak}_c(\psi)\bigr)_c.$
		This form is nondegenerate and Frobenius.  {Indeed, adjointness and $y\ol\Ffrak_c=\ol\Ffrak_c x$ give $\langle x\phi,\psi\rangle_{\mathrm{Fr}}=\langle\phi,x\psi\rangle_{\mathrm{Fr}}$. Since the $x$ variables generate $\Lrm_c$ as an algebra, this proves the Frobenius property.} Thus,
		\[
		\left\{ \phi\in\Lrm_c \,\middle|\, \afrak^{i+1}\phi=0 \right\} = (\afrak^{i+1})^{\perp_{\mathrm{Fr}}} = \ol{\Ffrak}_c^{-1} \left( (\afrak^{i+1})^{\perp} \right).
		\]
		
		Finally, the algebra automorphism \(\ol{\Ffrak}_c^2\) sends $x$ to $-x$ and $y$ to $-y$.  Since \(\afrak^{i+1}\) is homogeneous, \((\afrak^{i+1})^{\perp}\) is preserved by this involution.  Hence
		\[
		\ol{\Ffrak}_c^{-1} \left( (\afrak^{i+1})^{\perp} \right) = \ol{\Ffrak}_c \left( (\afrak^{i+1})^{\perp} \right) = F_i^\afrak\Lrm_c.\qedhere
		\]
	\end{proof}
	\begin{prop}
		\label{prop:hodge-alg}
		For every $i\in\Z$, we have $F_i^0\Lrm_c=F_i^\afrak\Lrm_c$.
	\end{prop}
	\begin{proof}
		We use Proposition \ref{prop: support gr Nc}. For $k\geq 1$, write  $P_k=\tr(X^k)\in\C[\gfrak]^G$.  {By Proposition \ref{prop: support gr Nc},} since $P_k$ vanishes on $\Cfrak_\nil$, we have $P_k\gr^0 N_{m/n}=0$, or equivalently $P_kF_j^0 N_{m/n}\subset F_{j-1}^0 N_{m/n}$.
		
		Now let $S=\langle P_2,\dots,P_n\rangle $ and define a $\C[\gfrak]$-linear map
		\[
		\lambda_j: \gr_j^0 N_c \rightarrow S^*\otimes\gr_{j-1}^0 N_c, \qquad \lambda_j( u)(P_k)=P_ku.
		\]
		
		On $\gfrak_r$, the relations $[P_k,\partial_{P_\ell}]=-\delta_{k\ell}$ give
		\[
		P_k\left( \partial_{P_2}^{a_2}\cdots \partial_{P_n}^{a_n}u \right) = -a_k \partial_{P_2}^{a_2}\cdots \partial_{P_k}^{a_k-1}\cdots \partial_{P_n}^{a_n}u.
		\]
		Thus $\lambda_j|_{\gfrak_r}$ is the collection of contractions on $\Sym^j\langle\partial_{P_2},\dots,\partial_{P_n}\rangle$. Their common kernel is zero for $j>0$, and therefore
		\begin{align}
			\label{formula:regular-common-kernel}
			\ker(\lambda_j)|_{\gfrak_r}=0 \qquad(j>0).
		\end{align}
		We now extend this injectivity beyond the regular locus. 
		{Use $R=\C[T^*\gfrak]$, $M=\gr^0N_c$ and the prime $\pfrak$ defining $\Cfrak_\nil$, as in the proof of Proposition \ref{prop: support gr Nc}.} Let $L_j$ be the $R$-submodule of $M$ generated by $\Ker(\lambda_j)$. Then $L_j$ is supported on $\pi^{-1}(\Nscr\setminus \Nscr_r)$ for $\pi:T^*\gfrak\to\gfrak$ being the projection.
		If $L_j\neq0$, then $\Ass_R(L_j)\subset\Ass_R(M)=\{\pfrak\}$, a contradiction {, since the generic point of $\Cfrak_\nil$ lies over $\Nscr_r$ and hence $(L_j)_\pfrak=0$}. Therefore $L_j=0$, and hence $\lambda_j$ is injective: 
		\begin{align}
			\label{formula:global-common-kernel}
			\bigcap_{k=2}^n \ker\left( P_k: \gr_j^0 N_c \rightarrow \gr_{j-1}^0 N_c \right) =0 \qquad(j>0).
		\end{align}
		
		Passing to $ {\gr^0\Lrm_{n/m}^{(m)}}$, let $p_k=\sum_{r=1}^m x_r^k\in\C[\hfrak_m]^{S_m}$ for $k\geq 1$. 
		Under the identification in Theorem \ref{thm:cee}, $(N_{m/n}\otimes (V^*)^{\otimes m})^G\cong \Lrm^{(m)}_{n/m}$, the action of $P_k$ corresponds, up to a nonzero scalar, to the action of $p_k$.
		Tensoring \eqref{formula:global-common-kernel} with $(V^*)^{\otimes m}$ preserves injectivity, and taking $G$-invariants is exact because $G$ is reductive. Hence $p_kF_j^0 \Lrm_{n/m}^{(m)} \subset F_{j-1}^0 \Lrm_{n/m}^{(m)}$. As a result, every element of $\afrak^{i+1}$ sends $F_i^0\Lrm_{n/m}^{(m)}$ into $F_{-1}^0\Lrm_{n/m}^{(m)}=0$. Moreover, \eqref{formula:global-common-kernel} plus Newton's identity implies
		\begin{align}
			\label{formula:spherical-common-kernel}
			\bigcap_{k=2}^m \ker\left( p_k: \gr_j^0\Lrm_{n/m}^{(m)} \rightarrow \gr_{j-1}^0\Lrm_{n/m}^{(m)} \right) =0
		\end{align}
		for all $j>0$. As a result, for $i\geq 0$ and  $\phi\in F_i^0\Lrm_{n/m}^{(m)}\setminus F_{i-1}^0\Lrm_{n/m}^{(m)}$, $\afrak^i\phi\neq 0$. Altogether, we have shown that $F_i^0\Lrm_{n/m}^{(m)}$ equals $\{\phi\in\Lrm_{n/m}^{(m)},\afrak^{i+1}\phi=0\}$, which equals $F_i^\afrak\Lrm_{n/m}^{(m)}$ by  Lemma \ref{lem:alg ann}. 
	\end{proof}
	
	\subsection{Shift functors}\label{subsec:shift-functors}
	Put $\e_-=\frac{1}{n!}\sum_{w\in {S_n}}\operatorname{sign}(w)w$. Define $B_{c,1}:=\Hrm_{c+1}\e_-\delta\Hrm_c.$
	For every $k\geq2$, define recursively $B_{c,k} := B_{c+k-1,1} \otimes_{\Hrm_{c+k-1}} B_{c,k-1}.$
	Thus $B_{c,k}$ is an $(\Hrm_{c+k},\Hrm_c)$-bimodule, naturally embeds inside the localization $(\Hrm_c)_\delta$ and hence inherits the order filtration. By \cite[Lemma 7.2]{gs1}, this sub-filtration coincides with the tensor-product filtration. 
	Let $Q_{c,k}:=\e B_{c,k} \e$, which is an $(\Arm_{c+k},\Arm_c)$- {bimodule}.
	
	\begin{lem}\label{lem:shifts}
		With respect to tensor product filtrations on the left hand sides, we have filtered isomorphisms for all $k\geq1$:
		\begin{align}
			B_{c,k}\otimes_{\Hrm_c}(\Lrm_c,F^{\afrak})
			&\xrightarrow{\sim}(\Lrm_{c+k},F^{\afrak}), \label{eq: shift Lc}\\
			Q_{c,k}\otimes_{\Arm_c}(\e\Lrm_c,F^{\afrak})
			&\xrightarrow{\sim}(\e\Lrm_{c+k},F^{\afrak}). \label{eq: shift eLc}
		\end{align}
	\end{lem}
	\begin{proof}
		By \cite[Theorem 4.2]{ma}, we have filtered isomorphisms: $B_{c,k}\e\otimes_{\Arm_c}(\e\Lrm_c,F^\afrak) \cong (\Lrm_{c+k},F^\afrak).$ This implies both filtered isomorphisms of the lemma. 
	\end{proof}
	
	\subsection{Equivariant Proj construction}\label{sec:proj}
	Consider the diagonal $S_n$ action on $\hfrak^2$. Set $R_0=\C[\hfrak^2]^{S_n}$ and  $R_1$ be the ideal $\delta\C[\hfrak^2]^{\sign}$. Take $R_k:=R_1^k$ to be the product of ideals in $R_0$ and 
	\begin{align*}
		R:=\bigoplus_{k\geq0}R_k.
	\end{align*}
	On $R_k$, we  {normalize} the $A$-grading such that $\delta^k$ contributes bidegree $(0,0)$.  
	In this way, by \cite[Proposition 2.6]{haimandiscrete}, there is an $A$-equivariant isomorphism
	\[
	\Hilb^n\cong\operatorname{Proj}(R), 
	\]
	such that $A$-equivariantly, $\det \Vcal\cong \Ocal(1)$.

	The algebra $R$ is noetherian and generated by $R_1$ over $R_0$.
	The equivariant form of the classical theorem of Serre gives an essentially surjective functor from finitely generated $A$-equivariant graded $R$-modules to $A$-equivariant coherent sheaves on $\Hilb^n$: For a coherent sheaf $\Fcal$ on $\Hilb^n$, set 
	\[
	\Gamma_*(\Fcal):= \bigoplus_{k\geq0}\Gamma(\Hilb^n,\Fcal\otimes\Ocal(k)).
	\]
	There is an $A$-equivariant isomorphism
	\begin{align*}
		\wt{\Gamma_*(\Fcal)}\xrightarrow{\sim}\Fcal.
	\end{align*}
	\subsection{Simple modules and sheaves} 
	\begin{coro}\label{cor:proj}
		Let $c=m/n>0$ with $(m,n)=1$. There is an    $A$-equivariant graded $R$-module isomorphism
		\begin{align}
			\label{iso: proj}
			\bigoplus_{k\geq0}\gr^{\afrak}(\e\Lrm_{c+k})\backsimeq \bigoplus_{k\geq0} \Gamma(\Hilb^n,\Fcal^0_{c+k}).
		\end{align}
	\end{coro}
	\begin{proof}
		First of all, the isomorphism $\Fcal_c(k)\cong\Fcal_{c+k}$, Theorem \ref{thm: vanishing} with $j=0$, and Proposition \ref{prop:hodge-alg} give an $A$-equivariant graded $R_0$-module isomorphism between two sides of \eqref{iso: proj}. We continue to compare the $R$-module structures. 
		
		Let $r=c+\ell$ for an integer $\ell\geq 0$. Write $\Dscr=\Dcal(\Gfrak)$. Then by  {\cite[Corollary 4.6, Theorem 5.3, Proposition 5.8]{ggs}} and using \eqref{eqn: cuspidal red}, we have the vertical isomorphisms in \eqref{diag: shift} and the commutativity of the diagram. The top isomorphism is \eqref{eq: shift eLc} from Lemma \ref{lem:shifts}. The lower arrow is given by multiplication.

		\begin{align}
			\label{diag: shift}
			\xymatrix{ \Omega_{r+k}(Q_{r,k}\otimes_{\Arm_r} \e\Lrm_r) \ar[r]^\sim\ar[d]^\sim&\Omega_{r+k}(\e\Lrm_{r+k})\ar[d]^\sim\\
				\left(\Dscr/\Dscr\tau_{-r}(\ol\gfrak)\right)^{\det^{-k}}\otimes_{\Arm_{-r-1}} \left(N_c\otimes\C[V]\right)^{\tau_{-r}(\ol\gfrak)}\ar[r]^ - \sim&\left(N_c\otimes\C[V]\right)^{\tau_{-r-k}(\ol\gfrak)}}
		\end{align}
		
		By Theorem \ref{thm: vanishing}, we have $\gr^0\left(N_c\otimes\C[V]\right)^{\tau_{-r}(\ol\gfrak)}\cong \Gamma(\Hilb^n,\Fcal_r).$ Moreover, by \cite[Corollary 6.14(1)]{gs1}, we have an $A$-equivariant isomorphism 
		\[\gr(Q_{r,k})\cong \gr\left(\Dscr/\Dscr\tau_{-r}(\ol\gfrak)\right)^{\det^{-k}}\backsimeq  R_k.\] 
		{Taking associated graded of \eqref{diag: shift} and precomposing with the surjection 
			\[\gr Q_{r,k}\otimes_{\gr\Arm_r}\gr^\afrak(\e\Lrm_r)\twoheadrightarrow\gr(Q_{r,k}\otimes_{\Arm_r}\e\Lrm_r)\] gives the following commutative diagram:}
		\begin{align*}
			\xymatrix{R_k\otimes_{R_0} \gr^\afrak(\e\Lrm_r) \ar[r]\ar[d]^\backsimeq&\gr^\afrak(\e\Lrm_{r+k})\ar[d]^\backsimeq\\
				R_k\otimes_{R_0} \Gamma(\Hilb^n,\Fcal^0_r)\ar[r]&\Gamma(\Hilb^n,\Fcal^0_{r+k})}
		\end{align*}
		Hence the isomorphism respects $R$-actions. 
	\end{proof}

	\subsection{Full modules and Procesi bundle}
	Write $\alpha:\IHilb^n\to\Hilb^n$ and $\beta:\IHilb^n\to\hfrak^2$ for the two projections.
	Recall that $\Pcal^{\sign}=\det\Vcal=\Ocal(1)$.
	By \cite[Proposition 5.5]{haiman_vanishing}, the inverse BKR functor is
	\[
	\Psi(M)=\bigl(R\alpha_*L\beta^*M\bigr)^{\sign}\otimes\Ocal(-1).
	\]
	\begin{thm}\label{thm:Lc F P}
		{Let $c=m/n>0$ with $(m,n)=1$.} Suppose $H^i(\Hilb^n, {\Fcal^0_{c}}\otimes \Pcal^\vee)=0$ for $i>0$ (which holds e.g, when $c\gg 0$). 
		There is a bigraded $\C[\hfrak^2]\rtimes S_n$-module isomorphism
		\[
		\gr^{\afrak}\Lrm_{c} \backsimeq\Gamma(\Hilb^n,\Fcal_c^0\otimes \Pcal^\vee).
		\]
	\end{thm}
	\begin{proof}
		Recall the $S_n\times A$-equivariant isomorphism
		\[
		\Pcal^\vee\cong\sign\otimes\Pcal\otimes\Ocal(-1).
		\]
		Together with $\Fcal_c^0\otimes\Ocal(-1)\cong\Fcal_{c-1}^0$, this gives
		\[
		\Fcal_c^0\otimes\Pcal^\vee \cong\sign\otimes\Fcal_{c-1}^0\otimes\Pcal.
		\]
		We will construct a bigraded $\C[\hfrak^2]\rtimes S_n$-module homomorphism $\varrho:\sign\otimes\gr^{\afrak}\Lrm_c \rightarrow\Gamma(\Hilb^n,\Pcal\otimes\Fcal_{c-1}^0)$ and show its injectivity. The theorem will then follow from comparison of dimensions.
		
		The truncation map $L\beta^*(\sign\otimes\gr^{\afrak}\Lrm_c) \rightarrow\beta^*(\sign\otimes\gr^{\afrak}\Lrm_c)$ and the affineness of $\alpha$ give a map
		\begin{align}
			\label{eqn: Psi truncation}
			\Psi(\sign\otimes\gr^{\afrak}\Lrm_c) \rightarrow \left(\alpha_*\beta^*(\sign\otimes\gr^{\afrak}\Lrm_c)\right)^\sign \otimes\Ocal(-1).
		\end{align}
		
		Write $J_0=\C[\hfrak^2]$ and
		\[
		J_1=J_0\delta\C[\hfrak^2]^\sign, \qquad J_k=J_1^k\subset J_0, \qquad J=\bigoplus_{k\geq0}J_k.
		\]
		Give $J$ the analogous $A$-equivariance as $R$.
		By \cite[Proposition 3.4.2]{haiman_macdonald}, $\IHilb^n=\Proj(J)$, $A$-equivariantly. As in Section \ref{sec:proj}, the functor ${(-)}^\sim$ takes finitely generated graded $J$-modules to coherent sheaves on $\IHilb^n$. Taking $S_n$-invariants after pushing forward along $\alpha$, the right hand side of \eqref{eqn: Psi truncation} is identified with
		\begin{align}
			\label{formula:psi-gr-lc}
			\left( \bigoplus_{k\geq0} \left(J_k\otimes_{J_0}\gr^{\afrak}\Lrm_c\right)^{S_n} \right)^\sim \otimes\Ocal(-1),
		\end{align}
		where the sheafification in \eqref{formula:psi-gr-lc} is taken over $R$.
		
		Taking the associated graded of \eqref{eq: shift Lc} and precomposing with the surjection from the tensor product of associated graded gives a surjection
		\begin{align}
			\label{map:gr B}
			\gr B_{c,k}\otimes_{\gr(\Hrm_c)}\gr^\afrak(\Lrm_c) \rightarrow \gr^\afrak(\Lrm_{c+k})
		\end{align}
		for every $k\geq1$.
		The identification $\e\gr B_{c,k}\cong J_k$ from \cite[Proposition 6.5]{gs1} gives
		\[
		\left( \gr B_{c,k}\otimes_{\gr(\Hrm_c)}\gr^\afrak(\Lrm_c) \right)^{S_n} \cong \left(J_k\otimes_{J_0}\gr^\afrak(\Lrm_c)\right)^{S_n}.
		\]
		The maps \eqref{eqn: Psi truncation} and \eqref{map:gr B} give 
		\begin{align}
			\label{iso: varrho Fc}
			\Psi\!\left(\sign\otimes\gr^\afrak(\Lrm_c)\right) \rightarrow \left( \bigoplus_{k\geq0}\gr^\afrak(\e\Lrm_{c+k}) \right)^\sim \otimes\Ocal(-1) \backsimeq\Fcal_c^0\otimes\Ocal(-1) \cong\Fcal_{c-1}^0,
		\end{align}
		using Corollary \ref{cor:proj}.
		Applying $\Phi$, we obtain a bigraded $J_0\rtimes S_n$-module map
		\[
		\varrho: \sign\otimes\gr^\afrak(\Lrm_c) \rightarrow R\Gamma\!\left(\Hilb^n,\Pcal\otimes\Fcal_{c-1}^0\right).
		\]
		By the hypothesis, the right hand side is concentrated in cohomological degree $0$.
		
		Next we show that the $\sign$ part of $\varrho$ is an isomorphism.
		Since $\Pcal^\sign\cong\Ocal(1)$ and $\Fcal_{c-1}^0\otimes\Ocal(1)\cong\Fcal_c^0$, we obtain
		\[
		\varrho^\sign: \gr^\afrak(\e\Lrm_c) \rightarrow\Gamma\!\left(\Hilb^n,\Fcal_c^0\right).
		\]
		
		For a homogeneous $v\in\gr^{\afrak}(\e\Lrm_c)$, multiplication by $v$ defines an $S_n$-equivariant $J_0$-module map $\sign\otimes J_0\rightarrow\sign\otimes\gr^{\afrak}\Lrm_c$ of the same bidegree as $v$.
		Applying $\Psi$ and composing with \eqref{iso: varrho Fc} gives a map $\Ocal(-1)\rightarrow\Fcal_{c-1}^0$.
		By naturality of truncation and sheafification, after tensoring with $\Ocal(1)$ this is the sheafification of the graded $R$-module map
		\[
		R\rightarrow \bigoplus_{k\geq0}\gr^{\afrak}(\e\Lrm_{c+k}).
		\]
		Under Corollary \ref{cor:proj}, this sends $1$ to the section corresponding to $v$. Hence $\varrho^\sign$ is the degree-$0$ component of the isomorphism in Corollary \ref{cor:proj}, and is therefore an isomorphism.
		
		By Proposition \ref{prop:socle} below, the socle of $\sign\otimes\gr^\afrak(\Lrm_c)$ as a $J_0$-module is $\sign\otimes\gr^\afrak(\e\Lrm_c)$.
		Since this module is finite-dimensional and supported at the origin, every nonzero submodule intersects its socle nontrivially.
		As $\varrho$ is injective on the socle, $\Ker(\varrho)=0$.
		
		Finally, Theorem \ref{thm:ch Lc Q} and Proposition \ref{prop:Fc and P} give equality of the dimensions.
	\end{proof}
	
	\subsection{Socle of $\gr^\afrak\Lrm_c$}
	Let $\mfrak\subset \C[\hfrak^2]$ be the maximal ideal at the origin. Among modules supported at the origin, the only simple $\C[\hfrak^2]$-module is $\C[\hfrak^2]/\mfrak$, so
	\begin{align*}
		\socle(\gr^\afrak\Lrm_c)&=\{\phi\in\gr^\afrak\Lrm_c:\mfrak\phi=0\}.
	\end{align*}
	\begin{prop}
		\label{prop:socle}
		$\socle\left(\gr^\afrak\Lrm_c\right)=\gr^\afrak(\e\Lrm_c).$
	\end{prop}
	\begin{proof}
		Suppose $\xi_i$ is the image of $y_i$ in $\gr(\Hrm_c)$.
		We first compute $\bigcap_i\Ker(\xi_i)$.
		Below all powers and operators are taken in $\Lrm_c$.
		Use the Frobenius pairing from Lemma \ref{lem:alg ann}, $\langle u,v\rangle_{\mathrm{Fr}}=(u, {\ol\Ffrak_c}(v))_c$.
		The adjointness relations for the Dunkl form give $\langle yu,v\rangle_{\mathrm{Fr}} =-\langle u,yv\rangle_{\mathrm{Fr}}$.
		By Lemma \ref{lem:alg ann}, this pairing induces an $S_n$-emodule isomorphism     $\gr_k^\afrak\Lrm_c\cong(\afrak^k/\afrak^{k+1})^*.$ 
		The Dunkl product rule gives $\hfrak\cdot\afrak^{k+1}\subset\afrak^k$.
		Thus, under this identification,
		\[
		\bigcap_i\Ker(\xi_i)\cap\gr_k^\afrak\Lrm_c\cong Q_k^*,\qquad \text{ where } Q_k:=\afrak^k/(\afrak^{k+1}+\hfrak\cdot\afrak^{k+1}),
		\]
		where $\hfrak\cdot\afrak^{k+1}=\sum_i y_i\afrak^{k+1}$.
		We claim that $Q_k=Q_k^{S_n}.$
		
		Let $p_i$ be the $i$-th power sum symmetric polynomial for $i=2,\dots,n$. Every element of $\afrak^k$ is a sum of terms $hp_{a_1}\cdots p_{a_k}$ with $h\in \C[\hfrak]$.
		We prove by downward induction on the exponent $a$ of $p_2$ that for $h$ nonconstant homogeneous and $r_s>2$ $hp_2^ap_{r_1}\cdots p_{r_{k-a}}=0$ in $Q_k$. It suffices to take $h=xh'$ with $x$ linear. Choose $y\in\hfrak$ such that $x=\partial_y p_2$, and put $g=p_{r_1}\cdots p_{r_{k-a}}$.
		Since $p_2$ and $g$ are invariant, product rule \cite[Corollary 5.3]{gorskyarc} gives
		\[
		y(h'p_2^{a+1}g) =(yh')p_2^{a+1}g +(a+1)hp_2^ag +h'p_2^{a+1}\partial_y g.
		\]
		The left hand side lies in $\hfrak\cdot\afrak^{k+1}$, and the first term on the right lies in $\afrak^{k+1}$.
		The last term is
		\[
		\sum_s(h'\partial_y p_{r_s})p_2^{a+1} \prod_{t\ne s}p_{r_t}.
		\]
		For $a=k$ this sum is empty. Otherwise each summand has $k$ invariant factors, $a+1$ of them equal to $p_2$, and a coefficient of positive degree. Hence this term vanishes in $Q_k$ by induction. Since $a+1\ne0$, this proves the vanishing of $hp_2^ag$ as desired. Consequently, $Q_k$ is spanned by products of $k$ invariant generators with constant coefficients, proving the claim. As a result, $\bigcap_i\Ker(\xi_i)\subset \gr^\afrak\e\Lrm_c$. 
		
		For the other direction, we need to show $\e(\hfrak\cdot\afrak^{k+1})\subset \e \afrak^{k+1}$.
		For $g=p_{a_1}\cdots p_{a_{k+1}}$, by product rule, 
		\[
		\e\bigl(y(hg)\bigr) =\e(yh)g +\sum_s\e(h\partial_y p_{a_s}) \prod_{t\ne s}p_{a_t}.
		\]
		Each coefficient $\e(h\partial_y p_{a_s})$ is an invariant polynomial with zero constant term, so every term belongs to $\afrak^{k+1}$. Therefore $Q_k=\e(\afrak^k/\afrak^{k+1})$. Dualizing gives $ \bigcap_i\Ker(\xi_i)=\gr^\afrak(\e\Lrm_c).$ 
		
		Let $u_i$ be the image of $x_i\in\hfrak^*$ in $\gr(\Hrm_c)$.
		Since the $u_i$ commute with the $\xi_i$, we have $u_i \bigcap_i\Ker(\xi_i)\subset \bigcap_i\Ker(\xi_i)$.
		For $v\in \bigcap_i\Ker(\xi_i)$, triviality of the $S_n$-action on $\bigcap_i\Ker(\xi_i)$ makes the map $\hfrak^*\to \bigcap_i\Ker(\xi_i)$, $u_i\mapsto u_iv$, $S_n$-equivariant.
		It is zero because $\Hom_{S_n}(\hfrak^*,\C)=0$.
		Thus $\bigcap_i\Ker(\xi_i)=\socle(\gr^\afrak\Lrm_c)$. 
	\end{proof}
	
	\subsection{Compactified Jacobians}\label{subsec:compactified-jacobians}
	We apply Theorem \ref{thm:ch Lc Q} and Proposition \ref{prop:hodge-alg} to a conjecture in \cite{oyjacobian}. Keep $m,n$ as above and $c=m/n$. 
	Let $J_{m,n}$ be the compactified Jacobian with a fixed degree of the compactification of the plane curve $y^m=x^n$ that is smooth away from $0$.
	Write $P$ for the perverse filtration on $H^*(J_{m,n})$. Here cohomology is taken with $\C$-coefficients.
	
	To state the connection with the Cherednik module, recall the isomorphism 
	\begin{align}
		\label{iso:eL quotient}
		\e\Lrm_c\cong\C[\hfrak]^{S_n}/I_c^{S_n},
	\end{align}
	which is the coordinate ring considered in \cite[Section 1.1.6]{oyjacobian}. Take the maximal ideal $\mfrak:=\C[\hfrak]^{S_n}_+/I_c^{S_n}$ and define a decreasing filtration on $\e\Lrm_c$ by $\mfrak^r$.  
	Let $\e\Lrm_c[j]$ denote the homogeneous degree $j$ subspace under the identification \eqref{iso:eL quotient}.
	
	\begin{conj}(\cite[Conjecture 1.1.7]{oyjacobian})\label{conj:oy}
		For every $i,j\in\Z$, there is a canonical isomorphism $\gr_j^P H^{2i}(J_{m,n})\cong\gr_{\mfrak}^{j-i}\e\Lrm_c[j]$ compatible with the ring structures. 
	\end{conj}
	By \cite[Theorem 1.1.8 and Corollary 1.1.9]{oyjacobian}, Conjecture \ref{conj:oy} is equivalent to the dimension equalities
	\begin{align}
		\label{eqn: oy graded}
		\dim H^{2(g-i)}(J_{m,n}) = \dim\gr_{\mfrak}^{i}\e\Lrm_c
	\end{align}
	for all $i\in \Z$ where $g=(m-1)(n-1)/2$. 
	
	The bridge between the two sides of \eqref{eqn: oy graded} is the rational shuffle theorem.  Let $\DP_{m/n}$ be the set of $(m,n)$-Dyck paths, and let $\mathrm{area}$ and $\dinv$ be the two standard combinatorial statistics on $\DP_{m/n}$, with values in $[0,d]$; we refer to \cite{mellitshuffle} for the definitions. 
	
	\begin{thm}[{\cite[Theorem 5.8]{mellitshuffle}}]
		\label{thm:shuffle}
		For coprime positive integers $m,n$,
		\begin{align*}
			\bigl(Q_{m,n}\cdot 1,\,s_{(1^n)}\bigr) =\sum_{\Delta\in \DP_{m/n}}q^{\mathrm{area}(\Delta)}t^{\dinv(\Delta)}.
		\end{align*}
	\end{thm}
	
	\begin{thm}\label{cor:oy}
		Conjecture \ref{conj:oy} holds.
	\end{thm}
	\begin{proof}
		We verify the dimension criterion \eqref{eqn: oy graded}. Theorem \ref{thm:ch Lc Q}, Proposition \ref{prop:hodge-alg} and Theorem \ref{thm:shuffle} give 
		\begin{align*}
			\sum_{j\geq0}\dim\gr_j^\afrak(\e\Lrm_c)t^j =\left.(Q_{m,n}1,s_{(1^n)})\right|_{q=1} =\sum_{\Delta\in\DP_{m/n}}t^{\dinv(\Delta)}.
		\end{align*}
		Let $ {P(t)=\sum_r\dim H^r(J_{m,n})t^r}$ be the Poincar\'{e} series of $J_{m,n}$. Take the $S_n$-invariant part in \cite[Theorems 4.15 and 4.17]{hikita_asf}, corresponding to the Grassmannian affine Springer fiber and hence to the compactified Jacobian of a fixed degree. We obtain 
		\[t^{2d}P(t^{-1})=\sum_{\Delta\in\DP_{m/n}}t^{2\dinv(\Delta)}.\]
		Hence $\dim H^{2(g-j)}(J_{m,n})=\dim\gr_j^\afrak(\e\Lrm_c).$
		Therefore \eqref{eqn: oy graded} holds. 
	\end{proof}
	

\begin{thebibliography}{99}
		\bibitem[Ach]{achar}
		P.~Achar, {\it Equivariant mixed Hodge modules}, Lecture notes.
		
		\bibitem[Bal23]{balibanu}
		A.~Balibanu, {\it The partial compactification of the universal centralizer}, Selecta Math. (N.S.) {\bf 29} (2023), no.~5, Paper No. 85, 36.
		
		\bibitem[Bar01]{Baranovsky}
		V.~Baranovsky, {\it The variety of pairs of commuting nilpotent matrices is irreducible}, Transform. Groups {\bf 6} (2001), no.~1, 3--8.
		
		\bibitem[BE09]{be_parabolic_ind_res}
		R.~Bezrukavnikov and P.~Etingof, {\it Parabolic induction and restriction functors for rational Cherednik algebras}, Selecta Math. (N.S.) {\bf 14} (2009), no.~3-4, 397--425.
		
		\bibitem[BEG03]{beg}
		Y.~Berest, P.~Etingof, and V.~Ginzburg, {\it Finite-dimensional representations of rational Cherednik algebras}, Int. Math. Res. Not. (2003), no.~19, 1053--1088.
		
		\bibitem[BFG06]{Bezru-Finkelberg-Ginzburg}
		R.~Bezrukavnikov, M.~Finkelberg, and V.~Ginzburg, {\it Cherednik algebras and Hilbert schemes in characteristic $p$}, Represent. Theory {\bf 10} (2006), 254--298.
		
		\bibitem[BG98]{science_fiction}
		F.~Bergeron and G.~Garsia, {\it Science Fiction and Macdonald's polynomials}, arXiv preprint math/9809128 (1998).
		
		\bibitem[BG13]{bgcombo}
		G.~Bellamy and V.~Ginzburg, {\it Some combinatorial identities related to commuting varieties and Hilbert schemes}, Math. Ann. {\bf 355} (2013), no.~3, 801--847.
		
		\bibitem[BG15]{bgnearbycycle}
		G.~Bellamy and V.~Ginzburg, {\it Hamiltonian reduction and nearby cycles for mirabolic $\mathscr{D}$-modules}, Adv. Math. {\bf 269} (2015), 71--161.
		
		\bibitem[BHM+23]{BHMPS}
		J.~Blasiak, M.~Haiman, J.~Morse, A.~Pun, and G.~H.~Seelinger, {\it A shuffle theorem for paths under any line}, Forum Math. Pi {\bf 11} (2023), Paper No. e5, 38.
		
		\bibitem[BKR01]{BKR}
		T.~Bridgeland, A.~King, and M.~Reid, {\it The McKay correspondence as an equivalence of derived categories}, J. Amer. Math. Soc. {\bf 14} (2001), no.~3, 535--554.
		
		\bibitem[CEE09]{cee}
		D.~Calaque, B.~Enriquez, and P.~Etingof, {\it Universal KZB equations: the elliptic case}, Algebra, arithmetic, and geometry: in honor of Yu. I. Manin. Vol. I, Progr. Math., vol. 269, Birkh\"{a}user Boston, Boston, MA, 2009, pp. 165--266.
		
		\bibitem[Del70]{deligneequation}
		P.~Deligne, {\it \'{E}quations diff\'{e}rentielles \`a points singuliers r\'{e}guliers}, Lecture Notes in Mathematics, vol. 163, Springer-Verlag, Berlin-New York, 1970.
		
		\bibitem[DO03]{dunklopdam}
		C.~F.~Dunkl and E.~M.~Opdam, {\it Dunkl operators for complex reflection groups}, Proc. London Math. Soc. (3) {\bf 86} (2003), no.~1, 70--108.
		
		\bibitem[EG02]{etingof-ginzburg}
		P.~Etingof and V.~Ginzburg, {\it Symplectic reflection algebras, Calogero-Moser space, and deformed Harish-Chandra homomorphism}, Invent. Math. {\bf 147} (2002), no.~2, 243--348.
		
		\bibitem[EH19]{Elias-Hogancamp}
		B.~Elias and M.~Hogancamp, {\it On the computation of torus link homology}, Compos. Math. {\bf 155} (2019), no.~1, 164--205.
		
		\bibitem[ESm88]{ellingsrud-stromme}
		G.~Ellingsrud and S.~A.~Str\o mme, {\it On a cell decomposition of the Hilbert scheme of points in the plane}, Invent. Math. {\bf 91} (1988), no.~2, 365--370.
		
		\bibitem[FG10]{finkelberg-ginzburg_mirabolic}
		M.~Finkelberg and V.~Ginzburg, {\it On mirabolic $\mathcal{D}$-modules}, Int. Math. Res. Not. IMRN (2010), no.~15, 2947--2986.
		
		\bibitem[FT11]{feigin-tsy}
		B.~L.~Feigin and A.~I.~Tsymbaliuk, {\it Equivariant $K$-theory of Hilbert schemes via shuffle algebra}, Kyoto J. Math. {\bf 51} (2011), no.~4, 831--854.
		
		\bibitem[GG06]{ganginzburg}
		W.~L.~Gan and V.~Ginzburg, {\it Almost-commuting variety, $\mathscr{D}$-modules, and Cherednik algebras}, IMRP Int. Math. Res. Pap. (2006), 26439, 1--54.
		
		\bibitem[GGOR03]{ggor}
		V.~Ginzburg, N.~Guay, E.~Opdam, and R.~Rouquier, {\it On the category $\mathscr{O}$ for rational Cherednik algebras}, Invent. Math. {\bf 154} (2003), no.~3, 617--651.
		
		\bibitem[GGS09]{ggs}
		V.~Ginzburg, I.~Gordon, and J.~T.~Stafford, {\it Differential operators and Cherednik algebras}, Selecta Math. (N.S.) {\bf 14} (2009), no.~3-4, 629--666.
		
		\bibitem[GH98]{garsia-haiman-pieri_rule}
		A.~M.~Garsia and M.~Haiman, {\it A random $q, t$-hook walk and a sum of Pieri coefficients}, J. Combin. Theory Ser. A {\bf 82} (1998), no.~1, 74--111.
		
		\bibitem[Gin00]{ginzburg_principal_nilpotent}
		V.~Ginzburg, {\it Principal nilpotent pairs in a semisimple Lie algebra. I}, Invent. Math. {\bf 140} (2000), no.~3, 511--561.
		
		\bibitem[Gin12]{isospec}
		V.~Ginzburg, {\it Isospectral commuting variety, the Harish-Chandra $\mathcal{D}$-module, and principal nilpotent pairs}, Duke Math. J. {\bf 161} (2012), no.~11, 2023--2111.
		
		\bibitem[GLWX17]{glwx}
		A.~M.~Garsia, E.~S.~Leven, N.~Wallach, and G.~Xin, {\it A new plethystic symmetric function operator and the rational compositional shuffle conjecture at $t=1/q$}, J. Combin. Theory Ser. A {\bf 145} (2017), 57--100.
		
		\bibitem[GN15]{gn}
		E.~Gorsky and A.~Negu\c{t}, {\it Refined knot invariants and Hilbert schemes}, J. Math. Pures Appl. (9) {\bf 104} (2015), no.~3, 403--435.
		
		\bibitem[GN25]{gn_commuting_idempotent}
		E.~Gorsky and A.~Negu\c{t}, {\it Hecke categories, idempotents, and commuting stacks}, Math. Z. {\bf 309} (2025), Paper No. 44, 30 pp.
		
		\bibitem[Gor03]{Gordon_diagonal_harmonics}
		I.~Gordon, {\it On the quotient ring by diagonal invariants}, Invent. Math. {\bf 153} (2003), no.~3, 503--518.
		
		\bibitem[Gor13]{gorskyarc}
		E.~Gorsky, {\it Arc spaces and DAHA representations}, Selecta Math. (N.S.) {\bf 19} (2013), no. 1, 125--140.
		
		\bibitem[GORS14]{gors}
		E.~Gorsky, A.~Oblomkov, J.~Rasmussen, and V.~Shende, {\it Torus knots and the rational DAHA}, Duke Math. J. {\bf 163} (2014), no.~14, 2709--2794.
		
		\bibitem[GS05]{gs1}
		I.~Gordon and J.~T.~Stafford, {\it Rational Cherednik algebras and Hilbert schemes}, Adv. Math. {\bf 198} (2005), no.~1, 222--274.
		
		\bibitem[Hai98]{haimandiscrete}
		M.~Haiman, {\it $t,q$-Catalan numbers and the Hilbert scheme}, Discrete Math. {\bf 193} (1998), no.~1-3, 201--224.
		
		\bibitem[Hai01]{haiman_macdonald}
		M.~Haiman, {\it Hilbert schemes, polygraphs and the Macdonald positivity conjecture}, J. Amer. Math. Soc. {\bf 14} (2001), no.~4, 941--1006.
		
		\bibitem[Hai02]{haiman_vanishing}
		M.~Haiman, {\it Vanishing theorems and character formulas for the Hilbert scheme of points in the plane}, Invent. Math. {\bf 149} (2002), no.~2, 371--407.
		
		\bibitem[Hai03]{haiman_survey}
		M.~Haiman, {\it Combinatorics, symmetric functions, and Hilbert schemes}, Current developments in mathematics, 2002, Int. Press, Somerville, MA, 2003, pp. 39--111.
		
		\bibitem[Hik14]{hikita_asf}
		T.~Hikita, {\it Affine Springer fibers of type A and combinatorics of diagonal coinvariants}, Adv. Math. {\bf 263} (2014), 88--122.
		
		\bibitem[HL]{HalpernLeistner_DerivedTheta}
		D.~Halpern-Leistner, {\it Derived $\Theta$-stratifications and the $D$-equivalence conjecture}, Annals of Mathematics, To appear.
		
		\bibitem[HL15]{DanielHL-GIT}
		D.~Halpern-Leistner, {\it The derived category of a GIT quotient}, J. Amer. Math. Soc. {\bf 28} (2015), no.~3, 871--912.
		
		\bibitem[HTT08]{htt}
		R.~Hotta, K.~Takeuchi, and T.~Tanisaki, {\it $D$-modules, perverse sheaves, and representation theory}, Progress in Mathematics, vol. 236, Birkh\"{a}user Boston, Inc., Boston, MA, 2008.
		
		\bibitem[KR08]{khr2}
		M.~Khovanov and L.~Rozansky, {\it Matrix factorizations and link homology. II}, Geom. Topol. {\bf 12} (2008), no.~3, 1387--1425.
		
		\bibitem[Lau83]{laumon}
		G.~Laumon, {\it Sur la cat\'{e}gorie d\'{e}riv\'{e}es des $D$-modules filtr\'{e}s}, Algebraic geometry, Springer, 1983, pp. 151--237.
		
		\bibitem[Lus86]{lusztig4}
		G.~Lusztig, {\it Character sheaves. IV}, Adv. in Math. {\bf 59} (1986), no.~1, 1--63.
		
		\bibitem[Lus87]{Lusztig_fourier_cuspidal}
		G.~Lusztig, {\it Fourier transforms on a semisimple Lie algebra over $\mathbf{F}_q$}, Algebraic groups Utrecht 1986, Lecture Notes in Math., vol. 1271, Springer, Berlin, 1987, pp. 177--188.
		
		\bibitem[Ma26]{ma}
		X.~Ma, {\it Rational Cherednik algebras and knot invariants: Filtrations}, Int. Math. Res. Not. {\bf 2026} (2026), no.~10, rnag089.
		
		\bibitem[Mel21]{mellitshuffle}
		A.~Mellit, {\it Toric braids and $(m,n)$-parking functions}, Duke Math. J. {\bf 170} (2021), no.~18, 4123--4169.
		
		\bibitem[Mel22]{mellitknothomology}
		A.~Mellit, {\it Homology of torus knots}, Geom. Topol. {\bf 26} (2022), no.~1, 47--70.
		
		\bibitem[Nak99]{nakajima}
		H.~Nakajima, {\it Lectures on Hilbert schemes of points on surfaces}, University Lecture Series, vol. 18, American Mathematical Society, Providence, RI, 1999.
		
		\bibitem[Neg14]{negutshuffle}
		A.~Negu\c{t}, {\it The shuffle algebra revisited}, Int. Math. Res. Not. IMRN (2014), no.~22, 6242--6275.
		
		\bibitem[Neg15a]{negutflag}
		A.~Negu\c{t}, {\it Moduli of flags of sheaves and their $K$-theory}, Algebr. Geom. {\bf 2} (2015), no.~1, 19--43.
		
		\bibitem[Neg15b]{negutthesis}
		A.~Negu\c{t}, {\it Quantum Algebras and Cyclic Quiver Varieties}, Thesis (Ph.D.)--Columbia University, ProQuest LLC, Ann Arbor, MI, 2015.
		
		\bibitem[Neg22]{negutwalgebra}
		A.~Negu\c{t}, {\it $W$-algebras associated to surfaces}, Proc. Lond. Math. Soc. (3) {\bf 124} (2022), no. 5, 601--679.
		
		\bibitem[OR18]{oblomkov-rozansky}
		A.~Oblomkov and L.~Rozansky, {\it Knot homology and sheaves on the Hilbert scheme of points on the plane}, Selecta Math. (N.S.) {\bf 24} (2018), no.~3, 2351--2454.
		
		\bibitem[OY17]{oyjacobian}
		A.~Oblomkov and Z.~Yun, {\it The cohomology ring of certain compactified Jacobians}, arXiv preprint arXiv:1710.05391 (2017).
		
		\bibitem[Pop18]{popa}
		M.~Popa, {\it $\mathcal{D}$-modules in birational geometry}, Proceedings of the International Congress of Mathematicians---Rio de Janeiro 2018. Vol. II. Invited lectures, World Sci. Publ., Hackensack, NJ, 2018, pp. 781--806.
		\bibitem[PS13]{popa-schnell}
		M.~Popa and C.~Schnell, {\it Generic vanishing theory via mixed Hodge modules}, Forum Math. Sigma {\bf 1} (2013), Paper No. e1, 60 pp.
		\bibitem[Sai88]{saitofrench}
		M.~Saito, {\it Modules de Hodge polarisables}, Publ. Res. Inst. Math. Sci. {\bf 24} (1988), no.~6, 849--995.
		
		\bibitem[Sai90]{saito}
		M.~Saito, {\it Mixed Hodge modules}, Publ. Res. Inst. Math. Sci. {\bf 26} (1990), no.~2, 221--333.
		
		\bibitem[Sch]{schnell}
		C.~Schnell, {\it An overview of Morihiko Saito's theory of mixed Hodge modules}, Representation theory, automorphic forms \& complex geometry, Int. Press, Somerville, MA, pp. 27--80.
		
		\bibitem[SS24]{Sabbah-Schnell}
		C.~Sabbah and C.~Schnell, {\it Mixed Hodge module project},  \url{http://www.math.polytechnique.fr/perso/sabbah/MHMProject/mhm.html} (2024).
		
		\bibitem[SV13]{SV_Hilbert}
		O.~Schiffmann and E.~Vasserot, {\it The elliptic Hall algebra and the $K$-theory of the Hilbert scheme of $\mathbb{A}^2$}, Duke Math. J. {\bf 162} (2013), no.~2, 279--366.
		\bibitem[VV22]{Varagnolo-Vasserot}
		M.~Varagnolo and E.~Vasserot, {\it K-theoretic Hall algebras, quantum groups and super quantum groups}, Selecta Math. (N.S.) {\bf 28} (2022), no. 1, Paper No. 7, 56 pp.
		
	\end{thebibliography}
\end{document}